\documentclass{article}
\usepackage[latin1]{inputenc}
\usepackage{amsmath}
\usepackage{amsthm,amssymb,amsfonts}
\usepackage{bbold}
\usepackage{graphicx}
\usepackage{verbatim}
\usepackage{color,cite}
\usepackage{latexsym}
\usepackage{lscape}
\usepackage[T1]{fontenc}
\usepackage[all,cmtip]{xy}
\usepackage{hyperref, cleveref}
\usepackage{caption}
\usepackage{xspace}
\usepackage{enumitem}
\usepackage{tikz}
\usepackage{relsize}
\usepackage{xcolor}
\usetikzlibrary{patterns, cd}
\usetikzlibrary{shapes,arrows,spy,positioning}
\usetikzlibrary{tqft}
\usepackage{pinlabel}
\usepackage{soul}
\usepackage{fullpage}
\usepackage{bbm}

\theoremstyle{plain}
\newtheorem{maintheorem}{Theorem}

\newtheorem{maincor}[maintheorem]{Corollary}

\newtheorem{theorem}{Theorem}[section]
\newtheorem{lemma}[theorem]{Lemma}
\newtheorem{proposition}[theorem]{Proposition}
\newtheorem{corollary}[theorem]{Corollary}

\theoremstyle{definition}
\newtheorem{definition}[theorem]{Definition}

\newtheorem*{example}{Example}

\newtheorem*{remark}{Remark}
\newtheorem*{convention}{Convention}

\newcommand{\nc}{\newcommand}

\nc\bB{\mathbb{B}}
\nc\bC{\mathbb{C}}
\nc\bD{\mathbb{D}}
\nc\bE{\mathbb{E}}
\nc\bF{\mathbb{F}}
\nc\bG{\mathbb{G}}
\nc\bH{\mathbb{H}}
\nc\bI{\mathbb{I}}
\nc{\bJ}{\mathbb{J}}
\nc\bK{\mathbb{K}}
\nc\bL{\mathbb{L}}
\nc\bM{\mathbb{M}}
\nc\bN{\mathbb{N}}
\nc\bO{\mathbb{O}}
\nc\bP{\mathbb{P}}
\nc\bQ{\mathbb{Q}}
\nc\bR{\mathbb{R}}
\nc\bS{\mathbb{S}}
\nc\bT{\mathbb{T}}
\nc\bU{\mathbb{U}}
\nc\bV{\mathbb{V}}
\nc\bW{\mathbb{W}}
\nc\bY{\mathbb{Y}}
\nc\bX{\mathbb{X}}
\nc\bZ{\mathbb{Z}}

\nc\cA{\mathcal{A}}
\nc\cB{\mathcal{B}}
\nc\cC{\mathcal{C}}
\nc\cD{\mathcal{D}}
\nc\cE{\mathcal{E}}
\nc\cF{\mathcal{F}}
\nc\cG{\mathcal{G}}
\nc\cH{\mathcal{H}}
\nc\cI{\mathcal{I}}
\nc{\cJ}{\mathcal{J}}
\nc{\cK}{\mathcal{K}}
\nc\cM{\mathcal{M}}
\nc\cN{\mathcal{N}}
\nc\cO{\mathcal{O}}
\nc\cP{\mathcal{P}}
\nc\cQ{\mathcal{Q}}
\nc\cS{\mathcal{S}}
\nc\cT{\mathcal{T}}
\nc\cU{\mathcal{U}}
\nc\cV{\mathcal{V}}
\nc\cW{\mathcal{W}}
\nc\cY{\mathcal{Y}}
\nc\cX{\mathcal{X}}
\nc\cZ{\mathcal{Z}}
\nc\brho{\boldsymbol\rho}
\nc\ba{\boldsymbol a}
\nc\bb{\boldsymbol b}
\nc\be{\boldsymbol e}
\nc\bh{\boldsymbol h}
\nc\bi{\boldsymbol i}
\nc\bm{\boldsymbol m}
\nc\bn{\boldsymbol n}
\nc\bu{\boldsymbol u}
\nc\bv{\boldsymbol v}
\nc\bw{\boldsymbol w}
\nc\balpha{\boldsymbol{\alpha}}
\nc\bbeta{\boldsymbol{\beta}}
\nc\bxi{\boldsymbol{\xi}}

\newcommand\numberthis{\addtocounter{equation}{1}\tag{\theequation}}

\usepackage{rotating}

\nc{\dmo}{\DeclareMathOperator}
\dmo{\Div}{Div}
\dmo\PAut{PAut}
\dmo{\PRelAut}{PRelAut}
\dmo{\Hom}{Hom}
\dmo\Sp{Sp}
\dmo{\Mat}{Mat}
\dmo{\Mod}{Mod}
\dmo{\PMod}{PMod}
\dmo{\SMod}{SMod}
\dmo{\PSMod}{PSMod}
\nc{\nick}[1]{{\color{red} \sf  N: [#1]}}
\nc{\pair}[1]{\langle #1 \rangle}
\nc{\redst}[1]{\red{\st{#1}}}
\nc{\para}[1]{\medskip \noindent \textbf{#1.}}

\nc{\Q}{\mathbb{Q}}
\nc{\F}{\mathbb{F}}
\nc{\R}{\mathbb{R}}
\nc{\Z}{\mathbb{Z}}
\nc{\C}{\mathbb{C}}
\nc{\N}{\mathbb{N}}
\nc{\Ell}{\mathcal{L}}
\nc{\M}{\mathcal{M}}
\nc{\K}{\mathcal{K}}
\nc{\I}{\mathcal{I}}
\nc{\T}{\mathcal T}
\nc{\U}{\mathcal U}
\nc{\disk}{\mathbb{D}}
\nc{\hyp}{\mathbb{H}}

\nc{\CP}{\mathbb{CP}}
\nc{\RP}{\mathbb{RP}}
\newcommand{\Br}{Br}
\nc{\inv}{^{-1}}
\nc{\s}{\sigma}
\dmo{\LMod}{LMod}
\dmo{\Diff}{Diff}
\dmo{\Homeo}{Homeo}
\dmo{\dist}{dist}
\dmo\BDiff{BDiff}
\dmo\SO{SO}
\dmo\SL{SL}
\dmo\rank{rank}
\dmo\sig{sig}
\dmo\Out{Out}
\dmo\Aut{Aut}
\dmo\Inn{Inn}
\dmo\GL{GL}
\dmo\PGL{PGL}
\dmo\Gr{Gr}
\dmo\PSL{PSL}
\dmo\EU{EU}
\dmo\BHomeo{BHomeo}
\dmo\EHomeo{EHomeo}
\dmo\EDiff{EDiff}
\dmo\Disc{Disc}
\dmo\Aff{Aff}
\renewcommand{\Re}{\operatorname{Re}}
\renewcommand{\hat}{\widehat}
\renewcommand{\bar}{\overline}

\nc{\fd}{\mathfrak{d}}
\nc{\rad}{\varepsilon}

\dmo\Teich{Teich}
\dmo\Fix{Fix}
\nc{\abs}[1]{\ensuremath{\left| #1 \right|}}
\nc{\action}{\circlearrowright}
\nc{\abcd}[4]{\ensuremath{\left(\begin{array}{cc} #1 & #2 \\ #3 & #4 \end{array}\right)}}
\nc{\into}\hookrightarrow
\dmo{\Isom}{Isom}
\nc{\normal}{\vartriangleleft}
\dmo{\Vol}{Vol}
\dmo{\im}{Im}
\dmo{\Push}{Push}
\dmo{\Conf}{Conf}
\dmo{\PConf}{PConf}
\dmo{\PB}{PB}
\dmo{\id}{id}
\dmo{\Jac}{Jac}
\dmo{\Pic}{Pic}
\dmo{\Stab}{Stab}
\dmo{\Arf}{Arf}
\dmo{\End}{End}
\dmo{\Gal}{Gal}
\dmo{\lcm}{lcm}
\dmo{\ab}{ab}
\dmo{\opp}{op}
\dmo{\SU}{SU}
\dmo{\OT}{\Omega \mathcal{T}}
\dmo{\OM}{\Omega \mathcal{M}}
\dmo{\PH}{\mathbb{P}\mathcal{H}}
\dmo{\spin}{spin}
\dmo{\even}{even}
\dmo{\odd}{odd}
\dmo{\comp}{\mathcal{H}}
\dmo{\Mgk}{\mathcal{M}_{g, \underline{\kappa}}}
\dmo{\orb}{orb}
\dmo{\AJ}{AJ}
\dmo{\Ck}{\mathsf{C}(\underline{\kappa})}
\dmo{\Int}{Int}
\dmo{\pr}{pr}
\dmo{\lab}{lab}
\dmo{\Sym}{Sym}
\dmo{\Ann}{Ann}
\dmo{\Rad}{Rad}
\dmo{\Ind}{Ind}
\dmo{\Res}{Res}
\dmo{\Hur}{Hur}
\dmo{\absol}{abs}

\nc{\Span}[1]{\operatorname{Span}(#1)}

\renewcommand{\epsilon}{\varepsilon}
\renewcommand{\tilde}{\widetilde}
\renewcommand{\le}{\leqslant}
\nc{\coloneq}{\mathrel{\mathop:}\mkern-1.2mu=}
\nc{\margin}[1]{\marginpar{\scriptsize #1}}
\definecolor{myblue}{RGB}{102,153, 255}
\definecolor{myred}{RGB}{204,0,0}
\definecolor{mygreen}{RGB}{0,204,0}
\definecolor{myorange}{RGB}{255,102,0}
\definecolor{mypurple}{RGB}{138,43,226}
\nc{\red}[1]{\textcolor{myred}{#1}}
\nc{\blue}[1]{\textcolor{myblue}{#1}}
\DeclareMathAlphabet\mathbfcal{OMS}{cmsy}{b}{n}

\dmo{\RelAut}{PRelAut}
\dmo{\RelArf}{RelArf}
\nc{\MAOI}{\begin{pmatrix}M&A\\0&I \end{pmatrix}}
\nc{\lGamma}{\overline{\Gamma}}

\title{Siegel--Veech Constants for Cyclic Covers\\of Generic Translation Surfaces}
\author{David Aulicino, Aaron Calderon, Carlos Matheus, Nick Salter, and Martin Schmoll}
\date{}

\begin{document}

\newcommand{\splin}{\text{SL}_2(\mathbb{R})}
\newcommand{\spolin}{\text{SO}_2(\mathbb{R})}
\nc{\ann}[1]{\marginpar{\small{\red{#1}}}}

\maketitle

\begin{abstract}
We compute the asymptotic number of cylinders, weighted by their area to any non-negative power, on any cyclic branched cover of any generic translation surface in any stratum.
Our formulas depend only on topological invariants of the cover and number-theoretic properties of the degree: in particular, the ratio of the related Siegel--Veech constants for the locus of covers and for the base stratum component is independent of the number of branch values.
One surprising corollary is that this ratio for $area^3$ Siegel--Veech constants is always equal to the reciprocal of the degree of the cover.
A key ingredient is a classification of the connected components of certain loci of cyclic branched covers.
\end{abstract}

\setcounter{tocdepth}{2}

\tableofcontents
\newpage

\section{Introduction}\label{section:intro}
It was discovered in \cite{VeechSiegelMeasures, EskinMasurAsymptForms} that the growth rate of the number of periodic trajectories on a generic translation surface satisfies exact asymptotics of the form $cL^2$, where $L$ is the length of the trajectory, and $c$ is now known as a {\em Siegel-Veech constant}.
Furthermore, \cite{EskinMasurAsymptForms} found these asymptotics also hold for appropriately weighted periodic trajectories and generic surfaces with respect to any $\splin$-invariant measure.
The next breakthrough came with \cite{EskinMasurZorich}, which related Siegel--Veech constants to \emph{computable} volumes of strata.  This opened an industry of computing Siegel--Veech constants for many different invariant measures; see \cite{EskinMasurZorich, EskinMasurSchmollRectBar,  BauerGoujardSVConstants, GoujardSVConstsQuadDiffs, AthreyaEskinZorichRightAnglesBilliards, ChenMollerZagierQuasiModSiegelVeech,  SauvagetVolsSVConstsMinStrat, ChenMollerSauvagetZagierMasurVeechVolsIntersect,  windtree, DGZZ} for examples.
The main result of the present work is Theorem~\ref{MainTheoremSVFormula}, in which we derive explicit formulas for the ${area}^{s}$-Siegel--Veech constant for all loci of cyclic branched covers of all components of all strata in terms of the ${area}^{s}$-Siegel--Veech constant of the base stratum component.

To define our terms, consider a translation surface $(X, \omega)$.  Let the width of a cylinder $w(cyl)$ on $(X, \omega)$ refer to its circumference, i.e., the absolute value of the period of a core curve of the cylinder.  Let $s \geq 0$, and for any $L > 0$ and any $(X, \omega)$, define the weighted counting function
$$N_{\text{area}^{s}}((X, \omega), L) =  \sum_{\substack{cyl \subseteq (X, \omega) \\ w(cyl) < L}} \frac{\text{area}^{s}(cyl)}{\text{area}^{s}(X, \omega)}.$$
By \cite{EskinMirzakhaniInvariantMeas, EskinMirzakhaniMohammadiOrbitClosures}, every $\splin$-invariant measure is {fully} supported on an invariant subvariety $\cM$ in the moduli space of translation surfaces.
It follows from \cite{EskinMasurAsymptForms} that the limit
$$c_{\text{area}^{s}}(\cM) = \lim_{L \rightarrow \infty} \frac{N_{\text{area}^{s}}((X, \omega), L)}{\pi L^2}$$
exists and is the same for almost every $(X, \omega)$ in $\cM$ with respect to its Lebesgue measure class.  This limit is called the \emph{$\text{area}^{s}$-Siegel--Veech constant} of $\cM$. When $s = 0$, this is the classical cylinder Siegel--Veech constant, and when $s=1$ it is the classical area-Siegel--Veech constant.  We warn the reader that these constants were defined in the same way in \cite{BauerGoujardSVConstants}, but defined differently (in a non-$\splin$-invariant way) in \cite[Section~1.1]{ChenMollerZagierQuasiModSiegelVeech}.

If $(\widetilde X, \widetilde \omega) \to (X, \omega)$ is a branched covering of translation surfaces, then the $\splin$-orbit closure of $(\widetilde X, \widetilde \omega)$ is readily seen to consist of branched covers of surfaces in the $\splin$-orbit closure of $(X, \omega)$.
Mirzakhani conjectured that sufficiently large orbit closures should {\em only} arise from such branched covering constructions \cite{WrightFieldofDef}.  While this has been disproven in its original form \cite{McMullenMukamelWrightGothicLocus}, its idea seems to carry through in many cases \cite{MirzakhaniWrightFullRank, ApisaHypAISClass, AulicinoNguyenGen3HighStrat, ApisaWrightHighRank}.
Thus, understanding the properties of branched covering constructions represents significant progress toward understanding all orbit closures.

Another motivation comes from studying the large-genus asymptotics of Siegel--Veech constants.  There has recently been a great deal of progress on this problem for sequences of generic translation surfaces \cite{EskinZorichLargeGenera, ChenMollerZagierQuasiModSiegelVeech, ChenMollerSauvagetZagierMasurVeechVolsIntersect, AggarwalLargeGenusSVConsts, AggarwalLargeGenusVolsAbDiffs, AggarwalLargeGenusIntNumbers}.
Here we consider a novel method for constructing sequences of surfaces of increasing genus: by taking cyclic branched covers with the same degree, but an increasing number of branch points, we obtain sequences of orbit closures with growing dimension and growing genus, but which all have {\em the same} area-Siegel--Veech constant (Corollary \ref{cor:areaSVindepmarkedpt}).

\para{Ratios of Siegel--Veech constants}
This paper has two main results. 
The first is to compute the $\text{area}^s$-Siegel--Veech constants for each component of the locus of cyclic branched covers of surfaces in a given stratum component (Theorem \ref{MainTheoremSVFormula}).
These computations crucially rely on our second main result, in which we classify connected components of these loci (Theorem \ref{mainthm:comps}).
The precise formulas are quite intricate and will not be stated in the introduction except in the simplest case; instead, we survey some notable corollaries.

We first set some notation.
Let $\kappa = (k_1, \ldots, k_n)$ be a partition of $2g-2$ into non-negative integers. Let $\cH$ be a component of a stratum of unit-area translation surfaces with (labeled) zeros of order $\kappa$; a zero of order $0$ corresponds to a regular marked point.
Pick $(X, \omega) \in \cH$ and take a degree $d$ cyclic cover branched over the set $\cB$ of zeros and regular marked points.
For each point $p_i \in \cB$, let $d_i$ denote the {\em local monodromy} at $p_i$, i.e., the element $d_i \in \Z/d\Z$ classifying the branched cover near $p_i$. We allow some $d_i$ to be $0$, corresponding to an unbranched cover around $p_i$.

Let $\cZ \subset \cB$ be the set of points corresponding to zeros of order at least $1$; up to reindexing, we may assume that these points correspond to the first $|\cZ|$ elements of $\cB$.
Recalling that $|\cB|=n$, define
\[
\delta := (d_1, \ldots, d_n) \in (\bZ/d\bZ)^{\cB}
\hspace{1em}
\text{and}
\hspace{1em}
\pi(\delta) = (d_1, \ldots, d_{|\cZ|}) \in (\Z/d\bZ)^{\cZ}.
\]
Denote $\gcd(\delta) = \gcd(d_1, \ldots, d_n, d)$.

Let $\cM_\delta \to \cH$ be the locus of all degree $d$ cyclic branched covers of surfaces in $\cH$ with local monodromy $\delta$,
normalized to have unit area.
Depending on the base stratum $\cH$ and the invariants defined above, $\cM_\delta$ may not be connected, and when it is not, its components are classified by certain ``$\psi$ or $b$ invariants'' defined at the level of the mod-2 homology of $X$ relative to $\cB$ (see Definitions \ref{definition:psi} and \ref{def:bint}, and compare Theorem \ref{mainthm:comps}).

One of the most surprising features of our formulas is that they do not depend on the number of branch points, but only on the structure of the base stratum and coarse numerical data about the branching.

\begin{maintheorem}[Corollary of \Cref{MainTheoremSVFormula}]
\label{IndependenceBPCondThm}
Let $\cH$ be a connected component of a stratum of genus $g$ translation surfaces (possibly with regular marked points) and let $\cC$ be a connected component of $\cM_\delta$.
Then
\[\frac{c_{\text{area}^{s}}(\cC)}{c_{\text{area}^{s}}(\cH)}\]
depends only on $g$, $s$, $d$, $\gcd(\delta)$, the $\psi$ or $b$ invariant of $\cC$ (when defined), and the Arf invariant of $\cH$ (when defined).
In particular, this quotient is independent of the number of branch values $n = |\cB|$.
\end{maintheorem}

\begin{remark}
We emphasize that Theorem~\ref{MainTheoremSVFormula} allows one to compute numerical values for $c_{\text{area}^{s}}(\cC)$.
The denominator $c_{\text{area}^{s}}(\cH)$ can be computed in terms of cylinder Siegel--Veech constants via \cite{BauerGoujardSVConstants}, which can in turn be computed recursively using \cite{EskinMasurZorich}. 
A similar reduction was used to compute area-Siegel--Veech constants in \cite[$\S$2.5.2]{EskinKontsevichZorich2}, which relied on a theorem of \cite{VorobetsPerGeodsGenTransSurfs}.
\end{remark}

The value of the denominator in Theorem \ref{IndependenceBPCondThm} generally {\em does} depend on $n$.
However, for $s=1$, adding regular marked points does not change the area-Siegel--Veech constant of the base stratum (since the areas of cylinders add linearly), and consequently, the numerator also does not change.

\begin{maincor}\label{cor:areaSVindepmarkedpt}
Let $\cH$ be a connected component of a stratum of genus $g$ translation surfaces (possibly with regular marked points) and let $\cC$ be a connected component of $\cM_\delta$. Let $\cH'$ be the stratum component obtained from $\cH$ by forgetting all regular marked points.
Then $c_{\text{area}}(\cC)$ depends only on $g$, $d$, $\gcd(\delta)$, the $\psi$ or $b$ invariant of $\cC$ (when defined), and the Arf invariant of $\cH'$ (when defined).
In particular, it is independent of the number of branch values $n = |\cB|$.
\end{maincor}

We now highlight a particularly striking example of the phenomenon observed in Theorem \ref{IndependenceBPCondThm}; note that the following result does not depend on which stratum component $\cH$ or which component of $\cM_\delta$ is chosen.
To our knowledge, $area^{3}$-Siegel--Veech constants have never been specifically considered in the literature.

\begin{maintheorem}[Corollary of \Cref{MainTheoremSVFormula}]
\label{Area3SVdOdd}
Let $\cH$ be a component of a stratum of unit-area translation surfaces (possibly with regular marked points) and let $\cC$ be a connected component of $\cM_\delta$.
Then
\[\frac{c_{\text{area}^{3}}(\cC)}{c_{\text{area}^{3}}(\cH)} = \frac{1}{d}.\]
\end{maintheorem}

\para{A (facile) heuristic}
Consider a degree $d$ translation covering $\pi: (\tilde X, \tilde \omega) \to (X, \omega)$ and suppose that there were some $c$ such that each cylinder on $(X, \omega)$ lifts to exactly $c$ cylinders on $(\tilde X, \tilde \omega)$, each of circumference $d/c$ times its circumference on $(X, \omega)$.\footnote{
To make this heuristic slightly more sophisticated, one could imagine that the cylinders of $(X, \omega)$ were ``well-distributed'' with respect to the cover, so that $c$ is an {\em average} number of preimages of a cylinder on $(X, \omega)$.}
We could then estimate
\begin{align*}
N_{\text{area}^s}((\tilde X, \tilde \omega), L)
& := 
\sum_{\substack{\widetilde{cyl} \subseteq (\tilde X, \tilde \omega) \\ w(\widetilde{cyl}) < L}} \frac{\text{area}^{s}(\widetilde{cyl})}{\text{area}^{s}(\tilde X, \tilde \omega)} \\
& =
\sum_{\substack{cyl \subseteq (X, \omega) \\ w(cyl) < Lc/d}} c \cdot \frac{\text{area}^{s}(cyl) (d/c)^s}{d^s}
= c^{1-s} N_{\text{area}^s}((X, \omega), Lc/d).
\end{align*}
Normalizing the covering surface to have unit area and using the fact that our counting function shrinks (approximately) linearly with area and grows (approximately) quadratically with length, this yields
\[
N_{\text{area}^s}((\tilde X, \tilde \omega / \sqrt{d}), L)
\approx
d \cdot N_{\text{area}^s}((\tilde X, \tilde \omega), L)
\approx
\frac{c^{3-s}}{d} N_{\text{area}^s}((X, \omega), L).\]
Thus, we would expect that
\[{c_{\text{area}^{s}}(\cC)} = \frac{c^{3-s}}{d}{c_{\text{area}^{s}}(\cH)}.\]
This heuristic is incredibly na{\"i}ve, and shockingly, correct when $s=3$.
Our formulas for general $s$ bear some resemblance to this heuristic: compare especially Theorem \ref{GeneralFormula}.

\para{Components of loci of covers}
The reason that the $\psi$ and $b$ invariants appear in Theorem \ref{IndependenceBPCondThm} and Corollary \ref{cor:areaSVindepmarkedpt} is because whenever $\cM_\delta$ is disconnected, they form a family of complete invariants distinguishing its connected components.
Observe again that the number $n$ of branch values does not appear in the classification below.

\begin{maintheorem}\label{mainthm:comps}
Let $\cH$ be a connected component of a stratum of unit-area translation surfaces with zeros of order $\kappa$, possibly with regular marked points. 
Let $\cM_\delta \to \cH$ be the locus of degree $d$ cyclic branched covers with local branching data $\delta$, and let $\cH'$ be obtained from $\cH$ by forgetting all regular marked points.
\begin{enumerate}
    \item If $d$ is odd or $g=1$, then $\cM_\delta$ is connected.
    \item If $\cH'$ is nonhyperelliptic and $d$ is even,
    \begin{enumerate}
        \item If $\delta \neq \kappa \pmod 2$, then $\cM_\delta$ is connected.
        \item If $\delta = \kappa \pmod 2$, then 
        $\cM_\delta$ has two components, classified by the $\psi$ invariant.
    \end{enumerate}
    \item If $\cH'$ is hyperelliptic and $d$ is even,
    \begin{enumerate}
        \item If $\pi(\delta) \neq 0$, then $\cM_\delta$ is connected.
        \item If $\pi(\delta) = 0$, then the components of $\cM_\delta$ are classified by the $b$ invariant.
        \begin{itemize}
            \item If $\cH' = \cH^{hyp}(2g-2)$, there are $g$ components.
            \item If $\cH' = \cH^{hyp}(g-1, g-1)$, there are $g+1$ components.
        \end{itemize}
    \end{enumerate}
\end{enumerate}
\end{maintheorem}

When $\cM_\delta$ is connected, the formulas in Theorem \ref{MainTheoremSVFormula} simplify.
To give the reader an idea of their form, we include this reduced version below.

Set $d_{abs}$ to be the greatest divisor of $d$ coprime to $\gcd(\delta)$ and set $d_{rel} = d/d_{abs}$.

\begin{maintheorem}
\label{MainThmIntro}
If $\cM_\delta$ is connected (cases 1, 2a, and 3a above), then for all $s \geq 0$,
\[\frac{c_{\text{area}^{s}}(\cM_\delta)}{c_{\text{area}^{s}}(\cH)} 
= \frac{d_{rel}}{d^{s-1}\Phi_{2g}(d_{rel})}\left(\sum_{\fd|d_{rel}} \Phi_{2g-1}\left(\frac{d_{rel}}{\fd}\right) \frac{\Phi(\fd)}{\fd^{4 - 2g - s}}  \right) \cdot \left(\sum_{\fd|d_{abs}} \left(\frac{\Phi(\fd)}{\fd^{3-s}}\right) \right),\]
where $\Phi_m$ denotes the Jordan totient function (see $\S$\ref{stmtsplit}) and $\Phi$ is the usual Euler totient function.
\end{maintheorem}

\begin{remark}
Even though this will not be used in this work, we note that the formula in Theorem~\ref{MainThmIntro} can be rewritten using a Dirichlet convolution (denoted $\ast$).  Let $Id_k(n) = n^k$ and $1(n) = 1$.  Then
$$\frac{c_{\text{area}^{s}}(\cM_\delta)}{c_{\text{area}^{s}}(\cH)} = \left(\frac{d_{rel}}{\Phi_{2g}(d_{rel})} \right) \left( (Id_{2g+s - 4} \cdot \Phi) \ast \Phi_{2g-1} \right)(d_{rel}) \cdot \left( (Id_{s - 3} \cdot \Phi) \ast 1 \right)(d_{abs}).$$
\end{remark}

In \cite{SchmollOrbitClassdSymmCovs}, the fifth-named author derived this formula for $\text{area}^s$-Siegel--Veech constants for cyclic torus covers branched over \emph{two} points. The formula there differs from ours due to a different choice of normalization.
We emphasize that the formula of Theorem
\ref{MainThmIntro} is new even for torus covers. Substituting $g = s = 1$ into \Cref{MainThmIntro} and using a standard identity for totient functions, we get the following notable corollary.

\begin{maincor}
\label{TorusCovAreaSVCor}
Let $\cM_\delta$ be the locus of degree $d$ cyclic covers branched over a unit area flat torus with $n \geq 2$ marked points, normalized to have unit area. Then
\[c_{\text{area}}(\cM_\delta) = \frac{3}{\pi^2} \sum_{\fd|d_{abs}} \frac{\Phi(\fd)}{\fd^{2}}= \frac{3}{\pi^2}\left( (Id_{-2} \cdot \Phi) \ast 1 \right)(d_{abs}),\]
where $\Phi$ denotes the Euler totient function.
\end{maincor}

In \cite{ChenMollerZagierQuasiModSiegelVeech}, formulas for area-Siegel--Veech constants for arbitrary branched covers of tori with arbitrarily many marked points are presented.
There does not appear to be an obvious simplification to arrive at the formula in Corollary~\ref{TorusCovAreaSVCor} from those of \cite{ChenMollerZagierQuasiModSiegelVeech}.
We also observe that the setting of this corollary is distinct from the results of \cite{EskinKontsevichZorichSqTiled} because they consider tori branched over the $2$-torsion points while our branching occurs over free marked points.

Given the importance of area-Siegel--Veech constants, we give the full general formula for them in all cases in Section~\ref{FullAreaSVConst:Section}.
On the other hand, we do not explicitly state the cylinder Siegel--Veech constant for $\cM_\delta$ in this manuscript because it is not particularly striking.  The interested reader can substitute $s = 0$ into the formulas in Theorem~\ref{MainTheoremSVFormula} to recover the values in all cases.

\para{Large genus asymptotics}
Since our formulas do not depend on the number of branch points, by adding more and more branch points to our cover we can build sequences of surfaces with growing genus but with the same area-weighted asymptotic number of cylinders.

Here is an example of one such sequence. Fix a stratum component $\cH' \subset \cH(k_1, \ldots, k_n)$, possibly with regular marked points.
Let $d$ be odd and fix some $d_1, \ldots, d_n \in \bZ/ d\bZ$ such that $\gcd(d_1, \ldots, d_n)=1$. 
By appending $m$ 0's to the partition (i.e., marking $m$ freely-varying regular points on the translation surfaces in $\cH'$), this picks out a stratum component $\cH \subset \cH(k_1, \ldots, k_n, 0^m)$.
Fix any local monodromy data about the $m$ regular marked points to complete $(d_1, \ldots, d_n)$ to $\delta$, and let $\cM_{m, \delta}$ be the resulting locus of degree $d$ cyclic branched covers of surfaces in $\cH$. Both the dimension of $\cM_{m,\delta}$ and the genus of the covering surfaces clearly grow as $m$ increases.

Theorem \ref{mainthm:comps} ensures that $\cM_{m, \delta}$ is connected (and hence the $\psi$ and $b$ invariants are not defined), and our assumptions imply that $\gcd(\delta) =1$.
Thus, Corollary \ref{cor:areaSVindepmarkedpt} states that $c_{\text{area}}(\cM_{m, \delta})$
is always the same for any choice of $m$ and any choice of local monodromy data about the regular points.\footnote{One could of course let $d$ be even and relax the assumption on $\gcd(d_i)$: one just needs to choose branching data so that $M$ and the $\psi$ and $b$ invariants agree.}

This should be compared with what is known about the area-Siegel--Veech constants of strata: any sequence of strata $\cH_n$ with genus tending to infinity has $c_{\text{area}}(\cH_n)$ tending to $1/2$ \cite{AggarwalLargeGenusSVConsts}, but any such sequence of constants is not a constant sequence \cite{EskinKontsevichZorich2}.
The most similar sequence of orbit closures of which we are aware lies in \cite{FougeronLyapExpsQuadDiffsPoles}, which considers sequences of strata of quadratic differentials with fixed genus but increasingly many simple poles. 
Taking the loci of canonical double covers over these strata yields a sequence of orbit closures with dimension and genus tending to infinity for which Fougeron bounds the sum of Lyapunov exponents and thus the area Siegel--Veech constant.

\para{Growth of Lyapunov exponents}
These results can also be used to bound the growth of the smallest Lyapunov exponent for the orbit closures $\cM_{m, \delta}$.
The Eskin--Kontsevich--Zorich formula \cite[Theorem 1]{EskinKontsevichZorich2} asserts that if $\mathcal{M}$ is an orbit closure in a stratum $\mathcal{H}(k_1,\dots, k_n)$ of translation surfaces of genus $g$, then the sum of the first $g$ Lyapunov exponents of the Kontsevich--Zorich cocycle over $\cM$ is 
\[1+\lambda_2(\mathcal{M}) +\dots+\lambda_g(\mathcal{M}) 
= \frac{1}{12}\sum\limits_{i=1}^n \frac{k_i(k_i+2)}{(k_i+1)}
+\frac{\pi^2}{3}c_{\text{area}}(\mathcal{M}),\]
where $c_{\text{area}}(\mathcal{M})$ is the usual ($s=1$) area-Siegel--Veech constant.

Since the combinatorial term 
is bounded by 
\[\frac{1}{12}\sum\limits_{i=1}^n\frac{k_i(k_i+2)}{(k_i+1)}\leq \frac{1}{12}\frac{3}{2}(2g-2) = \frac{(g-1)}{4},\]
and for each $2\leq l\leq g$ the sum of Lyapunov exponents is bounded by  
$$1+(l-1)\lambda_l(\mathcal{M})\leq 1+\lambda_2(\mathcal{M})+\dots+\lambda_g(\mathcal{M}),$$ 
it follows from Theorem~\ref{IndependenceBPCondThm} that if $g_{m}$ is the genus of the translation surfaces in $\mathcal{M}_{m, \delta}$,
then\footnote{More generally, if $h:\mathbb{N}\to\mathbb{N}$ is asymptotic to the identity (i.e., $h(N)/N\to 1$ as $N\to\infty$), one also has $\limsup_{m\to\infty} \lambda_{h(g_m)}(\mathcal{M}_{m, \delta})\leq 1/4$.} the smallest Lyapunov exponent of $\mathcal{M}_{m, \delta}$ satisfies 
\[\limsup_{m\to\infty} \lambda_{g_m}(\mathcal{M}_{m, \delta})\leq 1/4.\]
(This should be compared with the discussion in the beginning of \cite[$\S$2.4]{EskinKontsevichZorich2}.)

\begin{remark} There are heuristic reasons to believe that the inequality above is probably not sharp. Indeed, several numerical simulations by the first named author indicate that $\lambda_{g_m}(\mathcal{M}_{m, \delta})$ vanishes for all $m$ sufficiently large. 
Furthermore, it is likely that one could rigorously establish this numerical prediction by studying the signatures $(p,q)$ of the eigenspaces of the deck group $\bZ / d\bZ$ acting on the tangent space of $\mathcal{M}_{m, \delta}$.  In fact, these signatures can often be computed by direct methods (cf. \cite{McMullenBraidGroupsHodge} 
and \cite{AvilaMatheusYoccozVeechMcMullenFamily}) 
or by using Chevalley--Weil type formulas from representation theory (cf. \cite{MatheusYoccozZmiaikouHomOfOrigs}). 
Moreover, if the restriction of the Hodge inner product to an eigenspace has signature $(p,q)$ with $p>q$, then that component contributes at least $p-q$ zero Lyapunov exponents (cf. \cite{ForniMatheusZorichZeroLyapExpsHodge}).
Similarly, it is also possible that the inequality in Footnote 3 is not sharp.
\end{remark}

\para{Notable aspects of the proof}
Basic algebraic topology implies that $\Z/d\Z$ cyclic covers of $X$ branched over $\cB$ are classified by the cohomology group $H^1(X \setminus \cB; \Z/d\Z)$.
Theorem \ref{mainthm:comps} is therefore equivalent to understanding the set of orbits of the monodromy group of $\cH$ on $H^1(X \setminus \cB; \Z/d\Z)$.

The mapping class group-valued monodromy of a non-hyperelliptic stratum-component in genus $g \ge 5$ was studied by the second and fourth-named authors in a sequence of papers \cite{C_strata1,strata2,strata3}. In the companion paper \cite{CalderonSalterRelHomRepsFramedMCG}, the monodromy action on $H^1(X \setminus \cB; \Z/d\Z)$ was determined, again for $g \ge 5$. In spite of this progress, the results of the present paper require substantial new developments. The first issue is that the results of \cite{C_strata1,strata2,strata3,CalderonSalterRelHomRepsFramedMCG} require $g \ge 5$, but there are non-hyperelliptic components in genus three and four to be studied as well. To that end, in \Cref{theorem:homologicalmonodromy} we give a new proof of the monodromy calculation of \cite{CalderonSalterRelHomRepsFramedMCG} valid for all non-hyperelliptic stratum components in $g \ge 3$. 

The second and more substantial issue is that even knowing the monodromy action on $H^1(X\setminus \cB; \Z/d\Z)$, it remains to  understand the set and size of orbits. 
In \cite{MageeRuhrSaddConnsModq}, it was proven that for most $d$, the number of cylinders in a given mod $d$ homology class grows quadratically. Computing the actual growth rate for all $d$ requires finer information.
This turns out to be a subtle problem with some surprising features, e.g., the presence of ``split'' orbits: in certain numerical regimes, certain branched covers which are equivalent under the action of the full mapping class group are not monodromy-equivalent; ultimately this is explained by the presence of {\em framings} on the translation surfaces.
Finding the cardinalities is much easier in the non-split case, and such an analysis was already carried out in \cite[Lemmas~2.6-2.8]{MageeRuhrSaddConnsModq}.

The existence of split orbits is in one sense similar to Kontsevich--Zorich's classification \cite{KontsevichZorichConnComps} of connected components of strata, where certain strata are split into path components according to the Arf invariant of the framing, but in our case, {\em every} non-hyperelliptic stratum has exactly one split orbit (\Cref{prop:2splitorbits}).
In the non-hyperelliptic regime, a single mapping class group orbit of branched covers can split into two monodromy orbits, classified by the ``$\psi$ invariant'' defined in Section \ref{section:psi}. In the hyperelliptic setting, split orbits are governed by the ``$b$-invariant'' of Section \ref{section:splithyp}.

This work of determining monodromy groups, classifying orbits, computing their cardinalities, and computing a second associated numerical invariant (see \Cref{FixedHomClassSection}) occupies a large part of the paper, spanning Sections \ref{section:wnf}--\ref{section:splithyp} as well as the two short appendices.
This Part~\ref{OrbitCard:Part} of the paper can be largely read independently from the calculation of Siegel-Veech constants  carried out in Part~\ref{SiegelVeech:Part}.
The reader only interested in Part~\ref{SiegelVeech:Part} can safely take the formulas in Part~\ref{OrbitCard:Part} as black boxes.

With an understanding of the monodromy orbits and connected components of $\cM_\delta$ in hand, we compute the $\text{area}^s$-Siegel--Veech constants in Part~\ref{SiegelVeech:Part}.
Our main tool is a method for computing Siegel--Veech constants of $\cM_\delta$ by defining appropriate counting functions on the base stratum $\cH$.
This technique is well-established in the setting of (branched) torus covers \cite{EskinMasurSchmollRectBar,  ChenMollerZagierQuasiModSiegelVeech} and \cite[$\S$10]{EskinKontsevichZorich2}, but to our knowledge this paper represents the first generalization to covers of higher genus surfaces.
Strata of higher genus surfaces are far more complicated than moduli spaces of punctured tori, so we must consider the structure of cusps as in \cite{EskinMasurZorich}, thus making our arguments more involved.
We believe that the perspectives and computations used here for addressing cyclic covers will be valuable to those interested in non-cyclic covers of arbitrary translation surfaces.

\para{Organization} The paper is divided into three parts, accompanied by two short appendices. Part~\ref{PrelimSetup:Part} (Sections \ref{Sect:SetupRedHom}--\ref{MonodromyVectorAction}) contains introductory material and sets notation.
Part \ref{OrbitCard:Part} (Sections 
\ref{section:unsplitorbits}--\ref{section:splithyp}) concerns the topological problem of determining monodromy groups and connected components of loci of branched covers, and  Part \ref{SiegelVeech:Part} (Sections \ref{SVBackSetSect}-\ref{subsection:mainformula}) carries out the computation of Siegel--Veech constants. Appendix \ref{section:sympmodules} contains some standard results on the classification of orbits of vectors under the action of certain subgroups of the symplectic group over $\Z/d\Z$, and Appendix \ref{section:cyltrans} provides an alternative proof of the classification of monodromy orbits of homology classes of cylinders. A detailed section-by-section outline is provided at the beginning of each part.

\para{Acknowledgments} The authors are very grateful to Anton Zorich for numerous discussions and careful explanations of material related to Siegel--Veech constants as well as his continued interest throughout the entirety of the project.  The project originated from computer experiments using code written by Charles Fougeron that was made publicly available \cite{flatsurf}.  The authors are grateful to Vincent Delecroix and Julian R\"uth for assistance with coding.  The authors are also grateful to Jayadev Athreya, Vincent Delecroix, Alex Eskin, Michael Magee, Howard Masur and Martin M\"oller for helpful conversations and interest.

D.A. was partially supported by the National Science Foundation under Award Nos. DMS - 1738381, DMS - 1600360, the Simons Foundation: Collaboration Grant under Award No. 853471, and several PSC-CUNY Grants. A.C. was supported by NSF Award Nos.
DMS-2005328 %Yair
and DMS-2202703. %MSPRF
C.M. was partially supported by ANR-23-CE40-0020 MOST. N.S. was supported by NSF Awards No. DMS-2153879 and DMS-2338485. M.S. was partially supported by Simons Foundation: Collaboration Grant under Award No. 318898 and Simons Foundation International [SFI-MPS-TSM-00013668, M.S.].

\part{Preliminaries and setup}
\label{PrelimSetup:Part}

\para{Outline of Part~\ref{PrelimSetup:Part} (Sections \ref{Sect:SetupRedHom}-\ref{MonodromyVectorAction})}
\begin{itemize}
\item In \Cref{Sect:SetupRedHom}, we discuss relevant aspects of the theory of cyclic branched covers.
\item In \Cref{MonodromyVectorAction}, we recall various kinds of {\em monodromy} in the context of strata, and define the numerical invariants $\abs{\overline{\Gamma}_d \balpha}$ and $\abs{(\overline{\Gamma}_d \balpha)_a}$ whose computation will occupy us in Part~\ref{OrbitCard:Part}.
\end{itemize}

\section{Branched covers of translation surfaces}
\label{Sect:SetupRedHom}

In this brief section we recall the basics of branched covers and set some notation for the rest of the paper.

\subsection{Translation surfaces}
A \emph{translation surface} $(X, \omega)$ is a Riemann surface $X$ of genus at least one carrying a holomorphic Abelian differential $\omega$.  Throughout this manuscript, $\omega$ will often be suppressed. 
The differential $\omega$ induces an almost everywhere flat metric on the Riemann surface $X$, which has a finite number of conical singularities with angles that are a multiple of $2\pi$ when the genus is at least two.  Furthermore, $(X, \omega)$ admits vertical and horizontal singular foliations given by integrating the kernels of the real and imaginary parts of $\omega$.
We permit $(X, \omega)$ to have marked points.  A \emph{regular closed trajectory} is a closed trajectory that does not pass through any of the singularities or marked points of $(X, \omega)$.  A \emph{cylinder} is a maximal homotopy class of regular closed trajectories relative to the singularities and marked points of $(X, \omega)$.  For example, a torus with two marked points typically decomposes into two cylinders of parallel closed trajectories.  Cylinders will always be open, i.e., homeomorphic to the Cartesian product of a finite open interval with the circle.

There is a canonical area associated to a flat surface by
$$\text{area}(X, \omega) = \frac{i}{2}\int_X \omega \wedge \overline{\omega}.$$

\subsection{Branched covers and branching vectors}
\label{CoversMonRepSect}

For the translation surface $(X, \omega)$, let $\cB \subset X$ be a distinguished set of points {\em containing all of the zeros of $\omega$}, possibly along with some ordinary marked points. We are interested in describing the set of $\Z/d\Z$ cyclic branched covers $p: \widetilde X \to X$ branched only over points in $\cB$. {\em A priori}, such a cover is classified by a homomorphism

\[
\balpha: \pi_1(X \setminus \cB) \to \Z/d\Z,
\]
with $\tilde X$ connected if and only if $\balpha$ is surjective. Since the target is Abelian, it follows that $\balpha$ is equivalent to a class 
\[
\balpha \in H^1(X \setminus \cB; \Z/d\Z) \cong \Hom(H_1(X \setminus \cB; \Z), \Z/d\Z).
\]
We call $\balpha$ the {\em branching vector} of the branched covering.\footnote{In this manuscript, {\em coverings} appear in two distinct (but intertwined) settings: there is the {\em local} phenomenon of branched coverings of a fixed translation surface under discussion here, and the induced {\em global} covering of a stratum of translation surfaces consisting of the family of all branched covers. We will have occasion to consider ``monodromy actions'' of the fundamental group of the base on the fiber in each setting. In the interest of clarity, in the branched cover setting, we will speak as much as possible about the {\em branching vector}, but the reader should be aware of the potential for confusion.}

An element of $(\Z/d\Z)^k$ (for $d \ge 0$) is said to be {\em primitive} if its components generate $\Z/d\Z$.
Surjectivity of $\balpha: \pi_1(X\setminus \cB) \to \Z/d\Z$ (and hence connectivity of the associated branched cover) is equivalent to the condition that $\balpha \in H^1(X \setminus \cB; \Z/d\Z) \cong (\Z/d\Z)^k$ be primitive. Accordingly, we write
\[
H^1_{prim}(X\setminus \cB; \bZ/d\bZ) \subset H^1(X\setminus \cB; \bZ/d\bZ)
\]
for the subset of primitive elements. 

In the course of our arguments, we will need to understand the {\em local} behavior of a branched cover at a branch point, as well as some associated numerical invariants. 

\begin{definition}\label{definition:holonomy}
Given $\cB = \{p_1, \dots, p_n\} \subset X$, for $1 \le i \le n$, choose $\gamma_i \subset X$ to be an oriented loop encircling a small neighborhood of $p_i$. The {\em local monodromy of $\balpha$ at $p_i$} is the element
\[
d_i = \balpha(\gamma_i) \in \Z/d\Z.
\]
We define the quantity $\gcd(d_1, \ldots, d_n, d)$ to be the {\em monodromy factor} of the covering $\balpha$, and we denote it by $\gcd(\delta(\balpha))$, which we treat as a formal symbol for this quantity for the moment (for the definition of $\delta(\balpha)$, see \Cref{subsection:hommonodromy}).
\end{definition}

\subsection{Cylinders}
\label{CylinderSubSect}
The branching vector $\balpha$ determines how curves lift to the branched cover.
In general, this behavior is quite intricate
(see, for example, \cite[\S2.3, Eqn. (24)]{ChenMollerZagierQuasiModSiegelVeech} which describes how cylinders lift for general torus covers with $n$ generic marked points). 
In our case, the situation is dramatically simpler since the deck group is Abelian.

\begin{proposition}
\label{CylinderMonodromy:Prop}
Let $(X, \omega)$ be a translation surface with distinguished set $\cB \subset X$. Suppose that $C_1, \ldots, C_m \subset X  \setminus \cB$ are curves such that $[C_1] = \dots = [C_m]$ in $H_1(X; \bZ)$. Let $\balpha$ be a branching vector with monodromy factor $\gcd(\delta(\balpha))$. If $\balpha(C_1) = a$, then for each $C_i$, there exists $b_i(\balpha) \in \bZ/d\bZ$ such that $\balpha(C_i)$ is $a + b_i(\balpha) \gcd(\delta(\balpha))$.
\end{proposition}

When applying this Proposition, the curves $C_i$ will be core curves of cylinders.

\begin{proof}
If $[C_i] = [C_1]$ as elements of $H_1(X;\Z)$, then in the excision homology group $H_1(X \setminus\cB;\Z)$, 
\[
[C_i]-[C_1] \in \ker(H_1(X \setminus \cB; \Z) \to H_1(X; \Z)),
\]
so that $[C_i] - [C_1]$ can be written as a sum 
\[
[C_i]-[C_1] = \sum_{p_k \in \cB} m_k \gamma_k, 
\]
for $m_k \in \bZ$, with $\gamma_k$ a loop encircling $p_k$ as in \Cref{definition:holonomy}. Thus, taking $d_k = \balpha(\gamma_k)$ as in \Cref{definition:holonomy},
\[
\balpha(C_i) = \balpha(C_1) + \sum_{p_k \in \cB} m_k d_k,
\]
and the difference $\sum_{p_k \in \cB} m_k d_k$ is a multiple of $\gcd(\delta(\balpha)) \pmod d$ by definition.
\end{proof}

For later use, we observe that (the proof of) Proposition \ref{CylinderMonodromy:Prop} shows that if we fix local monodromy data around the punctures, then the monodromy of all of the $C_i$ depends only on the monodromy of one of them.

\begin{lemma}
\label{Constantbvector:lemma}
Suppose $\balpha$ and $\bbeta$ are branching vectors such that the local monodromies are equal at each $p_k \in \mathcal B$ (i.e. $\delta(\balpha) = \delta(\bbeta)$, in the notation of the next section) and let $C_1, \ldots, C_m \subset X \setminus \cB$ be a set of curves representing the same class in $H_1(X; \bZ)$.
If $\balpha(C_1) = \bbeta(C_1)$, then $\balpha(C_i) = \bbeta(C_i)$ for all $i$.
\end{lemma}

The following lemma follows from the basic principles of covering space theory.

\begin{lemma}\label{lemma:cyllifts}
Let $(\widetilde X, \widetilde \omega) \to (X, \omega)$ be a branched translation covering with branching vector $\balpha$.
Suppose that $C_1, \ldots, C_m$ are cylinders on $(X, \omega)$ each of width $w$ that satisfy the assumptions of \Cref{CylinderMonodromy:Prop}. Then $C_i$ lifts to
\[\sharp(C_i) = \gcd(\balpha(C_i), d)\]
cylinders on $(\widetilde X, \widetilde \omega)$, each of width
\[\frac{d}{\sharp(C_i)}w = \frac{d}{\gcd(\balpha(C_i), d)}w.\]
\end{lemma}

\section{Covers of strata and monodromy}
\label{MonodromyVectorAction}

We now shift our perspective to a {\em global} one, and consider {\em families} of branched covers of translation surfaces varying in a component of a stratum. 

\subsection{Strata}
A stratum $\cH(\kappa)$ of translation surfaces is a moduli space where the orders of the singularities are specified by $\kappa$, a partition of $2g-2$.  We permit $\kappa$ to be a non-negative partition with finitely many zero entries corresponding to marked points at which $\omega$ does not have a singularity.
Let $\cH_{lab}(\kappa)$ denote the cover of $\cH(\kappa)$ on which all of the singularities with orders specified by $\kappa$ are labeled.

Strata are not necessarily connected; a complete classification of their components was given in \cite{KontsevichZorichConnComps}. In summary, all strata are connected if they have genus two or an odd order zero and are not of the form $\cH(g-1, g-1)$. 
The strata $\cH(2g-2)$ and $\cH(g-1, g-1)$ always have a unique hyperelliptic component $\cH^{hyp}(\kappa)$, in which every translation surface in them admits a hyperelliptic involution.
For strata with all even order zeros, the surfaces admit a {\em spin structure} which can have even or odd parity (cf. the discussion of \Cref{section:wnf,section:psi}). In this case, if $g\ge 4$, then $\cH(\kappa)$ always has two components not consisting entirely of hyperelliptic differentials, classified by spin parity. In genus $3$, the even spin component coincides with the hyperelliptic component, and in genus 2, all components are hyperelliptic.
Throughout, we will use $\cH$ to denote a connected component of a stratum, suppressing the partition $\kappa$ when it is understood.

We will also consider strata with extra marked (regular) points. Forgetting the regular points associates to any component of $\cH(\kappa)$ a component of $\cH(\kappa')$, where $\kappa'$ is the result of deleting all of the $0$'s from $\kappa$.
Because our analysis will depend on which component of $\cH(\kappa')$ we are considering, we make the following convention:

\begin{convention}
We say that a stratum component $\cH$ is hyperelliptic if the result of forgetting all regular marked points is a hyperelliptic stratum component.
\end{convention}

For all $A > 0$, there is a natural cross section $\cH_A(\kappa)$ of $\cH(\kappa)$ given by taking those translation surfaces with area $A$.  In particular, the stratum of labeled unit area translation surfaces is denoted $\cH_{1, lab}(\kappa)$, and those with area at most one is denoted $\cH_{\leq 1, lab}(\kappa)$.

Each stratum is a quasi-projective variety that is finitely covered by a complex manifold, and there are natural ``period coordinates'' on such covers.  Given a translation surface $(X, \omega) \in \cH(\kappa)$, let $\cB$ denote the collection of all singularities and marked points of $(X, \omega)$. \emph{Period coordinates} locally model a stratum on $H^1(X, \cB; \bC)$ by integrating a basis of $H_1(X, \cB; \bZ)$ against the differential $\omega$.  Strata admit a natural measure known as the \emph{Masur-Veech measure} $\nu$, which is given by pulling back Lebesgue measure on $\bC^{2g+n-1} \cong H^1(X, \cB; \bC)$ under period coordinates. Coning off induces a finite measure $\nu_1$ on $\cH_1(\kappa)$ \cite{MasFinMeasErg, VeechFinMeasErg}.

\subsection{Monodromy representations}

Let $\cH$ be a component of $\cH_{lab}(\kappa)$. Taking (orbifold) $\pi_1$, the natural map $\cH \to \cM_{g,n}$ induces a {\em (topological) monodromy representation}
\[
\rho_{top}: \pi_1^{orb}(\cH) \to \PMod(X, \cB),
\]
where $\PMod(X,\cB)$ denotes the mapping class group of the surface $X$ relative to $\cB$, i.e., the isotopy classes of orientation-preserving diffeomorphisms of $X$ that fix $\cB$ pointwise (here $X \in \cH$ is an arbitrary basepoint). The map $\rho_{top}$ measures how a chosen marking of $(X, \cB)$ changes as $(X,\cB)$ is transported around loops in $\cH$, compare \cite[$\S$ 5.6.1]{FarbMargalitMCG}. Let 
\[
im(\rho_{top}) =: \Gamma \le \PMod(X, \cB)
\]
be the {\em topological monodromy group} of $\cH$. The topological monodromy group of a hyperelliptic component was classically understood to be the hyperelliptic mapping class group, and was recently determined for non-hyperelliptic components in genus $g \ge 5$ by the second and fourth named authors \cite{strata3}, where it is shown to be a ``framed mapping class group.''

The topological monodromy group acts on various objects associated to a surface. In our setting, the object of interest is the set of connected $\Z/d\Z$ branched covers, i.e., the set $H^1_{prim}(X \setminus \cB; \Z/d\Z) \subset H^1(X \setminus \cB; \Z/d\Z)$ discussed in the previous section. Accordingly, the monodromy group of primary interest in this manuscript will be the image 
\[
\overline{\Gamma}_d \subseteq \Aut(H^1(X\setminus \mathcal B; \Z/d\Z)).
\]
We call $\overline{\Gamma}_d$ the {\em mod $d$ homological monodromy group} of the stratum component $\cH$.  (We will omit the subscript in the case $d = 0$). Note that we will often be quite lax about reducing mod $d$, letting elements $A \in \Aut(H^1(X\setminus \mathcal B; \Z))$ act on $H_1(X\setminus \cB; \Z/d\Z)$.

\subsection{Homological monodromy: basic structure}\label{subsection:hommonodromy}

\para{Duality and intersection pairings} The relative intersection pairing
\[
\pair{\cdot, \cdot}: H_1(X \setminus \mathcal B; \Z/d\Z) \otimes H_1(X, \mathcal B; \Z/d\Z) \to \Z/d\Z
\]
furnishes a canonical isomorphism
$H^1(X \setminus \mathcal B; \Z/d\Z) \cong H_1(X, \mathcal B; \Z/d\Z)$, which intertwines the actions of $\PMod(X, \mathcal B)$ on both sides.
Thus, we can think of $\overline{\Gamma}_d$ as acting on $H_1(X, \mathcal B; \Z/d\Z)$, which will be better suited for our purposes. The reader should be aware that we will tacitly make use of this isomorphism to think of elements of $H_1(X, \mathcal B; \Z/d\Z)$ as {\em cohomology} classes (or elements of the dual space $(H_1(X \setminus \mathcal B; \Z/d\Z))^*$) via the intersection pairing; we will abuse notation and write $\balpha \in H_1(X, \cB; \bZ/d\bZ)$ when it is convenient to do so.

\para{Absolute and relative automorphism groups} As an abelian group, 
\[
H_1(X, \cB; \Z/d\Z) \cong (\Z/d\Z)^{2g+n-1},
\]
 so that $\overline{\Gamma}_d$ is a subgroup of $\GL_{2g+n-1}(\Z/d\Z)$. The underlying topological context leads $\overline{\Gamma}_d$ to preserve a great deal of additional structure. Ultimately, a precise determination of $\overline{\Gamma}_d$ will constitute one of the major results of the paper, appearing in \Cref{section:monononhyp,section:monohyp} in the non-hyperelliptic and hyperelliptic settings, respectively. Here, we begin the process of describing $\overline{\Gamma}_d$ by taking relative and absolute homology into account. We follow the discussion of \cite[Section 4.2]{CalderonSalterRelHomRepsFramedMCG}. 

Let $(X, \mathcal B)$ be a closed surface of genus $g$ endowed with a set $\mathcal B = \{p_1, \dots, p_n\} \subset X$ of distinguished points. For any $d \ge 0$, the long exact sequence of this pair specializes to the short exact sequence
\[
    0 \to H_1(X;\bZ/d\Z) \to H_1(X, \mathcal B; \bZ/d\Z) \xrightarrow{\delta} \widetilde H_0(\mathcal B; \bZ/d\Z) \to 0.
\]
The connecting map 
\begin{equation}\label{eqn:delta}
\delta: H_1(X, \mathcal B; \bZ/d\Z) \to \widetilde H_0(\mathcal B; \bZ/d\Z)
\end{equation}
is an important invariant, so we recall its basic features. Viewing an element of $H_1(X, \mathcal B; \bZ/d\bZ)$ as a relative homology class, represented by a formal sum of arcs, $\delta$ records the endpoints (with sign). Under the identification $H_1(X, \mathcal B; \Z/d\Z) \cong (H_1(X \setminus \mathcal B; \Z/d\Z))^*$ discussed above, $\delta$ is given by the {\em restriction} of $\balpha \in (H^1(X \setminus \mathcal B; \Z/d\Z))^*$ to the subspace $\ker(i_*) \le H_1(X \setminus \mathcal B; \Z/d\Z)$, with $i: X \setminus \mathcal B \to X$ the inclusion map. $\ker(i_*)$ is spanned by the classes of loops $[\gamma_i]$ encircling the $i^{th}$ puncture. Given a branching vector $\balpha \in H^1(X \setminus \mathcal B; \Z/d\Z)$, it follows that $\delta(\balpha)$ can be viewed as a tuple $(d_1, \dots, d_n) \in (\Z/d\Z)^n$ that records the local monodromies around each of the elements of $\mathcal B$. This explains the terminology $\gcd(\delta(\balpha))$ introduced in \Cref{definition:holonomy}.

The action of the pure mapping class group $\PMod(X, \mathcal B)$ on $H_1(X, \mathcal B; \bZ/d\Z)$ preserves the subspace $H_1(X; \bZ/d\Z)$ and the restriction preserves the  intersection pairing $\pair{\cdot, \cdot}$. Define
\[
\PRelAut(H_1(X, \mathcal B; \bZ/d\Z)) := \Hom(\widetilde H_0(\mathcal B; \bZ/d\Z), H_1(X; \bZ/d\Z)).
\]
The group $\PAut(H_1(X, \mathcal B;\bZ/d\Z))$ is then characterized by the following short exact sequence:
\begin{equation}\label{eqn:pautses}
1 \to \PRelAut(H_1(X, \mathcal B; \bZ/d\Z)) \to \PAut(H_1(X, \mathcal B;\bZ/d\Z)) \to \Sp(2g, \bZ/d\Z) \to 1,
\end{equation}
where, here and throughout, $\Sp(2g, \Z/d\Z)$ denotes the {\em symplectic group} of automorphisms of $H_1(X; \Z/d\Z)$ preserving the intersection pairing $\pair{\cdot, \cdot}$. By the above discussion, 
\[
\overline{\Gamma}_d \le \PAut(H_1(X, \mathcal B;\bZ/d\Z)).
\]

In the sequel, we will frequently fix a splitting
\[
H_1(X, \cB; \Z/d\Z) \cong H_1(X; \Z/d\Z) \oplus \widetilde H_0(\cB; \Z/d\Z),
\]
e.g. via a ``geometric splitting'' discussed in \Cref{subsection:geomsplit2}. In such coordinates, elements of $\PAut(H_1(X, \mathcal B;\bZ/d\Z))$ are upper-triangular, and will be written in the form
\[
\MAOI,
\]
with $M \in \Sp(2g, \Z/d\Z)$ and $A \in \Hom(\widetilde H_0(\mathcal B; \bZ/d\Z), H_1(X; \bZ/d\Z))$; the lower-right block is the identity by the hypothesis that $\cB$ is marked pointwise.

\subsection{Orbits of branching vectors}
\label{OrbitsBranchVect:Section}
In general, if $p: (\widetilde X, \widetilde \omega) \rightarrow (X, \omega)$ is a branched covering of translation surfaces specified by the branching vector $\balpha$, then $p$ induces a covering map 
\[
\cP: \cM_\delta(\balpha) \rightarrow \cH_{1, lab},
\]
where we recall from the introduction that $\cM_{\delta}$ is the set of all translation covers of surfaces in $\cH_{1,lab}$ with local monodromy given by $\delta$, and we define $\cM_\delta(\balpha)$ to be the component of $\cM_{\delta}$ containing $(\widetilde X, \widetilde \omega)$.
The locus $\cM_\delta(\balpha)$ can equivalently be defined by taking the $\splin$-orbit closure of $(\widetilde X, \widetilde \omega)$ when $(X, \omega)$ is generic in $\cH_{1, lab}$ with respect to the Masur--Veech measure.

The degrees of $p$ and $\cP$ rarely coincide.  From elementary topology we have the following:

\begin{proposition}
\label{CovDeg}
Let $p: (\widetilde X, \widetilde \omega) \rightarrow (X, \omega)$ be a branched cyclic translation covering of degree $d$ with branching vector $\balpha$ and set $\cP: \cM_\delta(\balpha) \rightarrow \cH_{1, lab}$ to be the induced covering map of moduli spaces. Then 
$\deg(\cP) = |\overline{\Gamma}_d \balpha|.$
\end{proposition}

\subsection{Orbits with fixed monodromy}
\label{FixedHomClassSection}
We turn now to the second set of relevance.
To compute Siegel--Veech constants, it is necessary to estimate the volume of the locus in $\cM_\delta (\balpha)$ 
consisting of surfaces that have thin cylinders (see Section \ref{subsec:SVandvols}).
In order to accomplish this, the fiber in $\cM_\delta(\balpha)$ over the point $(X, \omega) \in \cH_{1, lab}$ must be decomposed according to the sizes of the periods of the thin cylinders in each cover, which are necessarily multiples of the sizes of the periods of the corresponding cylinders on the base surface $(X, \omega)$.

As we saw in \Cref{lemma:cyllifts} above, the period of a cylinder on $(\widetilde X, \widetilde \omega)$ is determined by the branching vector classifying the cover $(\widetilde X, \widetilde \omega) \rightarrow (X, \omega)$.  Thus, for a given cylinder $C \in (X, \omega)$ with core curve $c$, we are interested in the set
\begin{equation}\label{equation:bargammasuba}
(\bar \Gamma_{d} \balpha )_{c;a} := \{\bbeta \in \bar \Gamma_{d} \balpha \mid \bbeta(c) = a \pmod{d}\}.
\end{equation}
Clearly,
\[\bigcup_{a = 1}^d (\bar \Gamma_{d} \balpha )_{c;a}
= \overline{\Gamma}_{d} \balpha.\]

Ultimately, we will see that the cardinality of the set $(\bar \Gamma_{d} \balpha )_{c;a}$ is independent of the choice of cylinder $C$.
We will see this in two ways in the non-hyperelliptic setting.
By direct inspection, the counting formulas \Cref{prop:oddpstats,prop:2splitstats} for $\abs{(\bar \Gamma_{d} \balpha )_{c;a}}$ are independent of $c$.
In \Cref{section:cyltrans}, we present a conceptual argument based on the results \Cref{prop:nonhypcyltrans,prop:hypcyltrans} which establish the transitivity of the monodromy group action on excision homology classes of cylinders. While this second proof is not logically necessary in the non-hyperelliptic case (the transitivity having been taken care of via the action of the symplectic group), we believe that the added conceptual clarity, as well as the potential independent interest of \Cref{prop:nonhypcyltrans,prop:hypcyltrans}, merit inclusion in an appendix. In keeping with this, in the sequel, we will write
\[
\abs{(\bar \Gamma_{d} \balpha )_{a}}
\]
in place of the more precise $\abs{(\bar \Gamma_{d} \balpha )_{c;a}}$.

\begin{remark}
A related concept called the ``cusp width'' of a locus of torus covers appears in both \cite[$\S$10]{EskinKontsevichZorich2} and \cite[Proof of Theorem 3.1]{ChenMollerZagierQuasiModSiegelVeech}.
In our setting, this corresponds to the order of $a$, and in computing Siegel--Veech constants we will group covers by $a$ because they will contribute the same integrals. See \eqref{eqn:pulloutastast}.
\end{remark}

\part{Structure and cardinality of orbits}
\label{OrbitCard:Part}

\para{Outline of Part~\ref{OrbitCard:Part} (Sections \ref{section:unsplitorbits}--\ref{section:splithyp})} The purpose of this Part is to classify and enumerate orbits of branched covers under the monodromy group of a stratum, and to understand the lifting properties of cylinders along these covers. Topologically, this corresponds to the problem of classifying the components of $\cM_\delta$. The ultimate objective is to determine the cardinalities of the sets $\overline{\Gamma}_d \balpha$ and $(\overline{\Gamma}_d \balpha)_{c;a}$ defined above in Sections~\ref{OrbitsBranchVect:Section} and \ref{FixedHomClassSection}. This will be accomplished as follows:

\begin{itemize}
\item In \Cref{section:unsplitorbits}, we begin the project of determining $\abs{\overline{\Gamma}_d \balpha}$ and $\abs{(\overline{\Gamma}_d\balpha)_a}$ by formulating the notion of {\em split} and {\em unsplit} orbits. Counting orbits in the unsplit case follows from relatively general principles, and we carry out those computations here.
\end{itemize}
The remaining work is to determine which orbits are (un)split, and to compute the quantities $\abs{\overline{\Gamma}_d \balpha}$ and $\abs{(\overline{\Gamma}_d\balpha)_a}$ in the split case. We proceed as follows:
\begin{itemize}
\item In \Cref{section:wnf}, we discuss the theory of framings, winding number functions, and quadratic forms. These notions will be used in our description of $\overline{\Gamma}_d$ in the non-hyperelliptic setting in \Cref{section:monononhyp}. 
\item In \Cref{section:monononhyp}, we determine the monodromy group $\overline{\Gamma}_d$ in the case of a non-hyperelliptic stratum component.

\item In \Cref{section:psi}, we define the {\em $\psi$ invariant} used to classify split orbits in the non-hyperelliptic setting.
\item In \Cref{section:splitnonhyp}, we determine $\abs{\overline{\Gamma}_d \balpha}$ and $\abs{(\overline{\Gamma}_d\balpha)_a}$ for split orbits in the non-hyperelliptic setting.
\item In \Cref{section:monohyp}, we shift our attention to hyperelliptic components, and give a description of the monodromy group $\overline{\Gamma}_d$ in this setting.
\item In \Cref{section:splithyp}, we determine $\abs{\overline{\Gamma}_d \balpha}$ and $\abs{(\overline{\Gamma}_d\balpha)_a}$ for split orbits in the hyperelliptic setting.
\end{itemize}

In addition to these sections in the main body of the paper, there are two short appendices related to the material in Part \ref{OrbitCard:Part}. In \Cref{section:sympmodules}, we state and sketch the proofs of some lemmas about the symplectic group over $\Z/d\Z$. In \Cref{section:cyltrans}, we present a conceptual explanation of the phenomenon that the cardinality of $(\overline{\Gamma}_d \balpha)_{c;a}$ is independent of $c$, as discussed above in \Cref{FixedHomClassSection}.\\

\section{Orbit counts (I): the unsplit case}\label{section:unsplitorbits}

We now begin the project of determining the cardinalities of the orbits $\overline{\Gamma}_d \balpha$ and their subsets $(\overline{\Gamma}_d \balpha)_{c;a}$ as defined in \eqref{equation:bargammasuba}. To formulate the results, we recall that there are two numerical invariants associated to the elements of the distinguished set $\mathcal B$. The first of these is the integers $k_i \ge 0$, the orders of vanishing of the translation surface at $p_i \in \mathcal B$. The second is $d_i\in \Z/d\Z$, the local monodromy of the branching at $p_i$. Recall that this can be computed as $\balpha(\gamma_i) \in \Z/d\Z$, where $\balpha \in H_1(X, \mathcal B; \Z/d\Z)$ classifies the branched cover and $\gamma_i \subset X \setminus \mathcal B$ is a small loop encircling $p_i$ once counterclockwise (as the covering group $\Z/d\Z$ is abelian, this is well-defined independent of a global basepoint). Recall from \Cref{subsection:hommonodromy} that the vector of the elements $d_i$ can also be expressed as $\delta(\balpha)\in \tilde H_0(\mathcal B; \Z/d\Z)$, where $\delta$ is the connecting map for the long exact sequence in homology, as defined in \eqref{eqn:delta}.

\subsection{Statement of results}\label{stmtsplit}

Here we consider the simplest type of orbit: {\em unsplit orbits}, where the action of $\overline{\Gamma}_d$ is determined by arithmetic data alone. This notion is sensible in both the hyperelliptic and non-hyperelliptic settings, and in fact can be defined for an {\em arbitrary} subgroup $\Gamma \leq \PAut(H_1(X, \cB; \Z/d\Z))$, not just the monodromy group of a stratum. \textbf{Accordingly, in this section, we will work in this more general context, except where otherwise specified.}

\begin{definition}[Unsplit orbit]\label{definition:unsplit}
Let $\balpha \in H_1^{prim}(X, \mathcal B; \Z/d\Z)$, and let $\Gamma \leq \PAut(H_1(X, \cB; \Z/d\Z))$ be arbitrary. The orbit $\Gamma \balpha$ is {\em unsplit} if $\Gamma$ acts transitively on vectors of given $\delta$-value, i.e. if
\[
\Gamma \balpha = \{\bbeta \in H_1^{prim}(X, \mathcal B; \Z/d\Z)\mid \delta(\bbeta) = \delta(\balpha)\}.
\]
We say $\balpha$ is unsplit if its orbit is. 
\end{definition}

Let $\PAut(H_1(X, \mathcal B; \Z/d\Z))[2]$ denote the ``level-$2$ subgroup'', i.e., the kernel of the reduction map 
\[
\PAut(H_1(X, \mathcal B; \Z/d\Z)) \to \PAut(H_1(X, \mathcal B; \Z/2\Z))
\]
(where the reduction map is trivial if $d$ is odd). In the sequel (\Cref{corollary:level2nonhyp}, \Cref{corollary:level2hyp}), we will see that in both the non-hyperelliptic and hyperelliptic settings, the monodromy group $\Gamma = \overline{\Gamma}_d$ contains the level-$2$ subgroup. Here, we see that for any $\Gamma$ containing the level-$2$ subgroup, the question of whether an orbit splits is determined exclusively by mod-$2$ data.

\begin{proposition}[Orbits are determined mod $2$] \label{prop:oddporbits} Let $\balpha, \bbeta \in H_1^{prim}(X, \mathcal B; \Z/d\Z)$ be given, and suppose that $\delta(\balpha) = \delta(\bbeta)$. Let $\Gamma \le \PAut(H_1(X, \cB; \Z/d\Z))$ be a subgroup that contains the level-$2$ subgroup $\PAut(H_1(X, \mathcal B; \Z/d\Z))[2]$. Then $\Gamma \balpha = \Gamma \bbeta$ if and only if $\bar \Gamma_2 \bar \balpha = \bar \Gamma_2 \bar \bbeta$, where $\bar \balpha, \bar \bbeta$ are the reductions mod $2$ of $\balpha, \bbeta$, and $\bar \Gamma_2$ denotes the image of $\Gamma$ under the reduction map 
\[
\PAut(H_1(X, \mathcal B; \Z/d\Z)) \to \PAut(H_1(X, \mathcal B; \Z/2\Z))).
\]
In particular, this is trivially satisfied if $d$ is odd, and so in this case, every orbit is unsplit.
\end{proposition}

A key component of \Cref{prop:oddporbits} is the assertion that the orbits obey a version of the Chinese remainder theorem. We extract this assertion and record it here as a corollary. It will be essential in the sequel, where we will work one prime at a time.

\begin{corollary}[CRT for orbits]\label{theorem:CRT}
Let $\balpha \in H_1(X, \mathcal B; \Z/d\Z)$ be given, and let $\Gamma \leq \PAut(H_1(X, \cB; \Z/d\Z))$ be a subgroup containing $\PAut(H_1(X, \mathcal B; \Z/d\Z))[2]$. Then under the decomposition
\[
H_1(X,\mathcal B;\Z/d\Z) \cong \bigoplus_{p_i \mid d} H_1(X, \mathcal B;\Z/p_i^{k_i}\Z)
\]
of the Chinese remainder theorem, there is an isomorphism
\[
\Gamma \balpha \cong \prod_{p_i \mid d} \bar \Gamma_{p_i^{k_i}} \bar \balpha,
\]
where $\bar \balpha$ denotes the reduction of $\balpha$ mod $p_i^{k_i}$, and likewise $\bar \Gamma_{p_i^{k_i}}$ is the reduction of $\Gamma$ mod $p_i^{k_i}$.
\end{corollary}

In the numerical formulas to be obtained below, the {\em Jordan totient function} will play a key role. 

\begin{definition}[$\Phi_n(d)$]
The {\em Jordan totient function} $\Phi_n(d)$ is the number of primitive ($\gcd =1$) elements of $(\Z/d\Z)^n$. It is computed as
\[
\Phi_n(d) = d^n \prod_{p \mid d} \left( 1 - \frac{1}{p^n} \right).
\]
Note that the case $n = 1$ is the classical Euler totient function; in this case, we will write $\Phi$ in place of $\Phi_1$.
\end{definition}

\begin{proposition}[Orbit size, unsplit case]\label{prop:oddpsize}
Let $p$ be prime, and let  $\balpha \in H_1^{prim}(X,\mathcal B; \Z/p^k\Z)$ be unsplit for the action of $\Gamma \le \PAut(H_1(X, \cB; \Z/p^k\Z))$.  Then
    \[
    \abs{\Gamma \balpha} = \begin{cases}
                                    p^{2gk}          & \gcd(\delta(\balpha)) = 1 \mod p\\
                                    \Phi_{2g}(p^k)  & \gcd(\delta(\balpha)) \ne 1 \mod p.
    \end{cases}
    \]
\end{proposition}

In the case of monodromy groups of translation surfaces, we will be interested in a finer-scale count than just the size of the orbit.

\begin{proposition}[Orbit statistics, unsplit case]\label{prop:oddpstats}
  Let $(X,\omega)$ be a translation surface in an arbitrary component of a stratum in genus $g \ge 1$. Let $p$ be prime, and let  $\balpha \in H_1^{prim}(X,\mathcal B; \Z/p^k\Z)$ be unsplit for the action of its monodromy group $\bar \Gamma_{p^k} \le \PAut(H_1(X, \cB; \Z/p^k\Z))$. Let $c \in H_1^{prim}(X \setminus \mathcal B; \Z/p^k\Z)$ be the homology class of a cylinder on $X$, and let $a \in \Z/p^k\Z$ be given. Then
    \[
|(\bar \Gamma_{p^k} \balpha )_{c;a}| = \begin{cases}
    p^{k(2g-1)} & \gcd(\delta(\balpha)) = 1\mbox{ or } a \ne 0 \pmod p\\
    \Phi_{2g-1}(p^k) & \gcd(\delta(\balpha)) \ne 1 \mbox{ and }a = 0 \pmod p.
\end{cases}
\]
\end{proposition}
Note that this formula is independent of $c$, as was asserted in Section \ref{FixedHomClassSection}.

\subsection{Proofs}\label{subsection:unsplitproofs}

\subsubsection{Orbit classification}
\begin{proof}[Proof of \Cref{prop:oddporbits}]
We first establish a version of the Chinese remainder theorem in our setting. Under the decomposition
\begin{equation}\label{eqn:crt}
H_1(X,\mathcal B;\Z/d\Z) \cong \bigoplus_{p_i \mid d} H_1(X, \mathcal B;\Z/p_i^{k_i}\Z)
\end{equation}
of the Chinese remainder theorem, we claim there is an isomorphism 
\begin{equation}\label{eqn:groupcrt}
\Gamma \balpha \cong \prod_{p_i \mid d} \bar \Gamma_{p_i^{k_i}} \bar \balpha,
\end{equation}
where $\bar \balpha$ denotes the reduction of $\balpha$ mod $p_i^{k_i}$. By reducing modulo each $p_i^{k_i}$, there is an injective map from the left-hand side of \eqref{eqn:groupcrt} to the right; it remains to see that this is surjective. Under the isomorphism \eqref{eqn:crt}, a given $(\bar \bbeta_{p_1}, \dots, \bar \bbeta_{p_n}) \in \prod \bar \Gamma_{p_i^{k_i}} \bar \balpha$ corresponds to some $\bbeta \in H_1(X, \mathcal B; \Z/d\Z)$; we seek to show that $\bbeta \in \Gamma\balpha$. 

By hypothesis, $\Gamma$ contains the level-$2$ subgroup $\PAut(H_1(X, \mathcal B; \Z/d\Z))[2]$. Thus, under the Chinese remainder isomorphism $\Z/d\Z \cong \prod_{p_i \mid d} \Z/p_i^{k_i} \Z$, there is a containment
\begin{equation}\label{eqn:crtgroup}
\PAut(H_1(X, \mathcal B; \Z/2^{k}\Z))[2] \times \prod_{p_i \mid d\mbox{ \small{odd}}}\PAut(H_1(X, \mathcal B; \Z/p_i^{k_i}\Z)) \le \Gamma.
\end{equation}
In particular, this shows that the reduction of $\balpha$ mod any $p_i$ can be altered via one of the factors in \eqref{eqn:crtgroup} while leaving the reductions mod $p_j$ for $p_i \ne p_j$ unchanged.

By elementary algebra, for odd $p_i$, the orbits of $\PAut(H_1(X, \mathcal B; \Z/p_i^{k_i}\Z))$ acting on $H_1(X, \mathcal B; \Z/p_i^{k_i}\Z)$ are classified by gcd and the value $\delta(\balpha) \in \tilde H_0(\mathcal B; \Z/p_i^{k_i}\Z)$. Using \eqref{eqn:crtgroup}, we can send $\balpha$ via $\Gamma$ to some $\bbeta'$ that satisfies $\bbeta' = \bbeta \pmod{p_i^{k_i}}$ for all odd primes $p_i \mid d$.\\

To complete the argument, we must show that if $\bar \Gamma_2 \bar \balpha = \bar \Gamma_2 \bar \bbeta$, then $\bar \Gamma_{2^k} \balpha = \bar \Gamma_{2^k} \bbeta$. Define 
\[
\delta := \delta(\balpha) = \delta(\bbeta).
\]
By hypothesis, there is $F \in \Gamma$ such that $F\balpha = \bbeta'$ with $\bbeta = \bbeta' \pmod 2$, and with $\delta(\bbeta') = \delta$. Relative to a splitting of $H_1(X, \mathcal B; \Z/2^k\Z)$, write $\bbeta = x + \delta$ and $\bbeta' = x' + \delta$. 

The projection $H_1(X, \mathcal B; \Z/2^k\Z) \to H_1(X; \Z/2^k\Z)$ associated to such a splitting induces a surjection
\[
\PAut(H_1(X, \mathcal B; \Z/2^k\Z))[2] \to \Sp(2g, \Z/2^k\Z)[2].
\]
According to \Cref{prop:symporbits}.\ref{item:sp2}, the group $\Sp(2g, \Z/2^k\Z)[2]$ acts transitively on elements of $H_1(X;\Z/2^k\Z)$ with the same gcd and mod-$2$ reduction. Thus, {\em if} $\gcd(x) = \gcd(x')$, there exists $M \in \Sp(2g, \Z/2^k\Z)[2]$ such that $Mx = x'$. The element $\begin{pmatrix}M&0\\0&I\end{pmatrix}$ is contained in $\PAut(H_1(X, \mathcal B; \Z/2^k\Z))[2]$ and takes $\bbeta$ to $\bbeta'$, showing $\bar \Gamma_{2^k} \balpha = \bar \Gamma_{2^k} \bbeta' = \bar \Gamma_{2^k} \bbeta$ as desired.

In general, however, $\gcd(x) \ne \gcd(x')$ - writing
\[
p = \gcd(x), \qquad p' = \gcd(x'), \qquad q = \gcd(\delta),
\]
the hypothesis $\gcd(\balpha) = \gcd(\bbeta)$ only implies that 
\[
\gcd(p,q) = \gcd(p', q) = 1.
\]
We will show that there exists $G \in \PAut(H_1(X, \mathcal B; \Z/2^k\Z))[2]$ such that $G\bbeta = x'' + \delta$ satisfies $\gcd(x'') =\gcd(p,2)$ and such that $x'' = x \pmod 2$. Note that $\gcd(p,2) = \gcd(p',2)$, so that applying this construction to both $\bbeta$ and $\bbeta'$ reduces to the case above.

Write $x = p v$ with $v \in H_1(X; \Z/2^k\Z)$ primitive, and let $w \in H_1(X; \Z/2^k\Z)$ be a primitive vector linearly independent from $v$. Let $A \in \Hom(\tilde H_0(\mathcal B;\Z/2^k\Z), H_1(X;\Z/2^k\Z))$ take the primitive vector $\delta/q \in \tilde H_0(\mathcal B; \Z/2^k\Z)$ to $w$. Then the element $G = \begin{pmatrix} I & 2A \\ 0 & I\end{pmatrix}$ has trivial mod-$2$ reduction, and 
\[
G\bbeta = (x +2qw) + \delta = (pv + 2q w) + \delta.
\]
Thus $x'' = pv + 2qw$ satisfies $\gcd(x'') = \gcd(p, 2)$ (since $\gcd(p, 2q) = \gcd(p, 2)$ and $v,w$ are linearly independent) and $x'' = x \pmod 2$ as required.
\end{proof}

\subsubsection{Orbit sizes}
\begin{proof}[Proof of \Cref{prop:oddpsize}]
By \Cref{prop:oddporbits}, $\bar \Gamma_{p^k} \balpha$ consists of all $\bbeta \in H_1^{prim}(X,\mathcal B; \Z/p^k\Z)$ such that $\delta(\bbeta) = \delta(\balpha):= \delta$. In the coordinates induced by the geometric splitting, $\bbeta = y + \delta$, with $y$ subject to the constraint that $\gcd(y + \delta) = 1$. If $\gcd(\delta) = 1$, then $y$ is unconstrained, showing $\abs{\bar \Gamma_{p^k} \balpha} = (p^k)^{2g}$ as claimed. Otherwise, if $\gcd(\delta) \ne 1$, then $y \in H_1(X; \Z/p^k \Z)$ must be primitive, showing $\abs{\bar \Gamma_{p^k} \balpha} = \Phi_{2g}(p^k)$.
\end{proof}

\subsubsection{Orbit statistics}
\begin{proof}[Proof of \Cref{prop:oddpstats}]
Since $\balpha$ is unsplit, $\bar \Gamma_{p^{k}} \balpha$ consists of all $\bbeta \in H_1^{prim}(X,\mathcal B; \Z/p^{k}\Z)$ such that $\delta( \bbeta) = \delta(\balpha):= \delta$. If $\gcd(\delta) = 1$, then every vector of the form $\bbeta = y + \delta$ is primitive, and so
\[
\bar \Gamma_{p^{k}} \balpha = \{y + \delta \mid y \in H_1(X; \Z/p^{k}\Z)\}.
\]
As $c\in H_1(X\setminus \mathcal B; \Z/p^k\Z)$ is the homology class of a cylinder, it is in particular the class of a simple closed curve, and so it is primitive. It follows that each $a \in \Z/p^{k}\Z$ is realized as the value $\bbeta( c)$ an equal number of times as $ \bbeta = y + \delta$ ranges over all $y \in H_1(X; \Z/p^{k}\Z)$. There are $p^{k}$ distinct values and $(p^{k})^{2g}$ elements of $H_1(X; \Z/p^{k}\Z)$, showing
\[
\abs{(\bar \Gamma_{p^k} \balpha)_{c;a}} := \abs{\{\bbeta\in\bar \Gamma_{p^{k}} \balpha \mid \bbeta( c) = a \pmod{p^{k}}\}} = p^{k(2g-1)}
\]
as claimed.

We next consider the other case that $\gcd(\delta) \ne 1$. In this case, the $\gcd$ condition is not automatically satisfied, so that 
\[
\bar \Gamma_{p^{k}} \balpha = \{y +  \delta \mid y \in H_1^{prim}(X; \Z/p^{k}\Z)\}.
\]
To count the elements of the orbit assuming the value $a \in \Z/p^{k}\Z$, we observe that $y \in H_1(X; \Z/p^{k}\Z)$ is primitive if and only if its projection to $H_1(X; \Z/p\Z)$ is nonzero. Thus,
\[
\abs{(\bar \Gamma_{p^k} \balpha)_{c;a}} = (p^{k-1})^{(2g-1)} \abs{\{y +  \delta \mid y\ne 0 \in H_1(X; \Z/p\Z),\  \bbeta( c) = a \pmod{p}\}}.
\]
If $a \ne 0 \pmod{p}$, then the condition $ \bbeta( c) = a \pmod{p}$ ensures that $y \ne 0$. The set of such $y$ is an affine hyperplane in $H_1(X; \Z/p\Z)$, necessarily of cardinality $p^{2g-1}$. Putting these calculations together,
\[
\abs{(\bar \Gamma_{p^k} \balpha)_{c;a}} = (p^{k-1})^{(2g-1)}p^{2g-1} = p^{k(2g-1)}
\]
as claimed.

If instead $a = 0 \pmod{p}$, the set of such {\em nonzero} vectors $y$ corresponds to the set of nonzero vectors in the hyperplane determined by the condition $ \bbeta(  c) = 0$, a set of size $p^{2g-1}-1$. In this case,
\[
\abs{(\bar \Gamma_{p^k} \balpha)_{c;a}} = (p^{k-1})^{(2g-1)}(p^{2g-1}-1) = \Phi_{2g-1}(p^k)
\]
as claimed.
\end{proof}

\section{Framings, winding number functions, and quadratic forms}\label{section:wnf}

\para{Basic principles} The object of central importance in Part \ref{OrbitCard:Part} is the monodromy group $\overline{\Gamma}_d$ of a stratum component acting on $H_1(X, \cB; \Z/d\Z)$. It turns out that in the non-hyperelliptic setting, this admits a description as the kernel of a certain {\em crossed homomorphism} $\Theta_{q,\kappa}$. Recall that a {\em crossed homomorphism} $f: G \to A$ from a group $G$ to a $G$-module $A$ is a function satisfying 
\[
f(gh) = f(g) + g \cdot f(h).
\]
  At root, the constraints on the action of $\overline{\Gamma}_d$ on $H_1(X, \cB; \Z/d\Z)$ all stem from {\em topological} constraints imposed by the fact that translation surfaces are endowed with canonical horizontal and vertical foliations. The invariance of such structures imposes subtle homological constraints on $\overline{\Gamma}_d$, as expressed by $\Theta_{q,\kappa}$. The goal of this section is to carefully describe these constraints.
  
\subsection{Framings, winding number functions, and quadratic forms}

Recall that a {\em framing} of a smooth manifold $X$ is a trivialization of the tangent bundle $TX$; equivalently, a framing is specified by $n = \dim(X)$ everywhere linearly-independent vector fields. \Cref{lemma:basicsofframings} (which is obvious) recalls how framings emerge in the context of translation surfaces.

\begin{lemma}\label{lemma:basicsofframings}
Let $(X,\omega) $ be a translation surface and let $\cB \subset X$ be a set of distinguished points containing all zeros of $\Omega$. Then $\Re(\omega), \operatorname{Im}(\omega)$ form a pair of everywhere linearly independent $1$-forms on $X \setminus \cB$. Endowing $X$ with a Riemannian metric, these dually determine a ``translation surface framing'' of $X\setminus\cB$.
\end{lemma}

It turns out that {\em isotopy classes} of framings are classified by a certain discrete invariant known as a {\em winding number function}. 

\begin{definition}[Winding number function]
Let $\Sigma$ be a surface equipped with a framing $\phi$. Let $\mathcal S^+(\Sigma)$ denote the set of isotopy classes of oriented simple closed curves on $\Sigma$. The framing $\phi$ determines a {\em winding number function} which by abuse of notation we write
\[
\phi: \mathcal S^+(\Sigma) \to \Z,
\]
with $\phi(c) \in \Z$ equal to the winding number of the forward-pointing tangent vector of $c$ relative to the framing. 
\end{definition}   

It should be clear that isotopic framings determine the same winding number function. See \cite[Section 2]{strata3} or \cite{RW} for a more thorough discussion of the relationship between these two notions, including the sense in which the converse holds.

In \cite[Lemma 7.10]{strata3}, it is shown that the presence of the translation surface framing imposes constraints on the mapping class group-valued monodromy of a stratum component. However, it is not {\em a priori} clear what kinds of constraints the framing imposes on the {\em homological} monodromy of interest here. This is because the associated winding number function $\phi$ is {\em not homological}, in the sense that $\phi(a)$ for a simple closed curve $a \subset \Sigma$ is not determined by $[a] \in H_1(\Sigma;\Z)$. It turns out that the reduction mod $2$ of $\phi$ {\em is} homological, but defining this correctly takes some care. To begin, recall the notion of a {\em $\Z/2\Z$-quadratic form}.

\begin{definition}[$\Z/2\Z$-quadratic form]
Let $V = (\Z/2\Z)^{2g}$ be a vector space over $\Z/2\Z$ equipped with a symplectic form $\pair{\cdot, \cdot}$, i.e., a nondegenerate bilinear form satisfying $\pair{x,x} = 0$ for all $x \in V$. A {\em $\Z/2\Z$-quadratic form refining $\pair{\cdot, \cdot}$} is a function $q: V \to \Z/2\Z$ that satisfies
\[
q(x+y) = q(x) + q(y) + \pair{x,y}
\]
for all $x,y \in V$.
\end{definition}

\Cref{lemma:induced} gives a first indication of the homological character of $\phi$.
\begin{lemma}\label{lemma:induced}
Let $S$ be a framed surface. Suppose that either $S$ has a single boundary component, or else $S = X \setminus \mathcal B$ is obtained from a translation surface for which every zero has even order. Then the function
\[
q: H_1(S; \Z/2\Z) \to \Z/2\Z
\]
given by
\[
q(x) = \phi(\xi) + 1 \pmod 2,
\]
where $\xi \subset S$ is a simple closed curve with $[\xi] = x$, is a $\Z/2\Z$-quadratic form refining the intersection pairing. Moreover, $q$ descends to a $\Z/2\Z$-quadratic form on $H_1(\bar S; \Z/2\Z)$, where $\bar S$ is the closed surface obtained from $S$ by filling in all boundary components and punctures.
\end{lemma}
\begin{proof}
See \cite[Remark 4.3]{CalderonSalterRelHomRepsFramedMCG}.
\end{proof}    
It turns out that even when the translation surface $(X, \omega)$ has zeros of odd order, the mod-$2$ winding number of a simple closed curve is still homological (determined in general by the class in {\em excision} homology $H_1(X\setminus \cB; \Z/2\Z)$, see \Cref{lemma:qformula}). However, to formulate this precisely, we require the notion of a {\em geometric splitting}, discussed in the next subsection.

\subsection{Geometric splittings}\label{subsection:geomsplit2}
Consider the pair $(X, \mathcal B)$ of topological spaces. For any coefficient ring $R$, the associated long exact sequence in homology degenerates to the short exact sequence
\[
0 \rightarrow H_1(X; R) \rightarrow H_1(X, \mathcal B; R) \xrightarrow{\delta} \tilde H_0(\mathcal B; R) \rightarrow 0.
\]

This gives a {\em noncanonical} decomposition of $H_1(X, \mathcal B; \Z/d\Z)$ into its ``relative'' and ``absolute'' parts.  The purpose of a ``geometric splitting'' is to give a geometric meaning to this decomposition.

\begin{definition}[Geometric splitting]\label{definition:geomsplit}
A {\em geometric splitting} of the pair $(X, \mathcal B)$ is a choice of subsurface $X^\circ \subset X \setminus \mathcal B$, such that the complement $X \setminus X^\circ = \Delta$ is a disk that contains all the points of $\mathcal B$. For any coefficient ring $R$, a geometric splitting induces a splitting
\[
H_1(X, \mathcal B; R) \cong H_1(X; R) \oplus \tilde H_0(\mathcal B; R)
\]
via the condition that classes in $\tilde H_0(\mathcal B; R)$ be supported on $\Delta$; then also $H_1(X;R)$ is canonically identified with $H_1(X^\circ; R)$.  Let $i_{\Delta,*}: \tilde H_0(\mathcal B; R) \to H_1(X, \cB; R)$ denote the inclusion map associated to the geometric splitting.
\end{definition}

\begin{definition}[Associated data of a geometric splitting]\label{lemma:discussion}
Let $(X, \omega)$ be a translation surface in the stratum $\cH(\kappa)$, with $\kappa = (k_1, \dots, k_n)$; let $\phi$ be the associated framing of $X \setminus \cB$. 
Let $X^\circ \subset X$ be a geometric splitting. The {\em associated data} of $X^\circ$ is the pair $(q,\tilde{K})$, where $q$ is the quadratic form on $H_1(X^\circ; \Z/2\Z)$ induced from the restriction of $\phi$ to $X^\circ$ in the sense of \Cref{lemma:induced}, and
\[
\tilde{K} = i_{\Delta,*}\left(\sum_{p_i \in \mathcal B} k_i p_i\right) \in i_{\Delta,*}(\tilde H_0(\mathcal B; \Z/2\Z)) \subset H_1(X, \mathcal B;\Z/2\Z);
\]
note that $\sum k_i = 2g-2 = 0 \pmod 2$ and so the expression defining $\tilde{K}$ is sensible. 
\end{definition}

The lemma below shows how the mod $2$ winding number of any simple closed curve can be computed from the associated data of any geometric splitting, thereby refining and extending \Cref{lemma:induced} (which is the case where $\kappa = 0$).

\begin{lemma}\label{lemma:qformula}
Let $a \subset X \setminus \mathcal B$ be a simple closed curve, and let $X^\circ \subset X \setminus \mathcal B$ be a geometric splitting with associated data $(q,\tilde{K})$. Then
\[
\phi(a) = q(\bar a) +1 + \pair{a, \tilde{K}} \pmod 2,
\]
where we view $\bar a$ in the first term on the right as the induced class in $H_1(X; \Z/2\Z)\cong H_1(X^\circ; \Z/2\Z)$ (with the isomorphism furnished by the geometric splitting), and in the expression $\pair{a, \tilde{K}}$, we view $a$ as an element of $H_1(X \setminus \mathcal B;\Z/2\Z)$.
\end{lemma}

\begin{proof}
Let $\Delta = X \setminus X^\circ$ be the disk associated to the geometric splitting. Choose a collection of embedded arcs on $\Delta$ representing $\tilde{K}$, and then choose a representative for $a$ meeting $\partial \Delta$ and $\tilde{K}$ transversely. To compute $q(\bar a)$, we must determine a representative simple closed curve that is homologous to $a$ in $H_1(X;\Z/2\Z)$ and supported on $X^\circ$. 

We do so by systematically altering $a$. By construction, $a \cap \Delta$ is a finite collection of disjoint arcs. Choose an {\em outermost} such arc $a_i$, in the sense that one component $U$ of $\Delta \setminus a_i$ contains no other arcs. Isotope $a$ (inside $X$) by pulling $a_i$ across the ``empty'' region $U$. The result is a simple closed curve $a'$ that is homologous to $a$ on $H_1(X;\Z/2\Z)$, but which will in general have different winding number. By the Poincar\'e--Hopf theorem, the difference $\phi(a') - \phi(a)$ is equal to the total order of the zeros associated to the points of $\mathcal B$ contained in $U$. Thus $\phi(a') = \phi(a) \pmod 2$ if and only if the total multiplicity of zeros in $U$ is even. By the definition of $\tilde{K}$, the arc $a_i$ intersects $\tilde{K}$ an even number of times if and only if this multiplicity is even. We conclude that
\[
\phi(a') + \pair{a', \tilde{K}} = \phi(a) + \pair{a, \tilde{K}} \pmod 2.
\]
After performing this procedure on all of the arcs of $a \cap \Delta$, we obtain a curve $a''$ supported on $X^\circ$ for which $\phi(a'') + \pair{a'', \tilde{K}} = \phi(a) + \pair{a, \tilde{K}} \pmod 2$. By the definition of $q(\bar a)$, and since $\pair{a'', \tilde{K}} = 0$ by construction,
\[
q(\bar a) = q(\bar a) + \pair{a'',\tilde{K}} = \phi(a'') + 1 + \pair{a'',\tilde{K}} = \phi(a) + 1 + \pair{a,\tilde{K}} \pmod 2
\]
as was to be shown.
\end{proof}

\subsection{The crossed homomorphism $\Theta_{q,\tilde{K}}$} \label{subsection:theta}

Here we define the crossed homomorphism $\Theta_{q,\tilde{K}}$. Our initial definition requires a choice of geometric splitting, and so we show in Lemma \ref{lemma:twocrossedhoms} that $\Theta_{q,\tilde{K}}$ coincides with a different crossed homomorphism $\Theta_\phi$ introduced in \cite{CalderonSalterRelHomRepsFramedMCG}. The latter does not depend on a geometric splitting, implying that the former is independent of this choice as well. It will not be necessary to give a full discussion of the original crossed homomorphism $\Theta_\phi$; for the sake of concision, we will refer to \cite{CalderonSalterRelHomRepsFramedMCG} at the few places this is required. 

\begin{definition}[The crossed homomorphism $\Theta_{q,\tilde{K}}$]\label{definition:theta}
Let
\[
H_1(X, \mathcal B;\Z/2\Z) \cong H_1(X; \Z/2\Z) \oplus i_{\Delta,*}(\tilde H_0(\mathcal B; \Z/2\Z))
\]
be a splitting of $H_1(X, \mathcal B;\Z/2\Z)$, and let $q$ be a quadratic form on $H_1(X; \Z/2\Z)$ and $\tilde{K} \in i_{\Delta,*}(\tilde H_0(\mathcal B; \Z/2\Z))$ be given. Define a function 
\[
\Theta_{q,\tilde{K}}: \PAut(H_1(X, \mathcal B;\Z/2\Z)) \to H^1(X; \Z/2\Z)
\]
via
\begin{equation}\label{eqn:thetaeval}
    \Theta_{q,\tilde{K}}\left(\MAOI\right)(x) = q(Mx) + q(x) + \pair{Mx, A\tilde{K}}.
\end{equation}

A routine verification shows that $\Theta_{q,\tilde{K}}$ is a crossed homomorphism under the obvious action of $\PAut(H_1(X,\cB;\Z/2\Z))$ on $H^1(X; \Z/2\Z)$.
\end{definition}

Below, we show the equivalence of $\Theta_{q,\tilde{K}}$ with the crossed homomorphism $\Theta_\phi$ studied in \cite{CalderonSalterRelHomRepsFramedMCG}, from which the invariance under the choice of geometric splitting will follow. The original $\Theta_\phi$ was defined at the level of the mapping class group, and so a formula like \eqref{eqn:thetaeval}, giving the value of the crossed homomorphism directly on an element of $\PAut(H_1(X,\cB;\Z/2\Z))$, was unavailable. Thus it is necessary to reformulate it in terms of $\Theta_{q,\tilde{K}}$, in order to carry out computations at the level of algebra and not the much more cumbersome setting of curves and arcs on framed surfaces. Unless the reader is interested in carefully reconciling the results of  \cite{CalderonSalterRelHomRepsFramedMCG} with those given here, the proof of \Cref{lemma:twocrossedhoms} can be safely skipped.

\begin{lemma}\label{lemma:twocrossedhoms}
            Let $X^\circ$ define a geometric splitting of $(X,\mathcal B)$ with associated data $(q,\tilde{K})$. Then there is an equality
            \[
            \Theta_{\phi} = \Theta_{q,\tilde{K}}.
            \]
            In particular, if $(q,\tilde{K})$ and $(q', \tilde{K}')$ are the associated data for geometric splittings $X^\circ$ and $X'^\circ$ of $X$, it follows that
            \[
            \Theta_{q,\tilde{K}} = \Theta_{q',\tilde{K}'}.
            \]
\end{lemma}

\begin{proof}
The group $\PAut(H_1(X, \mathcal B;\Z/2\Z))$ is a quotient of $\Mod(X,\mathcal B)$ via the latter's action on relative homology. As $\Mod(X,\mathcal B)$ is generated by Dehn twists, it suffices to verify that $\Theta_{\phi}(T_a) = \Theta_{q,\tilde{K}}(T_a)$ for any Dehn twist $T_a$.

Let $x \in H_1(X; \Z/2\Z)$ be given, and let $c \subset X \setminus \mathcal B$ be a representative simple closed curve. Then, following \cite[Section 4]{CalderonSalterRelHomRepsFramedMCG}, \cite[Lemma 2.4]{CalderonSalterRelHomRepsFramedMCG}, and \Cref{lemma:qformula} for the first, second and fourth equalities, respectively,
\begin{align*}
\Theta_\phi(T_a)(x) &= \phi(T_a(x)) - \phi(x)\\
 &= \phi(x) + \pair{a,x}\phi(a) - \phi(x)\\
 &= \pair{a,x}\phi(a)\\
 &= \pair{a,x}(q(a) + 1 + \pair{a,\tilde{K}});
\end{align*}
note that we abuse notation and use the symbol $a$ to denote a curve as well as the associated homology classes over $\Z$ and $\Z/2\Z$.

In order to compute $\Theta_{q,\tilde{K}}(T_a)$, we observe that the action on homology of $T_a$ is given by the transvection formula
\[
T_a(x) = x + \pair{a,x}a.
\]
Relative to the splitting 
\[
H_1(X, \mathcal B;\Z/2\Z) \cong H_1(X; \Z/2\Z) \oplus i_{\Delta,*}(\tilde H_0(\mathcal B; \Z/2\Z)),
\]
 we can write
\[
T_a = \MAOI,
\]
and we see that if $x \in H_1(X;\Z/2\Z)$, then $Mx = x + \pair{a,x}a$, but 
\[
A\tilde{K} = T_a(\tilde{K}) -\tilde{K} = \pair{a,\tilde{K}} a.
\]

Then (recalling we are working mod $2$, so that $\pair{a,x}^2 = \pair{a,x}$), 
\begin{align*}
    \Theta_{q,\tilde{K}}(T_a)(x) &= q(x+\pair{a,x}a) + q(x) + \pair{x+ \pair{a,x}a,\pair{a,\tilde{K}}a}\\
    &= \pair{a,x} q(a) + \pair{a,x}^2 +\pair{a,\tilde{K}}\pair{x,a}\\
    &= \pair{a,x}(q(\bar a) + 1 + \pair{a,\tilde{K}})\\
    &= \Theta_\phi(T_a)(x).
    \end{align*}
    
Since the definition of $\Theta_\phi$ (cf. \cite[Section 4]{CalderonSalterRelHomRepsFramedMCG}) does not involve a choice of geometric splitting, the final claim follows.    
\end{proof}

\section{The homological monodromy group, non-hyperelliptic case}\label{section:monononhyp}

Here we obtain the first major result of Part~\ref{OrbitCard:Part} -- \Cref{theorem:homologicalmonodromy}, which determines the monodromy action of a non-hyperelliptic stratum component on relative homology. \cite[Theorem B]{CalderonSalterRelHomRepsFramedMCG} obtains the same result in the range $g \ge 5$ via different methods; here we give a new and independent proof valid in the maximal range $g \ge 3$.  The methods of the proof of \Cref{theorem:homologicalmonodromy} will not recur in the rest of the paper, and the casual reader can safely take \Cref{theorem:homologicalmonodromy} as a black box.

\begin{theorem}\label{theorem:homologicalmonodromy}
Let $\cH \subset \cH_{lab}(\kappa)$ be a non-hyperelliptic component of a stratum of translation surfaces in genus $g \ge 3$. Choose a basepoint $(X, \omega) \in \cH$, and let $\mathcal B \subset X$ be the full set of distinguished points. Then the homological monodromy group $\bar \Gamma \le \PAut(H_1(X, \mathcal B;\bZ))$ is computed to be
\[
\bar \Gamma = \ker(\Theta_{q,\tilde{K}}),
\]
where $\Theta_{q,\tilde{K}}$ is the crossed homomorphism of \Cref{definition:theta}.
\end{theorem}

Recall that $\PAut(H_1(X, \mathcal B; \Z/d\Z))[2]$ denotes the kernel of the reduction map 
\[
\PAut(H_1(X, \mathcal B; \Z/d\Z)) \to \PAut(H_1(X, \mathcal B; \Z/2\Z))
\]
(where the reduction map is trivial if $d$ is odd). As an immediate corollary, we obtain the following result which will be used in the next section.
\begin{corollary}\label{corollary:level2nonhyp}
For any non-hyperelliptic stratum component and any integer $d$, there is a containment
\[
\PAut(H_1(X, \mathcal B; \Z/d\Z))[2] \le \overline{\Gamma}_d.
\]
\end{corollary}

\begin{proof}[Proof of \Cref{theorem:homologicalmonodromy}]
In \cite[proof of Proposition 5.2]{CalderonSalterRelHomRepsFramedMCG}, it is shown that $\bar \Gamma \leqslant \ker(\Theta_\phi) = \ker(\Theta_{q,\tilde{K}})$; this holds in the maximal range $g \ge 3$. To establish the opposite containment, we recall the decomposition
\[
1 \to \PRelAut(H_1(X, \mathcal B; \bZ)) \to \PAut(H_1(X, \mathcal B;\bZ)) \to \Sp(2g, \bZ) \to 1
\]
of \eqref{eqn:pautses}. We will show that the projections of $\bar \Gamma$ and $\ker(\Theta_{q,\tilde{K}})$ to $\Sp(2g, \bZ)$ coincide and that their intersections with the kernel $\PRelAut(H_1(X, \mathcal B; \bZ))$ are equal.

The monodromy action on {\em absolute} homology $H_1(X;\bZ)$ is determined by \cite[Theorem 6.8]{GutierrezRomoClassRauzyVeechGrps}. To state the result, we define
\[
r = \gcd(\kappa) = \gcd(k_1, \dots, k_n)
\]
as the gcd of the zero orders in the stratum. In the case of $r$ odd, the absolute monodromy is given as
\begin{equation}\label{equation:monodromyodd}
\bar \Gamma^{abs}= \Sp(2g, \bZ),
\end{equation}
and in the case of $r$ even, 
\begin{equation}\label{equation:monodromyeven}
\bar \Gamma^{abs} = \Sp(2g, \bZ)[q],
\end{equation}
where $q$ is the quadratic form on $H_1(X; \Z/2\Z)$ induced from $\phi$ via \Cref{lemma:induced}, and $\Sp(2g, \Z)[q]$ is the preimage in $\Sp(2g, \Z)$ of the stabilizer $\Sp(2g, \Z/2\Z)[q]$ of $q$ under the action of $\Sp(2g, \Z/2\Z)$ on the set of $\Z/2\Z$-quadratic forms on $H_1(X; \Z/2\Z)$. We must show that likewise, the projection of $\ker(\Theta_{q,\tilde{K}})$ to $\Sp(2g,\Z)$ is $\Sp(2g,\Z)[q]$. But this is clear: in the case that $r$ is even, the associated data $\tilde{K}$ is identically zero, so that in this case,
\[
\Theta_{q,\tilde{K}}\left(\MAOI\right)(x) = q(Mx) + q(x),
\]
showing that $\MAOI \in \ker(\Theta_{q,\tilde{K}})$ if and only if $M \in \Sp(2g, \Z)[q]$. 

It remains to consider the intersections of $\bar \Gamma$ and $\ker(\Theta_{q,\tilde{K}})$ with $\PRelAut(H_1(X, \mathcal B; \bZ))$. In coordinates, this latter group consists of matrices of the form $\begin{pmatrix}I &A\\ 0&I\end{pmatrix}$, so that $\Theta_{q,\tilde{K}}$ restricts to the (\emph{uncrossed}) homomorphism 
\begin{align*}
\tilde{K}^*: \PRelAut(H_1(X, \mathcal B; \bZ)) &\to H^1(X; \Z/2\Z)\\
A &\mapsto (x \mapsto \pair{x, A\tilde{K}}).
\end{align*}

Thus it suffices to show that there is a containment
\[
\ker(\tilde{K}^*) \le \bar \Gamma.
\]
This will be accomplished by exhibiting an explicit collection of elements of $\bar \Gamma \cap \ker(\tilde{K}^*)$ that constitute a generating set for $\ker(\tilde{K}^*)$. This computation will be greatly streamlined by the fact that $\ker(\tilde{K}^*)$ carries an action of $\bar \Gamma^{abs}$: it suffices to exhibit generators for $\ker(\tilde{K}^*)$ not merely as an abelian group, but as a $\bar \Gamma^{abs}$-module.

It will be useful to have a more explicit description of $\PRelAut(H_1(X,\mathcal B; \bZ))$. Having defined this above as 
\[
\PRelAut(H_1(X,\mathcal B; \bZ)):=\Hom(\widetilde H_0(\mathcal B; \bZ), H_1(X; \bZ)),
\]
we observe that by choosing bases for $\widetilde H_0(\mathcal B; \bZ) \cong \bZ^{n-1}$ and for $H_1(X; \bZ) \cong \bZ^{2g}$, there is an identification
\begin{equation}\label{equation:prelaut}
\PRelAut(H_1(X,\mathcal B; \bZ)) \cong \Mat_{2g,n-1}(\bZ).
\end{equation}
We will not need to choose an explicit basis for $H_1(X; \bZ)$, but it will be convenient to do so for $\widetilde H_0(\mathcal B; \bZ)$, and we take the basis 
\begin{equation}\label{equation:h0basis}
\{p_i - p_n \mid 1 \le i \le n-1\}.
\end{equation}

Recall that by convention, the distinguished point $p_i \in \cB \subset X$ is a zero of order $k_i \ge 0$, and that $\tilde{K} \in i_{\Delta,*}(\tilde H_0(\cB;\Z))$ is defined via
\[
\tilde{K} = s\left(\sum_{p_i \in \cB} k_i p_i\right).
\]

The lemma below gives a generating set for $\ker(\tilde{K}^*)$ in terms of the numerical invariants $k_i$.

\begin{lemma}\label{lemma:kergens}
As a $\bar \Gamma^{abs}$-module, $\ker(\tilde{K}^*)\le \Mat_{2g,n-1}(\bZ)$ is generated by the following collection of elements:
\begin{enumerate}
\item For each $1 \le i \le n -1$ such that $k_i$ is even, some matrix $U_1(i) \in \Mat_{2g,n-1}(\bZ)$ with $i^{th}$ column $\bold{v}$ for some primitive vector $\bold v$, and all other columns identically zero,
\item For each $1 \le i \le n-1$ such that $k_i$ is odd, some matrix $U_2(i) \in \Mat_{2g,n-1}(\bZ)$ with $i^{th}$ column $2 \bold v$ for some primitive vector $\bold v$, and all other columns identically zero,
\item For each pair of columns $1 \le i < j \le n-1$ such that $k_i$ and $k_j$ are both odd, some matrix $U_3(i,j) \in \Mat_{2g,n-1}(\bZ)$ with $i^{th}$ and $j^{th}$ columns $\bold v$ for some primitive vector $\bold v$, and all other columns identically zero.
\end{enumerate}
\end{lemma}
\begin{proof}
We first establish that, as an abelian group, $\ker(\tilde{K}^*)$ is generated by the collection of {\em all} elements of types $1,2,3$, where $\bold v$ ranges over the set of {\em all} primitive vectors. Following this, we will show that, as a $\bar \Gamma^{abs}$-module, only one element of each type is necessary.

Recall the definition of $\tilde{K}^*: \PRelAut(H_1(X,\mathcal B;\bZ)) \to H^1(X;\bZ/2\bZ)$:
\[
\tilde{K}^*(A)(x) = \pair{x, A\tilde{K}} \pmod 2.
\]
Concretely, viewing $A$ as an element of $\Mat_{2g, n-1}(\bZ)$, one can realize $\tilde{K}^*(A)$ as the column vector in $(\bZ/2\bZ)^{2g}$ obtained by summing the columns of $A$ for which the corresponding $k_i$ is odd. It is clear from this that $\ker(\tilde{K}^*)$ is generated as an abelian group by the set of all elements of types $1,2,3$, with $\bold v$ ranging over the set of all primitive vectors in $H_1(X;\bZ)$.

We turn now to the second step, bringing the action of $\bar \Gamma^{abs}$ into consideration. An element $M \in \bar \Gamma^{abs} \le \Sp(2g,\bZ)$ acts on $A \in \ker(\tilde{K}^*) \le \Mat_{2g,n-1}(\bZ)$ by left multiplication. In the case of $r$ odd, $\bar \Gamma^{abs} = \Sp(2g,\bZ)$ by \eqref{equation:monodromyodd}, which clearly acts transitively on primitive vectors, proving the claim in this case. 

In the case of $r$ even, $\bar \Gamma^{abs} = \Sp(2g,\bZ)[q]$ by \eqref{equation:monodromyeven}. According to \Cref{prop:symporbits}.\ref{item:sp}, there are exactly two orbits of primitive vectors, classified by the value $q(\bold v) \in \bZ/2\bZ$. Thus the $\bar \Gamma^{abs}$-orbit of a primitive vector $\bold v$ consists of all primitive vectors $\bold w$ with $q(\bold w) = q(\bold v)$. Since $\cH$ is non-hyperelliptic, $g \ge 3$ and so by \Cref{lemma:3tuple}, it is possible to find nonzero elements $\bold v_1, \bold v_2 \in H_1(X; \bZ)$ with any desired combination of values $q(\bold v_1), q(\bold v_2), \pair{\bold v_1, \bold v_2}$. Thus, given $\bold v$, let $\bold w$ satisfy $q(\bold w) = q(\bold v)$ and $\pair{\bold v, \bold w} = q(\bold v) + 1$. The mod-$2$ reduction of $\bold v+\bold w$ is necessarily nonzero, and so if $\bold v+ \bold w \in H_1(X; \bZ)$ is not primitive, its gcd must be odd. Altering the entries of $\bold w$ by even values, we may moreover select $\bold w$ such that $\bold v+ \bold w$ is primitive. Computing $q(\bold v + \bold w)$ yields
\[
q(\bold v + \bold w) = q(\bold v) + q(\bold w) + \pair{\bold v, \bold w} = q(\bold v) + 1.
\]
It follows that the $\bZ[\bar \Gamma^{abs}]$- span of $\bold v$ contains primitive vectors with both values of $q$, so that in the $r$ even case as well, only one generator of type $1$ is required (generators of types $2,3$ of course only exist in the $r$ odd case).
\end{proof}

Having established Lemma \ref{lemma:kergens}, we continue with the proof of Theorem \ref{theorem:homologicalmonodromy}. The outline of the remainder of the argument is as follows. Following Lemma \ref{lemma:kergens}, it suffices to exhibit loops in $\cH$ that represent the three types of generators. This will be accomplished by exhibiting translation surfaces with prescribed patterns of cylinders; combinations of shears about these cylinders will be used to realize the desired elements of $\bar \Gamma$. 

Before proceeding to the exhibition of these explicit elements, we explain the general method. In order to compute the monodromy action of our loops on $H_1(X,\mathcal B;\bZ)$, we first consider the more refined {\em topological} monodromy, valued in the pure mapping class group $\PMod(X,\mathcal B)$ (recall that this is the group of isotopy classes of orientation-preserving diffeomorphisms of $X$ that pointwise fix the set $\mathcal B$ of marked points). There is a natural map 
\[
\Psi^{rel}: \PMod(X,\mathcal B) \to \PAut(H_1(X,\mathcal B;\bZ))
\]
that records the action of a mapping class on $H_1(X,\mathcal B;\bZ)$, and a coarsening
\[
\Psi: \PMod(X, \mathcal B) \to \Sp(2g,\bZ)
\]
that tracks only the action on $H_1(X; \bZ)$. 

When a mapping class $f \in \PMod(X, \mathcal B)$ satisfies $\Psi(f) = I$, the refinement $\Psi^{rel}(f)$ is valued in the subgroup $\PRelAut(H_1(X,\mathcal B;\bZ)) \le \PAut(H_1(X,\mathcal B;\bZ))$. To attach an explicit matrix $A \in \Mat_{2g,n-1}(\bZ)$ to this, it is necessary to choose bases both for $H_1(X;\bZ)$ and for $\widetilde H_0(\mathcal B; \bZ)$, and further to lift these to a system of $2g$ simple closed curves, resp. $n-1$ arcs connecting the points of $\mathcal B$. In light of Lemma \ref{lemma:kergens}, it will not be necessary to specify explicit coordinates on $H_1(X;\bZ)$. 

If $\alpha, \alpha'$ are two distinct arcs connecting points $p_i, p_j \in \mathcal B$, then $\alpha-\alpha'$ is a simple closed curve. For $f \in \ker(\Psi)$, it follows that there is an equality of relative homology classes $f_*([\alpha]) -[\alpha] = f_*([\alpha']) - [\alpha']$, and so, having fixed a basis for $H_1(X;\bZ)$, the matrix $\Psi^{rel}(f) \in \Mat_{2g,n-1}(\bZ)$ depends only on the choice of basis for $\widetilde H_0(\mathcal B;\bZ)$, which we have fixed in \eqref{equation:h0basis} to be of the form $\{p_i - p_n\}$.

It is important to note that the loops we exhibit will be based at {\em different} basepoints in $\cH$, leading to the problem of apparently having to ``standardize'' by transporting everything to a single fixed basepoint. In spite of this apparent difficulty, the flexibility afforded by recalling the $\bar \Gamma^{abs}$-module structure on $\ker(\tilde{K}^*)$ will render this a non-issue. Indeed, the gcd of the entries in a fixed column of some $A \in \Mat_{2g, n-1}(\bZ)$ is invariant under the action of $\Sp(2g, \bZ)$. If a loop $\gamma \subset \cH$ based at some point $X \in \cH$ acts on $H_1(X,\mathcal B;\bZ)$ via some $A \in \Mat_{2g, n-1}(\bZ)$, a conjugate $\eta \gamma \eta^{-1}$ acts as $MA$, where $M = \Psi^{abs}(\eta)$ and $\eta \in \pi_1(\cH)$.  Thus, the gcd of column entries is invariant under conjugation in $\pi_1(\cH)$, and hence under change of basepoint. According to Lemma \ref{lemma:kergens}, the generators of type 1, 2 and 3 are characterized entirely in terms of the gcd of the entries of each column, and hence we are free to change basepoint as suits our needs. \\

Consider first a generator of type 1 as in Lemma \ref{lemma:kergens}. This is an element of $\PRelAut(H_1(X,\mathcal B;\bZ)) \cong \Mat_{2g,n-1}(\bZ)$ with a primitive vector $\bold v$ in some column $i$ such that $k_i$ is even, and all other columns are identically zero. 
	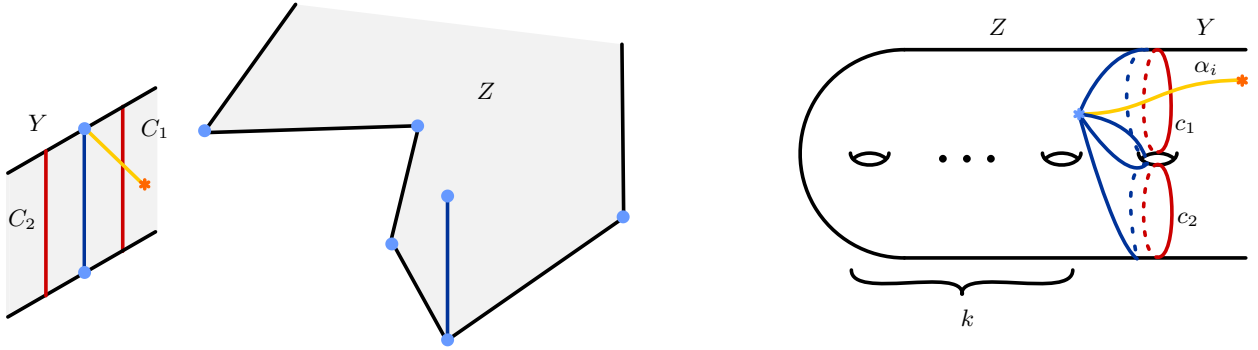
\begin{figure}[ht] 
		\labellist
		\small
\pinlabel $k$ [t] at 300.44 14.17
\pinlabel $c_1$ [bl] at 363.47 65
\pinlabel $c_2$ [l] at 363.47 39.68
\pinlabel $C_2$ [r] at 12 39.68
\pinlabel $C_1$ [l] at 39.68 68.03
\pinlabel $\alpha_i$ at 375 87
\pinlabel $Y$ at 10 70
\pinlabel $Z$ at 150 80
\pinlabel $Z$ at 310 100
\pinlabel $Y$ at 375 100
		\endlabellist
		\includegraphics[width=\textwidth]{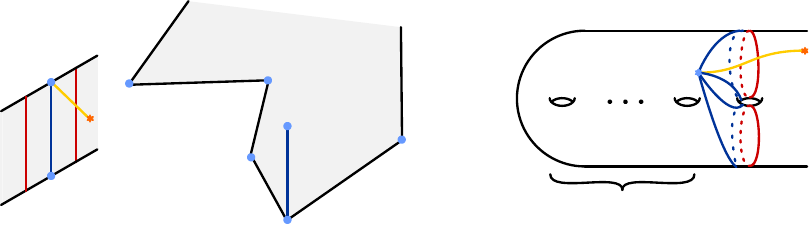}
		\caption{Exhibiting a generator of type 1.}
		\label{figure:gentype1}
	\end{figure}
	Consider Figure \ref{figure:gentype1}. The left side of the figure shows a translation surface constructed as follows: suppose first that $k_i = 2k$ for some $k \ge 1$ (the exceptional case $k_i = 0$ will be addressed at the end of the paragraph), and let $\kappa'$ be the partition of $2(g-k) - 2$ obtained by deleting $k_i$. Let $Y$ be a translation surface in $\cH_{lab}(\kappa')$, and let $Z$ be a translation surface in $\cH_{lab}(2k-2)$. Let $C$ be a cylinder on $Y$, and let $X \in \cH_{lab}(\kappa)$ be constructed as indicated in \Cref{figure:gentype1}, by joining $Y$ to $Z$ by cutting $Y$ along the core of $C$ and attaching to $Z$ via the slit construction. In the case where $r = \gcd(\kappa)$ is even, by choosing the components of $\cH_{lab}(\kappa')$ and $\cH_{lab}(2k-2)$ so that the induced Arf invariant on the glued surface matches with the Arf invariant of $\cH$, it is possible to construct such an $X$ in $\cH$. In the exceptional case $\kappa_i = 0$, we simply proceed by taking $X = Y$ with an extra marked point, ``splitting'' the cylinder $C$ in two - the reasoning encoded in Figure \ref{figure:gentype1} remains valid.

Any such $X$ is endowed with two cylinders $C_1, C_2$ which bound a subsurface, as indicated on the right side of Figure \ref{figure:gentype1}. Letting $S_C$ denote the (right-handed) shear about the cylinder $C$, we claim that the loop $S_{C_1} S_{C_2}^{-1}$ acts as a generator of type $1$ as in Lemma \ref{lemma:kergens}. To see this, we first note that the core curves $c_1, c_2$ of $C_1, C_2$ are homologous as elements of $H_1(X; \bZ)$. The action of $S_{C_i}$ on $H_1(X;\bZ)$ is given by the symplectic transvection $T_{c_i}$ which depends only on $[c_i] \in H_1(X; \bZ)$, and so $S_{C_1} S_{C_2}^{-1}$ acts trivially on $H_1(X; \bZ)$. It thus acts on $H_1(X, \mathcal B; \bZ)$ as some $A \in \Mat_{2g, n-1}(\bZ)$. Following the above discussion, we choose a system of arcs $\alpha_i$ connecting each $p_j (1 \le j \le n-1)$ to $p_n$, choosing $\alpha_i$ to cross $c_1$ once as shown in Figure \ref{figure:gentype1}. 

We see from this that 
\[
[S_{C_1}S_{C_2}^{-1}(\alpha_i) - \alpha_i] = [c_1];
\]
this is a primitive element of $H_1(X;\bZ)$. Any other $\alpha_j$ crosses $c_1$ and $c_2$ an equal number of times, and thus the homology class is unchanged under $S_{C_1}S_{C_2}^{-1}$. We conclude that $S_{C_1} S_{C_2}^{-1}$ acts on $H_1(X, \mathcal B; \bZ)$ as a generator of type 1 as claimed.\\

We next consider a generator of type 2 as in Lemma \ref{lemma:kergens}. This is an element of $\PRelAut(H_1(X,\mathcal B;\bZ)) \cong \Mat_{2g,n-1}(\bZ)$ with {\em twice} a primitive vector $2 \bold v$ in some column $i$ such that $k_i$ is odd, and all other columns identically zero. 
	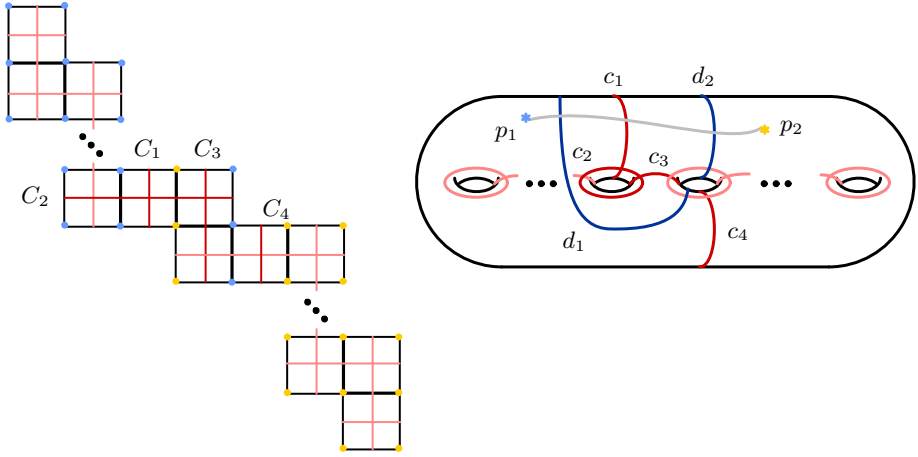
\begin{figure}[ht] 
		\labellist
		\small
		\pinlabel $C_1$ [b] at 53.85 107.71
\pinlabel $C_2$ [r] at 19.84 96.37
\pinlabel $C_3$ [b] at 76.53 107.71
\pinlabel $C_4$ [b] at 102.37 84.87
\pinlabel $c_1$ [b] at 229.58 135
\pinlabel $c_2$ [br] at 223.92 107.71
\pinlabel $c_3$ [b] at 246.59 104.87
\pinlabel $c_4$ [l] at 269.27 82.20
\pinlabel $d_1$ [tr] at 221.08 85.03
\pinlabel $d_2$ [b] at 263.27 135
\pinlabel $p_1$ [tr] at 195.57 124.71
\pinlabel $p_2$ [l] at 289.11 121.88
		\endlabellist
		\includegraphics[scale=1]{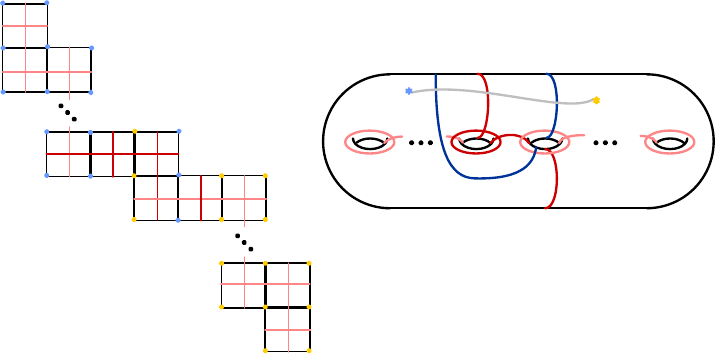}
		\caption{Exhibiting a generator of type 2.}
		\label{figure:gentype2}
	\end{figure}
	
We begin with the special case where $\kappa' = (k_1, k_2)$ with $k_1$ (and hence also $k_2$) odd. The left side of Figure \ref{figure:gentype2} shows a square-tiled surface $X$ in the stratum $\cH_{lab}(\kappa')$. Note the four cylinders labeled $C_1, \dots, C_4$ with associated core curves $c_1, \dots, c_4$. The {\em chain relation} asserts that the product of shears
\[
W = (S_{C_1} S_{C_2} S_{C_3})^4
\]
acts as a mapping class via
\[
w = (T_{c_1} T_{c_2} T_{c_3})^4 = T_{d_1} T_{d_2},
\]
where $d_1, d_2$ are the curves indicated on the right side of Figure \ref{figure:gentype2}. Note, that there is an equality
\[
[d_1] = [d_2] = [c_4]
\]
in $H_1(X;\bZ)$, and so the loop
\[
\gamma = W S_{C_4}^{-2}
\]
acts trivially on $H_1(X;\bZ)$. As in the first case, we can compute the action of $\gamma$ on $H_1(X, \mathcal B;\bZ)$ by choosing an arc connecting $p_1$ to $p_2 = p_n$. Doing so shows that the loop acts via the column matrix $A \in \Mat_{2g,1}(\bZ)$ given by $A = 2[c_4]$.

This exhibits a generator of type 2 in the stratum $\cH_{lab}(k_1, k_2)$ with $k_1, k_2$ both odd. In the general case, we can simply take a small perturbation $X'$ of $X$ which splits the zero $p_2$ of order $k_2$ to realize a translation surface with the same configuration of cylinders in any stratum $\cH_{lab}(\kappa)$ for $\kappa = \{k_1, \dots\}$ a refinement of $\kappa'$. The loop $\gamma$ exhibiting the generator of type 2 can be transported to a product of shears based at $X'$. It is clear that one can choose arcs connecting $p_n$ to $p_i$ for $i \ge 2$ that are left invariant by the resulting mapping class, so that the associated $A \in \Mat_{2g,n}(\bZ)$ has exactly one nonzero column, completing the argument in the general case.\\

It remains to exhibit a generator of type 3 as in Lemma \ref{lemma:kergens}. This is an element $A \in \Mat_{2g,n-1}(\bZ)$ with a primitive vector $\bold v$ in the $i^{th}$ and $j^{th}$ column, with $k_i, k_j$ both odd, and identically zero elsewhere. We let $\kappa'$ be the partition obtained from $\kappa$ by deleting $k_i, k_j$.

	\begin{figure}[ht] 
		\labellist
		\small
		\pinlabel $C_Y$ [r] at 53.85 102.04
\pinlabel $C_Z$ [r] at 206.91 102.04
\pinlabel $C_1$ [r] at 51.02 25.51
\pinlabel $C_2$ [l] at 76.53 25.51
\pinlabel $C_2$ [r] at 187.07 25.51
\pinlabel $C_1$ [l] at 209.74 25.51
\pinlabel $c_1$ [b] at 388.31 96.37
\pinlabel $c_2$ [t] at 388.31 25.51
\pinlabel $p_i$ [tr] at 345.79 85.03
\pinlabel $p_j$ [tr] at 362.80 42
\pinlabel $p_n$ [bl] at 416.65 76.53
\pinlabel $Y$ at 320 20
\pinlabel $Z$ at 465 20
		\endlabellist
		\includegraphics[scale=0.75]{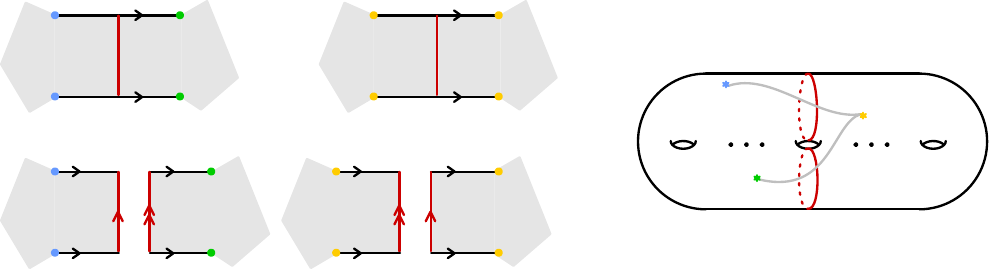}
		\caption{Exhibiting a generator of type 3.}
		\label{figure:gentype3}
	\end{figure}
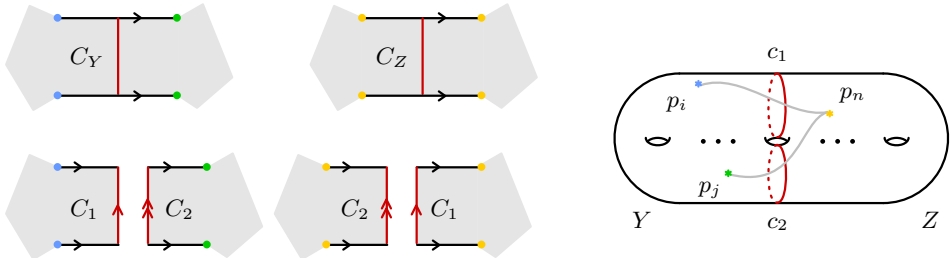
	Consider Figure \ref{figure:gentype3}. The top left portion of the figure depicts two translation surfaces: $Y \in \cH_{lab}(k_i, k_j)$ with $k_i, k_j$ both odd on the left, and $Z \in \cH_{lab}(\kappa')$ on the right (we assume here that $\kappa'$ partitions some {\em positive} integer; we will deal with the other possibility at the end of the paragraph). As shown, $Y$ and $Z$ have cylinders $C_Y$ and $C_Z$ with the same period. The bottom left of the figure shows the result of cutting $Y$ and $Z$ along these cylinders and re-attaching to make a translation surface $X \in \cH_{lab}(\kappa)$ equipped with two cylinders $C_1, C_2$ with homologous core curves $[c_1]=[c_2] = \bold v$. The resulting surface $X$ is depicted topologically on the right side of Figure \ref{figure:gentype3}. In the exceptional case $\kappa' = (0,\dots, 0)$, we proceed as in Case 1, constructing $X$ from $Y$ by splitting the cylinder $C_Y$ by adding the required number of marked points.

As $C_1,C_2$ are homologous, the product of shears $S_{C_1} S_{C_2}^{-1}$ acts trivially on $H_1(X;\bZ)$. Reasoning as in the prior two cases, one computes that $S_{C_1} S_{C_2}^{-1}$ acts on $H_1(X, \mathcal B;\bZ)$ via the matrix $A \in \Mat_{2g,n-1}(\bZ)$ with the primitive element $\mathbf v$ in the columns corresponding to $p_i$ and $p_j$, and is identically zero elsewhere, as required.
\end{proof}

\section{The $\mathbf{\psi}$ invariant}\label{section:psi}
Having computed the monodromy group in the non-hyperelliptic setting, we turn to the question of determining the structure of its orbits on the set of cyclic branched covers (and hence the classification of components of $\cM_\delta$). In this section, we define the {\em $\psi$ invariant} and establish its basic properties. In the following \Cref{section:splitnonhyp}, we then show that the $\psi$ invariant classifies split orbits, and we determine the associated cardinalities.

The $\psi$ invariant is closely related to the Arf invariant of a $\Z/2\Z$-quadratic form. We begin by recalling this notion.

\begin{definition}[Symplectic basis, Arf invariant]
    Let $(V, \pair{\cdot, \cdot})$ be a symplectic vector space over a field $K$. A {\em symplectic basis} for $V$ is a basis
    \[
    B = \{x_1, y_1, \dots, x_g, y_g\}
    \]
    for which $\pair{x_i, y_i} = 1$ and all other pairings among elements of $B$ are zero.

    Now specialize to $K = \Z/2\Z$, and let $q$ be a $\Z/2\Z$-quadratic form refining $\pair{\cdot, \cdot}$. The {\em Arf invariant} $\Arf(q) \in \Z/2\Z$ is defined as
    \[
    \Arf(q) = \sum_{i =1}^g q(x_i)q(y_i),
    \]
    where $B = \{x_i, y_i\}$ is any symplectic basis. 
    \end{definition}
    We recall the well-known facts that $\Arf(q)$ is well-defined independently of the choice of $B$, and that $\Arf(q)$ is the unique invariant of the action of $\Sp(2g, \Z/2\Z)$ on the set of $\Z/2\Z$-quadratic forms \cite{Arf}.
    \\

\begin{definition}[Arf invariant of a framing]\label{definition:arf}
    Let $X$ be a translation surface for which every zero has even order, and let $\phi$ be the associated framing of $X$. Define
    \[
    \Arf(\phi) := \Arf(q),
    \]
    where $q$ is the $\Z/2\Z$-quadratic form associated to $\phi$ via \Cref{lemma:induced}.
\end{definition}

\begin{remark}\label{remark:arfs}
In the setting of strata of translation surfaces, when the component $\mathcal H$ supports a spin structure, the Arf invariant $\Arf(\phi)$ of the translation surface framing coincides with the usual Arf invariant classifying the stratum component in the sense of the Kontsevich-Zorich classification.
\end{remark}

We are now ready to define the $\psi$ invariant, a $\Z/2\Z$-valued invariant of the coset $\delta^{-1}(\kappa) \subset H_1(X, \mathcal B; \Z/2\Z)$. {\em A priori}, it depends on a choice of geometric splitting; we show that it is well-defined in \Cref{lemma:psiWD}, which relies on the preliminary results \Cref{lemma:qanu,lemma:action}. Finally, in \Cref{lemma:qinvariant}, we show that $\psi$ is an invariant of the monodromy action.

 \begin{definition}[The $\psi$ invariant]\label{definition:psi}
        Fix a geometric splitting $X^\circ$ of $X$, giving rise to an isomorphism $H_1(X,\mathcal B; \Z/2\Z) \cong H_1(X; \Z/2\Z) \oplus i_{\Delta,*}(\tilde H_0(\mathcal B; \Z/2\Z))$ and associated data $(q,\tilde{K})$ as in \Cref{lemma:discussion}. Given $\balpha = x + \tilde{K}$, we define
        \[
            \psi(\balpha) = \psi(x + \tilde{K}) := q(x) + \Arf(q) \in \Z/2\Z.
        \]
\end{definition}

Recall that the set $Q(H_1(X; \Z/2\Z))$ of quadratic forms $q: H_1(X;\Z/2\Z) \to \Z/2\Z$ forms a {\em torsor} over $H^1(X; \Z/2\Z)$. In particular, given $q \in Q(H_1(X; \Z/2\Z))$ and $\bxi \in H^1(X;\Z/2\Z)$, then $q+\bxi \in Q(H_1(X; \Z/2\Z))$ is again a quadratic form. 

\begin{lemma} \label{lemma:qanu}
Let $\MAOI \in \PAut(H_1(X, \mathcal B;\Z/2\Z))$ be given, and define 
\[\bxi = \Theta_{q,\tilde{K}}\left(\MAOI\right) \in H^1(X; \Z/2\Z).\]
Then $q(A \tilde{K}) = \Arf(q) + \Arf(q+\bxi)$; in particular if $\MAOI \in \ker(\Theta_{q,\tilde{K}})$, then $q(A\tilde{K}) = 0$.
\end{lemma}

\begin{proof}
If $A \tilde{K} = 0$, then $\bxi(x) = q(Mx) + q(x)$ and so $(q+\bxi)(x) = q(Mx)$. As $M \in \Sp(2g,\Z)$, the forms $q(x)$ and $q(Mx)$ have the same Arf invariant, showing the result in this case.

If $A \tilde{K} \ne 0$, then choose a symplectic basis $x_1, y_1, \dots, x_g, y_g$ for $H_1(X; \Z/2\Z)$ for which $A\tilde{K} = M x_1$. As $M \in \Sp(2g, \Z/2\Z)$, it preserves the Arf invariant of $q$:
\[
\sum_{i = 1}^g q(M x_i)q(M y_i) = \sum_{i = 1}^g q(x_i) q(y_i).
\]
By the definition of $\bxi$, the condition $A \tilde{K} = M x_1$ and the symplectic invariance of the intersection pairing,
\[
q(Mx) = q(x) + \pair{Mx, A \tilde{K}} + \bxi(x) = (q+\bxi)(x) + \pair{x, x_1}.
\]
It follows that
\[
\sum_{i = 1}^g q(x_i) q(y_i)= \sum_{i = 1}^g q(M x_i)q(M y_i) = \sum_{i =1}^g((q+\bxi)(x_i)+ \pair{x_i,x_1})((q+\bxi)(y_i)+ \pair{y_i,x_1}).
\]

In the above, every expression of the form $\pair{z,x_1}$ is zero except $\pair{y_1,x_1}=1$. Comparing the left and right sides then shows that 
\[
\Arf(q) = \Arf(q + \bxi) + (q + \bxi)(x_1).
\]
Expanding, recalling that $M x_1 = A \tilde{K}$, 
\[
(q + \bxi)(x_1) = q(Mx_1) + \pair{Mx_1, A\tilde{K}} = q(A \tilde{K}),
\]
which completes the proof.
\end{proof}

\begin{lemma} \label{lemma:action}
For $i \in \{1,2\}$, let $X_i^\circ$ be geometric splittings with associated data $(q_i, s_i(\kappa))$. Let $f \in \Mod(X, \mathcal B)$ satisfy $f(X_1^\circ) = X_2^\circ$ and induce the identity on $H_1(X;\Z/2\Z)$. Then the quadratic forms $q_i$ on $H_1(X;\Z/2\Z)$ are related by the formula
\[
q_2(x) = q_1(x) + \Theta_{q_1, \tilde{K}_1}(F)(x) = q_1(x) + \pair{x,A \tilde{K}_1},
\]
where $F = \begin{pmatrix} I& A\\0&I\end{pmatrix} \in \PAut(H_1(X, \mathcal B;\Z/2\Z))$ is the automorphism on homology associated to $f$. 
\end{lemma}

\begin{proof}
For $x \in H_1(X;\Z/2\Z)$, recall that $q_i(x) = \phi(c_i) +1$, where $c_i$ is a representative for $x$ supported on $X_i^\circ$. Given $f \in \Mod(X,\mathcal B)$ as in the statement of the lemma, we therefore have $q_2(x) = \phi(f(c_1)) + 1$.  By \Cref{lemma:qformula},
\[
q_2(x) = \phi(f(c_1)) + 1= q_1(x) + \pair{f(c_1),\tilde{K}_1}.
\]
It therefore suffices to show $\pair{f(c_1), \tilde{K}_1} = \pair{x, A\tilde{K}_1}$. To see this, we note that the pairing
\[
\pair{\cdot, \cdot}: H_1(X \setminus \mathcal B; \Z/2\Z) \otimes H_1(X, \mathcal B; \Z/2\Z) \to \Z/2\Z
\]
intertwines the action of $\PMod(X, \mathcal B)$ on the factors, so that
\[
\pair{f(c_1),\tilde{K}_1} = \pair{c_1, f^{-1}(\tilde{K}_1)}.
\]
The action of $f^{-1}$ on $H_1(X, \mathcal B; \Z/2\Z)$ is by $F^{-1}$; as we are working mod $2$, it follows that $F^{-1} = F = \begin{pmatrix} I& A\\0&I\end{pmatrix}$. Thus $[f^{-1}(\tilde{K}_1)] = A \tilde{K}_1 + \tilde{K}_1 \in H_1(X, \mathcal B; \Z/2\Z)$. As $c_1$ is supported on $X_1^\circ$ and so is disjoint from the representative for $\tilde{K}_1$, it follows that $\pair{c_1, A\tilde{K}_1 + \tilde{K}_1} = \pair{c_1, A\tilde{K}_1}$. This latter expression is now an intersection of {\em absolute} homology classes, and as $[c_1] = x \in H_1(X; \Z/2\Z)$, the claim follows.
\end{proof}

\begin{lemma}\label{lemma:psiWD}
The value $\psi(\balpha) \in \Z/2\Z$ is well-defined independent of the choice of geometric splitting.
\end{lemma}
            
\begin{proof}
Consider two geometric splittings $X_i^\circ$ for $i \in \{1,2\}$, along with associated data $(q_i, s_i(\kappa))$. Represent $\balpha$ with $\delta(\balpha) = \kappa$ in the forms
\[
\balpha = x_i + s_i(\kappa),
\]
and choose representative simple closed curves $c_i \subset X_i^\circ$ for the classes $x_i$. Via Lemmas~\ref{lemma:qformula} and \ref{lemma:action}, and because $\pair{x_2,\tilde{K}_2} = 0$ by construction,
\[
\phi(c_2) = q_2(x_2) + 1 = q_1(x_2) + 1 + \pair{x_2, \tilde{K}_1}.
\]
Express $\tilde{K}_2 = \tilde{K}_1 + y$ for $y \in H_1(X;\Z/2\Z)$. As $\balpha = x_1 + \tilde{K}_1 = x_2 + \tilde{K}_2$, then $x_2 = x_1 + y$, and so 
\begin{align*}
    q_1(x_2) + \pair{x_2, \tilde{K}_1} &= q_1(x_1 + y) + \pair{x_2, \tilde{K}_2 + y}\\
                                        &= q_1(x_1) + q_1(y) + \pair{x_1,y} + \pair{x_1 + y, y}\\
        &= q_1(x_1) + q_1(y),
\end{align*}
since $\pair{x_2,\tilde{K}_2} = 0$. It therefore suffices to show that
\[
q_1(y) = \Arf(q_1) + \Arf(q_2). 
\]
To see this, let 
$f \in \Mod(X, \mathcal B)$ and $F \in \PAut(H_1(X, \mathcal B;\Z/2\Z))$ be as in Lemma \ref{lemma:action}. Since $f(X_1^\circ) = X_2^\circ$,
it follows that $y = A\tilde{K}_1$. Moreover, $q_2 = F\cdot q_1 = q_1 + \pair{\cdot, A \tilde{K}_1}$ by \Cref{lemma:action}, and so
\[
q_1(y) = q_1(A \tilde{K}_1) = \Arf(q_1) + \Arf(q_1 + \Theta_{q_1,\tilde{K}_1}(F)) = \Arf(q_1) + \Arf(q_2)
\]
by \Cref{lemma:qanu}.
\end{proof}    
        
            Having seen that $\psi$ is well-defined independent of choice of geometric splitting, we proceed to show that it is invariant on orbits of $\ker(\Theta_{q,\tilde{K}})$.
            
\begin{lemma}\label{lemma:qinvariant}
            Let $\balpha = x + \tilde{K} \in H_1(X, \mathcal B; \Z/2\Z)$ and $F \in \ker(\Theta_{q,\tilde{K}})$ be given. Then
            \[
            \psi(F \balpha) = \psi (\balpha).
            \]
\end{lemma}

\begin{proof}
Write 
\[
F = \MAOI;
\]
then we must verify that
\[
q(Mx + A \tilde{K}) = q(x).
\]
As $q$ is a quadratic form,
\begin{equation}\label{equation:bunchaqs}
q(Mx + A\tilde{K}) = q(Mx) + q(A \tilde{K}) + \pair{Mx,A\tilde{K}}.
\end{equation}
From \Cref{lemma:qanu}, $q(A\tilde{K}) = 0$. As $F \in \ker(\Theta_{q,\tilde{K}})$, the equation 
\[
q(Mx) + q(x) + \pair{Mx, A\tilde{K}} = 0
\]
holds for all $x \in H_1(X; \Z/2\Z)$, and so we can replace $q(Mx) + \pair{Mx,A\tilde{K}}$ in \eqref{equation:bunchaqs} with $q(x)$. Altogether, this shows
\[
\psi(F\balpha) = q(Mx + A \tilde{K})+ \Arf(q) = q(x)+ \Arf(q)= \psi(\balpha)
\]
as claimed. 
\end{proof}

\section{Orbit counts (II): Split non-hyperelliptic case}\label{section:splitnonhyp}

Having defined the $\psi$ invariant, we now see how it gives a complete classification of orbits in the non-hyperelliptic setting. Recall that \Cref{prop:oddporbits} and \Cref{corollary:level2nonhyp} show that under the Chinese remainder theorem decomposition, the orbit $\overline{\Gamma}_d \balpha$ is split if and only if the mod-$2$ reduction $\bar \Gamma_2 \bar \balpha$ is split, and that the Chinese remainder theorem decomposition \Cref{theorem:CRT} allows us to consider orbits one prime at a time. Thus in this section, we need only consider the prime $p = 2$.

\subsection{Statement of results}\label{stmtnh}

\begin{proposition}[$p = 2$ orbit classification, non-hyperelliptic]\label{prop:2splitorbits}
Let $(X,\omega)$ be a translation surface in a non-hyperelliptic component of a stratum. Let $\balpha, \bbeta \in H_1^{prim}(X, \mathcal B; \Z/2^k\Z)$ be given, and suppose that $\delta:= \delta(\balpha) = \delta(\bbeta) \in \tilde H_0(\mathcal B; \Z/2^k\Z)$. 
\begin{itemize}
\item If $\delta \ne \kappa \pmod{2}$, then every orbit is unsplit: $\bar \Gamma_{2^k} \balpha = \bar \Gamma_{2^k} \bbeta$. 
\item If $\delta = \kappa \pmod{2}$, then $\bar \Gamma_{2^k} \balpha = \bar \Gamma_{2^k} \bbeta$ if and only if $\psi(\balpha) = \psi(\bbeta)$. 
\end{itemize}
\end{proposition}

\begin{proposition}[$p= 2$ split orbit sizes, non-hyperelliptic]\label{prop:2splitsize}
Let $(X,\omega)$ be a translation surface in a non-hyperelliptic component of a stratum, and let  $\balpha \in H_1^{prim}(X,\mathcal B; \Z/2^k\Z)$.  Assume further that $\delta(\balpha) = \kappa \pmod 2$, so that the orbit of $\balpha$ is split. Define
    \[
    \tau(\balpha) = \begin{cases}
              1 & \mbox{if $2 | \gcd(\kappa)$ and $\psi(\balpha) = \Arf(\phi)$}\\ 
            0 & \mbox{otherwise.}\\
    \end{cases}
    \]
    Then
            \[
            \abs{\bar \Gamma_{2^k}\balpha} = 2^{2g(k-1)}\left(2^{g-1}(2^g + (-1)^{\psi(\balpha)})-\tau(\balpha)\right).
            \]
\end{proposition}

\begin{proposition}[$p=2$ split orbit statistics, non-hyperelliptic]\label{prop:2splitstats}
Let $(X,\omega)$ be a translation surface in a non-hyperelliptic component of a stratum. Let $\bar \Gamma_{2^k} \balpha \subset H_1^{prim}(X,\mathcal B; \Z/2^k\Z)$ be an orbit.  Suppose that $\delta(\balpha) = \kappa$, so that the orbit of $\balpha$ is split. Let $c \in H_1^{prim}(X \setminus \mathcal B; \Z/p^k\Z)$ be the homology class of a cylinder on $X$, and let $a \in \Z/p^k\Z$ be given. Then
\[
\abs{(\bar \Gamma_{2^k}\balpha)_{c;a})} = \begin{cases}
2^{(2k-1)g -k}(2^{g-1}+(-1)^{\psi(\balpha)}) & a=1 \pmod 2\\
2^{(2g-1)(k-1)}(2^{2g-2}-\tau(\balpha))      & a=0 \pmod 2,
\end{cases}
\]
where $\tau(\balpha)$ is defined as in \Cref{prop:2splitsize}.
\end{proposition}
Note that as discussed in Section~\ref{FixedHomClassSection} and as in Proposition~\ref{prop:oddpstats}, these quantities are independent of $c$.

\subsection{Proofs}
Throughout this section, we will tacitly make use of \Cref{theorem:homologicalmonodromy}, to identify the monodromy group $\overline{\Gamma}_d$ with $\ker(\Theta_{q,\tilde{K}})$. As in \Cref{subsection:geomsplit2}, we will also fix a geometric splitting and the induced algebraic splitting.

\subsubsection{Orbit classification}
\begin{proof}[Proof of \Cref{prop:2splitorbits}]
According to \Cref{prop:oddporbits} and \Cref{corollary:level2nonhyp}, the orbit $\bar \Gamma_{2^k} \balpha$ is unsplit if and only if its reduction $\bar \Gamma_2 \bar \balpha$ is unsplit. Thus it suffices to consider the case $k =1$ in what follows.

We first consider the case $\delta \ne \kappa$; suppose first that $\delta \neq 0$ as well. Let $y \in H_1(X; \Z/2\Z)$ be arbitrary. Since $\delta \neq 0$, we can take $A \in \Hom(i_{\Delta,*}(\tilde H_0(\mathcal B;\Z/2\Z)), H_1(X;\Z/2\Z))$ such that $Ai_{\Delta,*}(\delta) = y-x$.
Moreover, since $\delta \neq \kappa$, either $\delta$ and $\kappa$ are linearly independent or else $\kappa = 0$; in either case, we can also stipulate that $A \tilde{K} =0$.
In particular, $\pair{z,A\tilde{K}} = 0$ for all $z \in H_1(X;\Z/2\Z)$ and so
$\begin{pmatrix}
I & A\\ 0&I
\end{pmatrix} \in \ker(\Theta_{q,\tilde{K}}).$
Now by construction this element takes $x + \delta$ to $x + A\delta + \delta = y + \delta$, and thus $\ker(\Theta_{q,\tilde{K}})$ acts transitively on vectors of the form $x + \delta$, as desired.

In the case $\delta \ne \kappa$ and $\delta = 0$, note first that by the assumption that the cover be connected that $x$ must be nonzero. Take an arbitrary nonzero $y \in H_1(X; \Z/2\Z)$; then by Proposition \ref{prop:symporbits}.\ref{item:sp} there is an element $M \in \Sp(2g,\Z/2\Z)$ such that $Mx = y$.
We seek to construct $A \in \Hom(\tilde H_0(\mathcal B;\Z/2\Z), H_1(X;\Z/2\Z))$ such that 
\begin{equation}\label{equation:qcond}
q(Mz) + q(z) + \pair{Mz,A\tilde{K}} = 0
\end{equation}
for all $z \in H_1(X; \Z/2\Z)$. Since the symplectic pairing is nondegenerate, \eqref{equation:qcond} completely characterizes $A\tilde{K}$, and so a suitable $A$ can be constructed.

Finally, we consider the case $\delta = \kappa$. Here we seek to show that a primitive $x+\tilde{K}$ can be taken to any primitive $y + \tilde{K}$ with $\psi(x+ \tilde{K}) = \psi(y + \tilde{K})$, i.e., with $q(x) = q(y)$.
In the case that both $x$ and $y$ are nonzero, \Cref{prop:symporbits}.\ref{item:spq} tells us there is some $M \in \Sp(2g, \Z/2\Z)[q]$ taking $x$ to $y$. Now any element $\begin{pmatrix}M & 0\\0&I \end{pmatrix}$ for $M \in \Sp(2g,\Z/2\Z)[q]$ is an element of $\ker(\Theta_{q,\tilde{K}})$ and takes $x$ to $y$. If $\kappa = 0$, then this argument completes the proof in that case.

The last thing to prove is that if $\kappa \ne 0$, there is an element of $\ker(\Theta_{q,\tilde{K}})$ taking $\tilde{K}$ to some $x + \tilde{K}$ for a nonzero $x \in H_1(X;\Z/2\Z)$.
This can be done by choosing some $M \in \Sp(2g,\Z/2\Z)$ that {\em does not} preserve $q$. Then \eqref{equation:qcond} characterizes $A\tilde{K}$, and as long as $M \not \in \Sp(2g, \Z/2\Z)[q]$, then $A\tilde{K} \ne 0$. Then $\MAOI \in \ker(\Theta_{q,\tilde{K}})$ takes $\tilde{K}$ to $A \tilde{K} + \tilde{K}$ as required; applying the previous paragraph then allows us to take $A\tilde{K} + \tilde{K}$ to any other $y + \tilde{K}$ so that $q(A\tilde{K}) = q(y)$.
\end{proof}

\subsubsection{Orbit sizes}

Before turning to the proof of \Cref{prop:2splitsize}, we record a lemma. The following result is standard in the theory of the Arf invariant, and gives a quantitative formulation of the slogan that the Arf invariant is the ``democratic invariant,'' equal to whichever value in $\Z/2\Z$ is assumed more often by $q$.
\begin{lemma}\label{lemma:qcount}
Let $q: H_1(X; \Z/2\Z) \to \Z/2\Z$ be a quadratic form. Then for $\epsilon\in \Z/2\Z$, there are 
\[
2^{g-1}(2^g + (-1)^{\Arf(q) +\epsilon})
\]
elements $x \in H_1(X; \Z/2\Z)$ satisfying $q(x) = \epsilon$.
\end{lemma}
\begin{proof}
Induction on $g$.
\end{proof}

\begin{proof}[Proof of \Cref{prop:2splitsize}]
Fix a geometric splitting $X^\circ$, inducing an algebraic splitting as in Definition~\ref{definition:geomsplit}.
By \Cref{prop:2splitorbits}, $\bar \Gamma_{2^k} \balpha$ consists of all $\bbeta = y + \tilde{K} \in H_1^{prim}(X,\mathcal B; \Z/2^k\Z)$ such that $\psi(\bbeta) = \psi(\balpha)$. Expressing $\balpha = x + \tilde{K}$, this latter condition is equivalent to $q(x) = q(y)$. We consider the two possibilities $\gcd(\kappa) = 1$ or $2|\gcd(\kappa)$ in turn.

If $\gcd(\kappa) = 1$, the orbit $\bar \Gamma_{2^k} \balpha$ is in bijection with $y \in H_1(X; \Z/2^k\Z)$ for which $q(y) = q(x)$. The value $q(x) = q(y)$ is determined by the mod-$2$ reduction. By \Cref{lemma:qcount}, there are
\[
2^{g-1}(2^g + (-1)^{\Arf(q)+q(x)}) = 2^{g-1}(2^g + (-1)^{\psi(\balpha)})
\]
elements $\bar x \in H_1(X; \Z/2\Z)$ with $q(x) = q(y)$, for a total of
\begin{equation}\label{eqn:nonhypsplitorbitct}
 2^{2g(k-1)}\left(2^{g-1}(2^g + (-1)^{\psi(\balpha)})\right)
\end{equation}
elements $y \in H_1(X; \Z/2^k\Z)$ that satisfy $q(y) = q(x)$. Recalling $\tau(\balpha) = 0$ if $\gcd(\kappa) = 1$, this proves the claim in this case. 

If $2|\gcd(\kappa)$, then $y \in H_1(X; \Z/2^k\Z)$ must additionally be primitive. If $q(y) = 1$, then $y$ is {\em necessarily} primitive, so that the orbit count is again given by \eqref{eqn:nonhypsplitorbitct}.
We observe that since $2 | \gcd(\kappa)$, the condition $q(x) = 1$ is equivalent to setting $\tau(\balpha) = 0$: indeed, $\tau(\balpha) = 0$ is equivalent to
\[\Arf(\phi) = \psi(\balpha) + 1 = q(x) + \Arf(q) + 1 = q(x) + \Arf(\phi) + 1.\]

In case $q(x) = 0$, we must exclude all nonprimitive $y \in H_1(X; \Z/2^k\Z)$ from consideration. Such $y$ is primitive if and only if its reduction mod $2$ is nonzero, so that there are 
\[
2^{2g(k-1)}\left(2^{g-1}(2^g + (-1)^{\psi(\balpha)})-1\right) = 2^{2g(k-1)}\left(2^{g-1}(2^g + (-1)^{\psi(\balpha)})-\tau(\balpha)\right)
\]
primitive $y \in H_1(X; \Z/2^k\Z)$ satisfying $q(y) = 0$.
\end{proof}

\subsubsection{Orbit statistics}

The proof of \Cref{prop:2splitstats} requires the following preliminary lemma. Let $q$ be a quadratic form on $H_1(X; \Z/2\Z)$, and fix $c \in H_1(X; \Z/2\Z)$ with $q(c) = 1$. For $\epsilon, a \in \Z/2\Z$, define 
\[
N(g, q, \epsilon,a) = \abs{\{y \in H_1(X; \Z/2\Z) \mid q(y) = \epsilon,\ \pair{c,y}=a \}}.
\]
\begin{lemma}\label{lemma:pairingcount}
For any $c \in H_1(X; \Z/2\Z)$ and any quadratic form with $q(c) = 1$,
\[
N(g,q, \epsilon,a) = \begin{cases}
2^{2g-2} & a = 0\\
2^{g-1}(2^{g-1}+(-1)^{\epsilon+ \Arf(q)}) & a = 1.
\end{cases}
\]
\end{lemma}

\begin{proof}
We first remark that since $\Sp(2g,\Z/2\Z)$ acts transitively on the set of quadratic forms of given Arf invariant, we are free to work with whichever particular $q$ is most convenient. Likewise, by \Cref{prop:symporbits}.\ref{item:spq}, the stabilizer $\Sp(2g,\Z/2\Z)[q]$ acts transitively on the set of elements $c \in H_1(X;\Z/2\Z)$ satisfying $q(c) = 1$ (note that all such $c$ are nonzero), and so we are free to choose a convenient such $c$. 

We fix such choices as follows. Let $x_1, y_1, \dots, x_g, y_g$ be a symplectic basis for $H_1(X; \Z/2\Z)$, and define $q$ by the conditions that $q(x_1) = 1$ and $q(x_i) = q(y_i) = 0$ for $i > 1$; then $q(y_1) = \Arf(q)$. We choose $c = x_1$. Having made such choices, the proof proceeds by induction on $g$. 

In the base case $g = 1$, we observe that $\pair{x_1,0} = \pair{x_1,x_1} = 0$, and $\pair{x_1, y_1}= \pair{x_1, x_1 + y_1} = 1$. The first pair of equalities shows that
\[
N(1,q, \epsilon, 0) = 1 = 2^{2g-2}.
\]
As $q(x_1 + y_1) = q(x_1) + q(y_1) + \pair{x_1,y_1} = q(y_1)$, it follows that $N(1, q, \epsilon, 1)$ is either $2$ or $0$ according to whether $\epsilon = \Arf(q)$ or not, establishing the formula in the case $g = 1, a = 1$.

Now consider $N(g+1, q, \epsilon, a)$ for $g \ge 1$. Every vector in $(\Z/2\Z)^{2g+2}$ is of the form $v + w$ for $v \in (\Z/2\Z)^{2g}$ and $w \in \{0, x_{g+1}, y_{g+1}, x_{g+1} + y_{g+1}\}$. Note that $\pair{x_1, v+w} = \pair{x_1, v}$ and that $q(v+w) = q(v) + q(w)$. Of the four possibilities for $w$, only $x_{g+1} + y_{g+1}$ has $q = 1$. We therefore deduce the recursion
\[
N(g+1, q, \epsilon, a) = 3N(g, q, \epsilon, a) + N(g, q, \epsilon + 1, a),
\]
A simple inductive argument then shows that the claimed formulas hold. 
\end{proof}

\begin{proof}[Proof of \Cref{prop:2splitstats}]
According to \Cref{prop:2splitorbits}, $\bar \Gamma_{2^k} \balpha$ consists of all $\bbeta \in H_1^{prim}(X, \mathcal B; \Z/2^k\Z)$ such that 
\begin{enumerate}
    \item $ \bbeta = y + \tilde{K}$ for $y \in H_1(X; \Z/2^k\Z)$,
    \item $\psi( \bbeta) = \psi(\balpha) \in \Z/2\Z$.
\end{enumerate}

Here we must make a further distinction between $\gcd(\kappa) = 1$ and $2|\gcd(\kappa)$. In the case $\gcd(\kappa) = 1$, the condition $\gcd( \bbeta) = 1$ is automatically satisfied. Choose a geometric splitting under which the cylinder with class $c$ is contained in the subsurface $X^\circ$, and let $(q,\tilde{K})$ be the associated data. As $c$ is the class of a cylinder, necessarily $\phi(c) = 0$, i.e., $q(c) = 1$. 

We seek to count $y \in H_1(X;\Z/2^k \Z)$ such that $\psi(y+\tilde{K}) = q(y) + \Arf(q) = \psi(\balpha)$ and such that
\[
\pair{c, y + \tilde{K}} = \pair{c, y} = a \pmod{2^k}.
\]
As $q(y)$ depends only on the class of $y$ in $\Z/2\Z$, it suffices to work with $\Z/2\Z$ coefficients; each such $\bar y$ then corresponds to $2^{(k-1)(2g-1)}$ such $y \in H_1(X; \Z/2^k\Z)$ with the value $\pair{c, y}$ fixed. Applying \Cref{lemma:pairingcount}, this quantity is given as
\begin{align*}
\abs{\{y \in H_1(X; \Z/2^k\Z) \mid q(y) + \Arf(q) = \psi(\balpha),\ \pair{c,y} = a\}}\\
 = 2^{(k-1)(2g-1)} N(g,q, \Arf(q) + \psi(\balpha),a).
\end{align*}
This completes the argument in the case $\gcd(\kappa) = 1$.

It remains to consider the case $2 | \gcd(\kappa)$. Recall from \Cref{definition:arf} that in this case, the Arf invariant $\Arf(\phi) = \Arf(q)$ is globally well-defined independent of the choice of geometric splitting. In this case, the condition $\gcd(\bbeta) = 1$ will be satisfied if and only if $ \bbeta = y + \tilde{K}$ with $y \in H_1(X; \Z/2^k\Z)$ primitive, and this latter condition holds if and only if the reduction $y \pmod 2$ is nonzero. As in the previous case, we can reduce mod $2$:
\begin{align*}
&\abs{\{y \in H_1(X; \Z/2^k\Z) \mid y \mbox{ primitive, } q(y) + \Arf(q) = \psi(\balpha),\ \pair{c,y} = a\}} \\
& = 2^{(k-1)(2g-1)} \abs{\{y \in H_1(X; \Z/2\Z) \mid y \ne 0, \ q(y) + \Arf(q) = \psi(\balpha),\ \pair{c,y} = a\}}.
\end{align*}

If $a = 1$ then the condition $y \ne 0$ is redundant, and we can count as before:
\begin{align*}
    \abs{\{y \in H_1(X; \Z/2^k\Z) \mid y \mbox{ prim., } q(y) + \Arf(q) = \psi(\balpha),\ \pair{c,y} = 1\}}\\
     = 2^{(k-1)(2g-1)} N(g, q, \Arf(q) + \psi(\balpha),1)\\
    = 2^{(2k-1)g -k}(2^{g-1}+(-1)^{\psi(\balpha)})
\end{align*}
as claimed.

If $a = 0$ and $\Arf(q) + \psi(\balpha) = 0$, then we must exclude $y = 0$ (and all its preimages before reducing mod 2) from the count. Recall that $\tau(\balpha) = 1$ if and only if $\gcd(\kappa) > 1$ and $\psi(\balpha) = \Arf(q)$, so that 
\begin{align*}
&\abs{\{y \in H_1(X; \Z/2^k\Z) \mid y \mbox{ prim., } q(y) + \Arf(q) = \psi(\balpha),\ \pair{c,y} = 0\}}\\
&= 2^{(k-1)(2g-1)}(N(g,q, \Arf(q) + \psi(\balpha),0)- \tau(\balpha))\\
&= 2^{(k-1)(2g-1)}(2^{2g-2}-\tau(\balpha)).
\end{align*} 
\end{proof}

\section{The homological monodromy group, hyperelliptic case}\label{section:monohyp}
In these last two sections of Part II, we turn our attention to the hyperelliptic setting. The goal of this section is to establish \Cref{corollary:level2hyp}, which is the counterpart of \Cref{corollary:level2nonhyp} in the hyperelliptic setting. We begin with a recollection of some of the basics of the {\em topological} monodromy of a hyperelliptic component.

\subsection{The hyperelliptic mapping class group}\label{section:hyperellipticmcg}
We first make some general comments about hyperelliptic translation surfaces; see also \cite{LM, C_strata1}, and \cite{KontsevichZorichConnComps}.

Recall \cite{KontsevichZorichConnComps} that the strata $\cH(2g-2)$ and $\cH(g-1,g-1)$ each contain a special {\em hyperelliptic} component, written $\cH^{hyp}(2g-2)$ and $\cH^{hyp}(g-1,g-1)$, or just as $\cH^{hyp}$ when we wish to be agnostic about the number of zeros. These consist of differentials $\omega$ on hyperelliptic Riemann surfaces arising as the global square root of a quadratic differential on $\CP^1$ with a unique zero and either $2g+1$ or $2g+2$ simple poles, respectively. For any $(X, \omega) \in \cH^{hyp}(2g-2)$, the hyperelliptic involution $\iota$ fixes the zero, and for any $(X, \omega) \in \cH^{hyp}(g-1, g-1)$ it interchanges the zeros \cite[Remark 3]{KontsevichZorichConnComps}. 

\begin{definition}\label{definition:br}
Set $\Br$ to be the number of simple poles of $q$, i.e., $\Br = 2g+1$ if $\cH = \cH^{hyp}(2g-2)$ and $\Br = 2g+2$ if $\cH = \cH^{hyp}(g-1, g-1)$.
\end{definition}

So long as the quadratic differential $q$ is nonsingular at infinity, it admits the explicit expression
\[
q = \frac{(z-z_0)^{2g-2}}{(z-z_1) \ldots (z-z_{\Br})} dz^2
\]
and we can take the positions of the singularities $\cS = \{z_0, \dots, z_{\Br}\}$ of $q$ as moduli for $\cH^{hyp}$.
This implies that the component $\cH^{hyp}$ is a $K(G,1)$ for (a finite extension of) the spherical braid group on $\Br+1$ strands, with one strand corresponding to the zero of $q$ and the other to its poles \cite[Section 1.4]{LM}.
Via the Birman--Hilden theory of symmetric mapping class groups (cf. \cite{MW_BirmanHilden} or \cite[$\S$ 9.4]{FarbMargalitMCG}), this spherical braid group can be identified with the centralizer of $\iota$ in the mapping class group of $X$ relative to the zero(es) $\cZ$ of $\omega$.
We therefore arrive at the following:

\begin{lemma}\label{lemma:HyperellipticTopMonodromy}
The topological monodromy group of a hyperelliptic stratum component where the zeros are unlabeled is the centralizer, denoted $\SMod(X,\cZ)$, of the hyperelliptic involution $\iota$.
\end{lemma}

As a consequence, the topological monodromy group of the cover on which the zeros are labeled is exactly
\[\PSMod(X, \cZ) := \PMod(X, \cZ) \cap \SMod(X, \cZ),\]
where the pure and symmetric mapping class groups are both thought of as subgroups of $\Mod(X, \cZ)$.
Note that when $X$ has two zeros, $\iota$ interchanges them and so is not an element of $\PMod(X, \cZ)$, hence we cannot describe $\PSMod(X, \cZ)$ as its centralizer.

For both $\cH^{hyp}(2g-2)$ and $\cH^{hyp}(g-1, g-1)$, fix the curves $c_i$ as shown in Figure \ref{fig:chainbasis}; then the hyperelliptic mapping class group $\PSMod(X,\cZ)$ corresponding to the hyperelliptic involution that setwise fixes each $c_i$ is generated by Dehn twists about the $c_i$'s \cite[$\S$ 9.4]{FarbMargalitMCG}.
While we will not explicitly use this generating set, the $c_i$ are a useful set of curves, as we will see in the proof of \Cref{corollary:level2hyp} below.

\begin{figure}
\centering
		\labellist
        \small
        \pinlabel $c_1$ at 15 65
        \pinlabel $c_2$ at 45 70
        \pinlabel $c_{2g}$ at 170 70
        \pinlabel $c_1$ at 290 65
        \pinlabel $c_2$ at 320 70
        \pinlabel $c_{2g+1}$ at 460 65
		\endlabellist
\includegraphics[scale=.6]{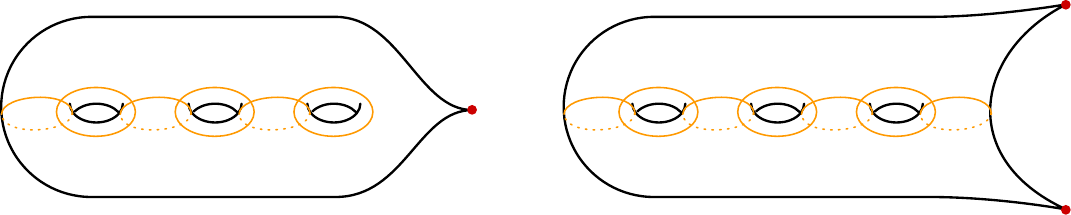}
\caption{Generators $\{c_i\}$ for hyperelliptic mapping class groups.}
\label{fig:chainbasis}
\end{figure}
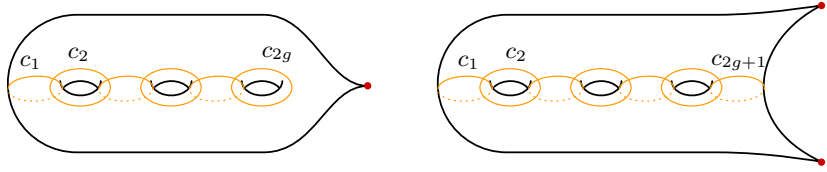

\subsubsection{Ordinary marked points} In our setting, we consider strata of differentials on which some number (possibly zero) of ordinary marked points are distinguished. Let $\cP$ denote this set of ordinary marked points, so that the full set $\cB$ of distinguished points is given by
\[
\cB = \cZ \sqcup \cP,
\]
with $\cZ$ consisting of the zeros, a set of size at most two (and at least one, in the regime $g \ge 2$).

For any partition $\kappa$ of $2g-2$, the stratum $\cH_{lab}(\kappa, 0^m)$ is a {fiber bundle} over the stratum $\cH_{lab}(\kappa)$, with fiber $\PConf_m(X \setminus \cB_{\kappa})$, where $\PConf_m(\cdot)$ denotes the space of ordered configurations of $m$ points on a space, and $\cB_\kappa$ is the distinguished set for the stratum $\cH_{lab}(\kappa)$. Note in particular that the components of $\cH_{lab}(\kappa)$ and $\cH_{lab}(\kappa, 0^m)$ are in natural bijective correspondence. 

On the level of monodromy, the topological monodromy groups of components of $\cH_{lab}(\kappa)$ and $\cH_{lab}(\kappa,0^m)$ are related by the following, which is an immediate consequence of the {\em Birman exact sequence}  \cite[Theorem 9.1]{FarbMargalitMCG}.

\begin{lemma}\label{lemma:forgetpoints}
Let $\cH \subset \cH_{lab}(\kappa)$ be a component of a stratum and let $\cH^+ \subset \cH_{lab}(\kappa, 0^m)$ be obtained from $\cH$ by marking regular points. Let $\cB_{\kappa}$ be the distinguished set of points for $X \in \cH$, and let $\cB^+ = \cB_\kappa \sqcup \cP$ be the distinguished set for $(X, \omega) \in \cH^+$. Let $\Gamma \leqslant \PMod(X, \cB_\kappa)$ and $\Gamma^+ \leqslant \PMod(X, \cB^+)$ be the topological monodromy groups for $\cH$ and $\cH^+$, respectively. Then there is a short exact sequence
\[
1 \to PB_m(X \setminus \cB_\kappa) \to \Gamma^+ \to \Gamma \to 1,
\]
where $PB_m(X \setminus \cB_\kappa) := \pi_1(\PConf_m(X \setminus \cB_\kappa))$ is the {\em surface braid group}. 
\end{lemma}

The group $PB_m(X \setminus \cB_\kappa)$ is generated by {\em point-pushing elements}: these are mapping classes associated to choosing a marked point $q \in \cB^+ \setminus \cB_\kappa$, and moving $q$ around $X \setminus \cB_\kappa$ while fixing all other marked points, eventually returning to $q$.

\subsection{The homological monodromy group}\label{subsection:monohyp}
We come to the main result of the section, a variant of a result of A'Campo \cite{ACampoTressesMonodromie}. To state it, recall from \Cref{section:monononhyp} that $\PAut(H_1(X,\cB;\Z))[2]$ is the kernel of the action on mod-2 relative homology. 

Relative to the short exact sequence \eqref{eqn:pautses}, this admits the following decomposition:
\[
1 \to \Hom(\tilde{H}_0(\cB; \Z), H_1(X; 2\Z)) \to \PAut(H_1(X,\cB;\Z))[2] \to \Sp(2g, \Z)[2] \to 1.
\]
We will also need to examine the kernel term more closely. The set $\cB$ of distinguished points has two constituents: zeros $\cZ$ and ordinary marked points $\cP$. This leads to a decomposition
\begin{equation}\label{eqn:bsplit}
\tilde H_0(\cB; \Z) \cong \tilde H_0(\cZ; \Z) \oplus H_0(\cP; \Z),
\end{equation}
with the latter summand identified with a subspace of $\tilde H_0(\cB; \Z)$ via the correspondence $[q] \in H_0(\cP;\Z) \mapsto [q] - [p_1] \in \tilde H_0(\cB;\Z)$ (here, $p_1 \in \cZ$ is a zero).

The next two lemmas produce certain useful subgroups of $\bar \Gamma$ in the hyperelliptic setting.

\begin{proposition}\label{corollary:level2hyp}
Let $\cH_{lab}^{hyp}$ be a hyperelliptic component of labeled translation surfaces of genus $g \ge 2$, possibly with some marked ordinary points. Then there is a containment
\[
\PAut(H_1(X,\cB;\Z))[2] \leqslant \bar \Gamma.
\]
\end{proposition}

\begin{proof}
To prove the result, we must show that $\bar \Gamma$ surjects onto $\Sp(2g, \Z)[2]$, and that $\bar \Gamma$ contains $\Hom(\tilde{H}_0(\cB; \Z), H_1(X; 2\Z))$.

The surjection of $\bar \Gamma$ onto $\Sp(2g, \Z)[2]$ follows readily from the literature, cf. \cite[Theorem 2.9]{AvilaMatheusYoccozRegHypZorichConj} or \cite{ACampoTressesMonodromie}. To establish the second claim, we will exhibit special ``point-pushing'' monodromy elements that induce the identity on absolute homology, {\em transvect} a chosen basis element $[x] - [p_1]$ to $[x] - [p_1] + 2y$ for $y \in H_1(X; \Z)$ arbitrary, and fix the remaining basis elements of $\tilde H_0(\cB; \Z)$. 

To show that every element of $H_1(X; 2\bZ)$ is achieved as a transvection of $[p_2]-[p_1]$, we will make use of the discussion of \Cref{section:hyperellipticmcg}. There, we gave a description of the topological monodromy group as lifted from a braid group on $\CP^1$. We will use this to directly exhibit point-pushes of $z_0$ on $\CP^1 \setminus \cS$ whose lifts to $X$ generate $H_1(X; 2\bZ)$.

Choose a sequence of pairwise disjoint simple arcs $\alpha_i$ connecting $z_i$ to $z_{i+1}$ for $i = 1, \ldots, 2g+1$; we can then take the set of every other arc $\alpha_1, \alpha_3, \ldots, \alpha_{2g+1}$ as branch cuts for the ramified cover $X \to \CP^1$.
Finally, choose simple arcs $\varepsilon_i$ whose interiors are disjoint from each other and all $\alpha_j$ and that connect the basepoint $z_0$ to $\alpha_i$. See Figure \ref{fig:sympush}.

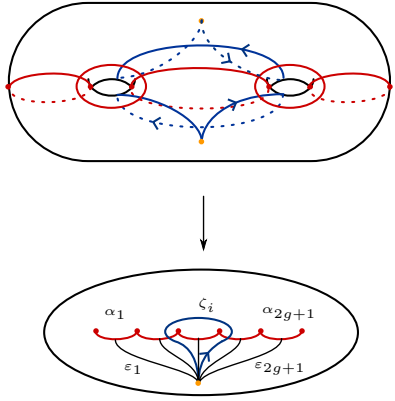
\begin{figure}[ht]
\centering
\labellist
\tiny
\pinlabel $\alpha_1$ at 60 45
\pinlabel $\alpha_{2g+1}$ at 155 45
\pinlabel $\epsilon_1$ at 70 15
\pinlabel $\epsilon_{2g+1}$ at 150 15
\pinlabel $\zeta_i$ at 110 50
\endlabellist
\includegraphics[scale=1]{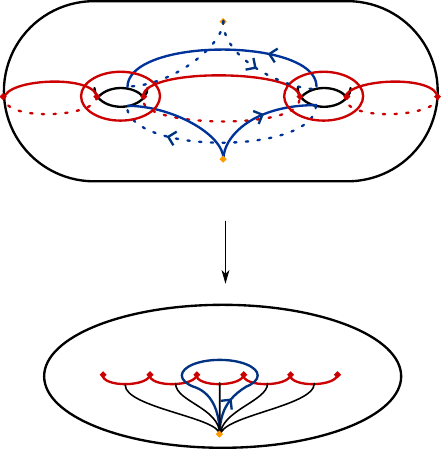}
\caption{Lifting a point-push via the hyperelliptic involution.}
\label{fig:sympush}
\end{figure}

The lifts of the $\alpha_i$ form a chain $c_1, \ldots, c_{2g+1}$ of simple closed curves on $X$, that is, a set of curves with intersection given by
\[i(c_i, c_j) = \left\{ \begin{array}{rl}
1 & \text{ if } |i-j| = 1\\
0 & \text{ otherwise.}
\end{array}\right.\]
In particular, $\{c_1, \ldots, c_{2g}\}$ is a basis for $H_1(X; \bZ)$.
Now consider the loop $\zeta_i$ on $\CP^1 \setminus \cS$ formed by traveling along $\epsilon_i$, going once around $\alpha_i$, and back along $\epsilon_i$. Pushing $z_0$ along $\zeta_i$ has the effect of pushing each of $p_1$ and $p_2$ (the zeros of the differential) once around $c_i$, but in opposite directions. Thus, the lift of $\zeta_i$ to $\Hom(\tilde{H}_0(\cB; \Z), H_1(X; 2\Z))$ transvects any relative homology class by $2[c_i]$.
Since the $\{c_i\}$ generate the integral homology of $X$ and every $\zeta_i$ can be realized as a loop of quadratic differentials, we see that $[p_2]-[p_1]$ can be transvected by an arbitrary element of $H_1(X; 2\bZ)$, completing the proof of the proposition.
\end{proof}

\begin{lemma}\label{lemma:ordinarypush}
Let $\cH_{lab}$ be a component of labeled translation surfaces of genus $g \ge 1$ (hyperelliptic or otherwise). With respect to the splitting \eqref{eqn:bsplit} of $\tilde H_0(\cB; \Z)$, there is a containment
\[
\Hom(H_0(\cP;\Z), H_1(X; \Z)) \leqslant \bar \Gamma.
\]
\end{lemma}
\begin{proof}
Following the discussion in \Cref{corollary:level2hyp}, we will exhibit further ``point-pushing'' transvections that send a chosen basis element $[x] - [p_1]$ to $[x]-[p_1] + y$ for $x \in \cP$ and $y \in H_1(X; \Z)$ arbitrary, and that fix the remaining basis elements of $\tilde H_0(\cB; \Z)$.

Consider a basis element $[q] - [p_1]$ for $q \in \cP$. Since $q$ is an ordinary marked point, it is free to move around $X$ along an arbitrary path that avoids the remaining points of $\cB$. Such a path can trace out an arbitrary homology class $y \in H_1(X; \Z)$, and so long as this path stays disjoint from the remaining arcs representing a basis of $\tilde H_0(\cB; \Z)$, it will preserve the associated homology classes.
\end{proof}

\section{Orbit counts (III): Split hyperelliptic case}\label{section:splithyp}

In this section, we complete our discussion of the orbits of $\overline{\Gamma}_d$ on $H_1(X, \mathcal B; \Z/d\Z)$ (and hence the path component structure of the corresponding space of covers $\cM_\delta$), by considering the case of hyperelliptic strata.
As in the non-hyperelliptic setting, \Cref{prop:oddporbits} and \Cref{corollary:level2hyp} together imply that we may restrict our attention to the case of mod-2 coefficients.
We will see that in the hyperelliptic case, orbits are classified by how many branch points of the involution are encircled by a nice representative.
We state our results in \Cref{stmth} after introducing this invariant.

\subsection{Encircling branch points: the $b$ invariant}\label{subsec:encircle}
In this subsection, we assume that $g \ge 2$ and so the set $\mathcal Z$ of zeros of $\omega$ is nonempty. Since Lemma \ref{lemma:forgetpoints} allows us to split off the action of the monodromy on regular points, we begin by analyzing the action on homology rel $\mathcal Z$.

In order to state our orbit classification, we first define an invariant which records how many branch points of $\iota$ that a curve encircles; see Figure \ref{fig:bw}. To define the $b$ invariant, let $a$ be a simple arc on $X$ with endpoints on $\cZ$ (possibly at the same point) with the property that $\iota(a)$ is disjoint from $a$. Consider the image $\bar a$ of $a$ on the $(\Br+1)$-punctured sphere $X/ \langle \iota \rangle$; by assumption it is a simple arc that begins and ends on the marked point corresponding to $\mathcal Z$. The image $\bar a$ therefore separates $X/ \langle \iota \rangle$ into two spheres, one of which contains
\[b(a)\in \{1, \ldots, \lfloor \Br/2 \rfloor\}\]
branch points and the other of which contains $\Br - b(a)$ branch points. 

\begin{definition}[$b$ invariant]\label{def:bw}
The {\em $b$ invariant} of the simple arc $a$ is the integer $b(a)$ following the discussion above.
\end{definition}

\begin{figure}[ht]
\centering
\includegraphics[scale=.6]{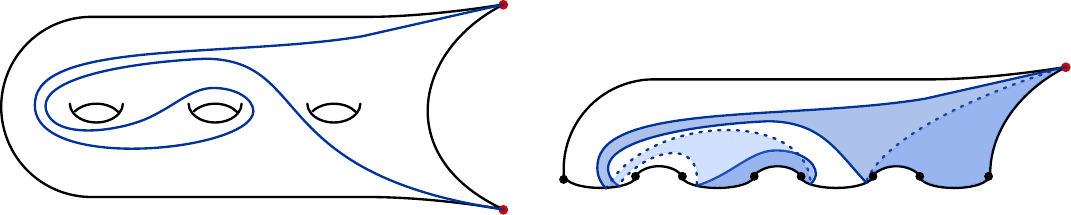}
\caption{An arc disjoint from its image under $\iota$ and how it separates the quotient $X/ \pair{\iota}$. In this example, $b(a) = 3$.}
\label{fig:bw}
\end{figure}

This invariant is not defined on every isotopy class, and its definition requires finding a nice representative in an isotopy class; we deal with this issue just below.
All the same, this topological definition of $b(a)$ makes it immediate that it is an invariant of $\PSMod(X, \cZ)$ orbits.

\begin{lemma}\label{lem:bw_inv}
Suppose that $a$ is an arc on $X$ with endpoints on $\mathcal Z$ and so that $\iota(a)$ is disjoint from $a$. Then for any $c \in \PSMod(X, \cZ) \cdot a$, we have that $b(c) = b(a)$.
\end{lemma}
\begin{proof}
By our discussion in Section \ref{section:hyperellipticmcg}, $\PSMod(X,\cZ)$ is isomorphic to a spherical braid group on $\Br+1$ strands (after modding out by $\pair{\iota}$ in the case of a single zero). Moreover, this identification is equivariant with respect to the action of $\PSMod(X, \cZ)$ on arcs on $X$ and the action of the braid group on arcs on $X / \langle \iota \rangle$.
The lemma now is a consequence of the fact that ``the number of distinguished points that an arc cuts off'' is a homeomorphism invariant and hence an invariant of the braid group action.
\end{proof}

We now extend our definition of $b(a)$ to all simple closed curves, essentially by counting the number of branch points $p$ so that the mod-2 winding number of $a$ about $p$ is nonzero.
We will not pursue this exact definition, instead giving a homological formulation that will be more convenient in the sequel.\\

\para{Encircling as intersection}
Recall the definition of the curves $c_i$ from Figure \ref{fig:chainbasis}. 
We observe that $\{ [c_i] \mid i \le \Br - 1 \}$ is a basis for the (mod 2) homology of $X \setminus \mathcal Z$; however, it will be convenient to instead consider a different basis given by
\[e_j := \sum_{i \le j} [c_i].\]
Representatives for these homology classes can be expressed by curves arranged in a star-like pattern; see Figure \ref{fig:fjint}.

We now construct a generalization of the $b$ invariant.
Define a map
\[\underline{e} : H_1(X, \mathcal Z; \bZ/2\bZ) \to (\bZ /2 \bZ)^{\Br-1}\]
by pairing $\balpha \in  H_1(X, \mathcal Z; \bZ/2\bZ)$ with each $e_j$.

\begin{definition}\label{def:bint}
For $\balpha \in  H_1(X, \mathcal Z; \bZ/2\bZ)$, set $b(\balpha)$ to be either the number of 0's or the number of 1's in the vector $\underline{e}(\balpha)$, whichever is smaller.
\end{definition}

Of course, since we have given a new definition for an old invariant, we must show that it coincides with our original one.

\begin{figure}[ht]
\centering
\labellist
    \tiny
    \pinlabel $e_1$ at 15 50
    \pinlabel $e_2$ at 50 60
    \pinlabel $e_{2g+1}$ at 160 85
\endlabellist
\includegraphics[scale=.6]{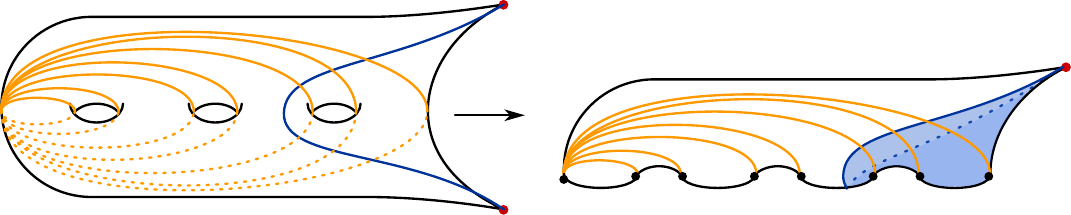}
\caption{Symmetric curves representing the classes $e_j$ and their intersection with an arc. In this example, $b(a) = 3$ and the subsurface $U$ is shaded.}
\label{fig:fjint}
\end{figure}

\begin{lemma}\label{lem:bdefsagree}
Let $a$ be a simple arc on $X$ with endpoints on $\mathcal Z$ so that $\iota(a)$ is disjoint from $a$. Then $b(a) = b([a])$.
\end{lemma}
\begin{proof}
The proof of this lemma essentially follows from inspection of Figure \ref{fig:fjint}. A similar figure exists for $\cH^{hyp}(2g-2)$; we leave its construction to the reader.
More explicitly, in either case, consider the (isotopy classes of) curves
\[T_{c_j} T_{c_{j-1}} \ldots T_{c_2} (c_1)\]
representing the classes $e_j$. These curves (have representatives which) are symmetric under (a representative for) the hyperelliptic involution $\iota$, and their quotients are a set of arcs $\varepsilon_1, \ldots, \varepsilon_{\Br-1}$. Each $\varepsilon_i$ has one endpoint at a distinguished branch point $p_0$; enumerate the other branch points as $p_1, \ldots, p_{\Br-1}$, so that $\varepsilon_i$ has endpoints $p_0$ and $p_i$.

Now the image $\bar{a}$ of $a$ on $X / \pair{\iota}$ separates the surface into two pieces; consider the subsurface $U$ not containing $p_0$. For $i \ge 1$, we observe that $p_i \in U$ if and only if the geometric intersection number of $\epsilon_i$ with $\bar a$ is odd. 
The proof is concluded by noting that for each $i$, the geometric intersection number of $\epsilon_i$ and $\bar a$ is the same as the geometric intersection number of $T_{c_i} T_{c_{i-1}} \ldots T_{c_2} (c_1)$ and $a$, which has the same parity as the (algebraic) intersection number of $e_j$ with $[a]$.
\end{proof}

\subsection{Statement of results}\label{stmth}

As in \Cref{subsection:monohyp}, when $g \ge 2$ we may write $\mathcal{B} = \mathcal Z \cup \mathcal P$ where $\mathcal Z$ are zeros and $\mathcal P$ are regular marked points and split
\[\tilde H_0(\mathcal B; \bZ/d\Z) \cong
\tilde H_0({\mathcal Z}; \bZ/d\Z) \oplus H_0(\mathcal P; \bZ/d\Z).\]
Let $\pi$ denote the projection to the second factor and $H_{\mathcal Z} = \ker(\pi)$.

It turns out that the invariant $b$ defined above is in fact a complete invariant of the $\bar \Gamma_{2^k}$ action on $H_{\mathcal Z}$. Ultimately, we will establish the following:
\begin{proposition}[$p = 2$ orbit classification, (hyper)elliptic]\label{prop:2splitorbitshyp}
Let $\cH_{lab}^{hyp}$ be a (hyper)elliptic stratum component of labeled translation surfaces in genus $g \ge 1$, possibly equipped with some ordinary marked points. Let $\balpha, \bbeta \in H_1^{prim}(X, \mathcal B; \Z/2^k\Z)$ be given, and suppose that $\delta(\balpha) = \delta(\bbeta) \in \tilde H_0(\mathcal B; \Z/2^k\Z)$. 
\begin{itemize}
\item If $g=1$, then there is a single (unsplit) orbit: 
$\bar \Gamma_{2^k} \balpha = \bar \Gamma_{2^k} \bbeta$.
\item If $g \ge 2$ and $\pi(\delta(\balpha))$ is nonzero, then there is a single (unsplit) orbit.
\item Otherwise, $\bar \Gamma_{2^k} \balpha = \bar \Gamma_{2^k} \bbeta$ if and only if $b(\balpha)=  b(\bbeta)$.
\end{itemize}
\end{proposition}

\begin{remark}
The case $g=1$ is easy and does not require the technology developed in this section (and $\pi$ is not even defined when $g=1$), but since it uses the same argument as the unsplit hyperelliptic case, we have chosen to include it here.
\end{remark}

\begin{remark}
In the case where $\cH= \cH^{hyp}(g-1,g-1,0^{n-2})$ and $\pi(\delta)=0$, the parity of $b(\balpha)$ determines whether $\balpha$ is actually an absolute class (even) or a strictly relative one (odd).
So whereas the monodromy group of $\cH^{hyp}(2g-2,0^{n-1})$ acts on mod $2^k$ absolute homology with $g$ orbits of primitive vectors, the monodromy group of $\cH^{hyp}(g-1, g-1,0^{n-2})$ acts on the mod $2^k$ reduction of $H_{\mathcal Z}$ with $g+1$ orbits, exactly $\lceil g/2 \rceil$ of which are absolute and $\lfloor g/2 \rfloor + 1$ of which are relative.
\end{remark}

To deal with the fact that the complement of a subset of $2g+2$ elements of size $g+1$ also has size $g+1$, define the following correction factor: set 
\[
\sigma(\balpha) = \begin{cases}
   1 & \cH = \cH^{hyp}(g-1,g-1,0^{n-2}) \mbox{ and }b(\balpha) = g+1 \\
   0 & \mbox{ otherwise. }
\end{cases}
\]

\begin{proposition}[$p= 2$ split orbit sizes, hyperelliptic]\label{prop:2splitsizehyp}
With notation as in Proposition \ref{prop:2splitorbitshyp}, suppose that $g \ge 2$ and that $\balpha \in H_1^{prim}(X,\mathcal B; \Z/2^k\Z)$ satisfies $\pi(\delta(\balpha))=0$, so the orbit of $\balpha$ is split and the $b$ invariant is well-defined.
Then
\[\abs{\bar \Gamma_{2^k}\balpha} = 2^{2g(k-1)-\sigma(\balpha)} \binom{Br}{b(\balpha)}.\]
\end{proposition}

\begin{proposition}[$p=2$ split orbit statistics, hyperelliptic]\label{prop:2splitstatshyp}
With notation as in Propositions \ref{prop:2splitorbitshyp} and \ref{prop:2splitsizehyp}, let $c \in H_1(X \setminus \mathcal B; \Z/2^k\Z)$ be the homology class of a cylinder on $(X, \omega)$, and let $a \in \Z/2^k\Z$ be given. 
Suppose $\pi(\delta(\balpha))=0$, so that the orbit of $\balpha$ is split and the $b$ invariant is well-defined.
Then
\[
\frac{
\abs{(\bar \Gamma_{2^k} \balpha)_{c;a}}
}
{2^{(2g-1)(k-1)-\sigma(\balpha)}}
 = \begin{cases}
 2 \binom{Br-2}{b(\balpha)-1} & a=1 \pmod 2\\
\binom{Br-2}{b(\balpha)} + \binom{Br-2}{b(\balpha)-2} & a=0 \pmod 2.
\end{cases}\]
\end{proposition}

\begin{remark}
There was an arbitrary choice made in our definition of $b(\balpha)$: it would make just as much sense to take instead $\Br - b(\balpha)$.
In the topological language of Definition \ref{def:bw}, this is because if an arc cuts off $b(\balpha)$ branch points on one side, then the other side contains $\Br - b(\balpha)$ branch points.
We observe that making this substitution in Propositions \ref{prop:2splitsizehyp} and \ref{prop:2splitstatshyp} above results in the same formulas by the symmetry of the binomial coefficients.
\end{remark}

\subsection{Proofs}

\subsubsection{Orbit classification}

We first prove a useful result about realizing homology classes by nice arcs; this allows us to freely pass back and forth between our definitions of the $b$ invariant.

\begin{lemma}\label{lem:realizebyarc}
For any class $\balpha \in H_{\mathcal Z}$ with mod-2 coefficients, there is a simple arc $a$ on $X$ with endpoints on $\mathcal Z$ so that $\iota(a)$ is disjoint from $a$ and $[a] = \balpha$.
\end{lemma}
\begin{proof}
By the duality between relative and excision homology, and the fact that the $e_j$ curves span $H_1(X \setminus \mathcal Z; \bZ/2\bZ)$, it suffices to build an arc $a$ with specified intersection with each $e_j$.
One can do this explicitly as follows: the classes $e_j$ can be represented by curves as in the proof of Lemma \ref{lem:bdefsagree} that project down to arcs $\varepsilon_j$ on $X / \pair{\iota}$.
The $\varepsilon_j$ are arranged in a ``star'' pattern, in that they all share one common endpoint and are otherwise disjoint.
Thus, one can build a simple arc $a'$ on $X / \pair{\iota}$ with endpoints the (unique) marked point corresponding to $\mathcal{Z} / \pair{\iota}$ that meets any specified subset of the $\varepsilon_j$ exactly once and is disjoint from the rest. See Figure \ref{fig:constructarc}.
Either choice of lift of this arc to $X$ yields an arc $a$ with (not necessarily distinct) endpoints on $\mathcal{Z}$. Necessarily $a$ and $\iota(a)$ are disjoint (except possibly at common endpoints) since $a$ projects down to a simple arc on $X / \pair{\iota}$. Note finally that by construction, $a$ has the specified intersection with each $e_j$.
\end{proof}

\begin{figure}[ht]
\centering
\includegraphics[scale=.6]{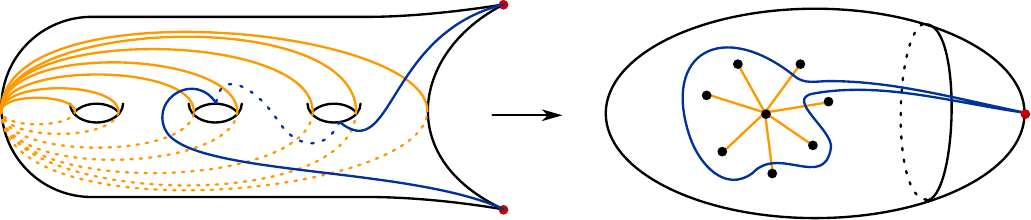}
\caption{Constructing an arc with specified intersection with the $\varepsilon_j$ arcs.}
\label{fig:constructarc}
\end{figure}

\begin{proof}[Proof of \Cref{prop:2splitorbitshyp}]
By \Cref{prop:oddporbits} and \Cref{corollary:level2hyp}, we need only consider the mod-$2$ orbits, not the mod-$2^k$ ones.
If either $g = 1$ or $g \ge 2$ and $\pi(\delta(\balpha)) \neq 0$, then \Cref{lemma:ordinarypush} allows us to transvect arbitrarily along the fibers of $\delta$, so $\bar \Gamma_2$ acts with unsplit orbits.

So suppose $g \ge 2$ and $\pi(\delta(\balpha)) = 0$ mod $2$, i.e., $\balpha \in H_{\mathcal Z}$. Fix any $F \in \bar \Gamma_2$.
Using Lemma \ref{lem:realizebyarc}, choose some arc $a$ with mod-$2$ homology class $\balpha$ so that $\iota(\balpha)$ is disjoint from $\balpha$, and choose any $f \in \PSMod(X, \cB)$ that induces the $F$ action on relative homology. 
Then Lemmas \ref{lem:bdefsagree} and \ref{lem:bw_inv}, respectively, imply that 
\[b(F \balpha) = b(f(a)) = b(a) =b(\balpha),\]
showing that $b$ is an invariant of the $\bar \Gamma_2$ action.

Now suppose that $\bbeta \in H_1^{prim}(X, \mathcal B; \Z/2\Z)$ with $\delta(\bbeta) = \delta(\balpha)$.
Then $\pi(\delta(\bbeta)) = 0$, and so applying Lemma \ref{lem:realizebyarc} again, we can represent $\bbeta$ by some arc $b$ that is disjoint from $\iota(\bbeta)$.
But now $a$ and $b$ both project to simple arcs $\bar a$ and $\bar b$ on the quotient $X / \pair{\iota}$ that cut off $b(\balpha) = b(\bbeta)$ branch points on each side.
By the change-of-coordinates principle \cite[Section 1.3]{FarbMargalitMCG}, there is thus some homeomorphism
 of $X / \pair{\iota}$ rel the branch points, i.e., some braid, taking $\bar a$ to $\bar b$.
Lifting that element gives an element in the hyperelliptic mapping class group taking $a$ to $b$; hence, its induced action on homology must take $\balpha$ to $\bbeta$. Thus, $\bar \Gamma_2$ acts transitively on the set of $\balpha$ with fixed $b(\balpha)$.
\end{proof}

\subsubsection{Orbit sizes}

\begin{proof}[Proof of \Cref{prop:2splitsizehyp}]
When $k=1$, we are simply counting the number of primitive $\balpha$ with fixed $\delta(\balpha)$ and $b(\balpha)$.
Using Definition \ref{def:bint}, the number $b(\balpha)$ is either the number of $e_j$ so that $(\balpha, e_j) \neq 0$ or $\Br$ minus that number.
To count the number of $\balpha$ with fixed $b(\balpha)$, then, we simply need to count the number of ways to have nonzero pairing with either $b(\balpha)$ or $\Br- b(\balpha)$ of the $e_j$, which is
\[\binom{\Br-1}{b(\balpha)} + \binom{\Br-1}{\Br - b(\balpha)} 
= \binom{\Br-1}{b(\balpha)} + \binom{\Br-1}{b(\balpha)-1} 
= \binom{\Br}{b(\balpha)}.\]

In the case when $\Br = 2g+2$ and $b(\balpha)=g+1$, we observe that we must also multiply by a factor of 1/2: this is because choosing $g+1$ elements also picks out $g+1$ elements in the complement, and if we have nonzero intersection with the chosen set then we have zero intersection with each $e_j$ in the complement. So since $b(\balpha)$ is defined in terms of the minimum of the 0's and 1's in the vector $\underline{e}(\balpha)$, we see in this case that $\binom{\Br}{b(\balpha)}$ is an overcount by a factor of $2$.

To pass from the mod 2 to the mod $2^k$ case, we just observe there are $2^{2g(k-1)}$ mod $2^k$ classes with given mod 2 reduction.
\end{proof}

\subsubsection{Orbit statistics}

\begin{proof}[Proof of \Cref{prop:2splitstatshyp}]
By \Cref{prop:hypcyltrans}, $\bar \Gamma$ acts transitively on the homology classes of core curves of cylinders. We are therefore free to choose any convenient representative: compare \Cref{CylInv:Lemma}.
Without loss of generality, we will therefore assume that the curve $c = c_1$ (which represents the homology class $e_1$) is a cylinder.

As in the proof of \Cref{prop:2splitsizehyp}, we compute first for $k=1$. 
We want to count the number of ways to have nonzero pairing with either $b(\balpha)$ or $\Br- b(\balpha)$ of the $e_j$, constrained so that the intersection with $c$ has fixed parity.
In the case that $\balpha(c) = 0 \pmod 2$, then among the other $\Br-2$ of the $e_j$'s, we count that there are
\[\binom{\Br-2}{b(\balpha)} + \binom{\Br-2}{\Br - b(\balpha)} =
\binom{\Br-2}{b(\balpha)} + \binom{\Br-2}{ b(\balpha)-2}\]
ways for this to happen.  If we assume $\balpha(c) = 1 \pmod 2$, then there are 
\[\binom{\Br-2}{b(\balpha)-1} + \binom{\Br-2}{\Br - b(\balpha)-1} =
2\binom{\Br-2}{b(\balpha)-1}\]
possible ways to distribute the remainder of the intersections.
As before, if $\Br = 2g+2$ and $b(\balpha)=g+1$, then we must also introduce an extra factor of $1/2$ due to the fact that $b(\balpha) = \Br-b(\balpha)$. 

To get the final statement for arbitrary $k$, we note that the number of mod-$2^k$ classes that have specified intersection with $c$ and that lie over a given mod-2 class is exactly $2^{(2g-1)(k-1)}$.
\end{proof}

\begin{remark}
Both Propositions \ref{prop:2splitsizehyp} and \ref{prop:2splitstatshyp} can also be proven by counting the number of ways to partition the branch points into two sets, perhaps subject to the constraint that the arc $\varepsilon_1$ connects zeros in different parts.
We leave it as an exercise to the reader to reprove both statements in this language.
\end{remark}

\part{Computing the Siegel--Veech constants}
\label{SiegelVeech:Part}

\para{Outline of Part~\ref{SiegelVeech:Part} (Sections \ref{SVBackSetSect}-\ref{subsection:mainformula})}
\begin{itemize}
\item In Section~\ref{SVBackSetSect}, we collect all of the necessary background on Siegel--Veech transforms and constants.
\item In Section~\ref{SVConstsCyclicCovsSection}, we compute the Siegel--Veech constants in the main formula in terms of ``orbit quotient functions'' under the assumption that these functions satisfy several number theoretic properties.
\item In Section~\ref{OrbitQuotientFcn:Sect}, we use the results of Part~\ref{OrbitCard:Part} to derive explicit formulas for the orbit quotient functions in all cases, thereby showing that all of the number theoretic assumptions are satisfied.
\item In Section~\ref{subsection:mainformula}, we present the main formula for the ratios of Siegel--Veech constants with a guide to all of the notation involved.
This section also combines all of the previous results to prove Theorem~\ref{MainTheoremSVFormula} and its corollaries from the Introduction.
\end{itemize}

\section{Siegel--Veech preliminaries}
\label{SVBackSetSect}

\subsection{The Siegel--Veech transform}
We begin in the most general context before specifying to strata.
Let $\cM_A$ be an $\splin$-invariant subvariety of area $A$ translation surfaces and let $\nu_A$ denote the $\splin$-invariant measure on $\cM_A$ induced from the affine measure on (the relevant linear subspaces of) period coordinates (this is unique up to scale).
The starting point for the Siegel--Veech transform is a choice of {\em Siegel--Veech measure,} an $\splin$-equivariant map
\[\cV: \cM_A \rightarrow \{\text{measures on }\R^2 \text{ with discrete support}\}\]
satisfying a certain list of further axioms (see \cite[$\S$2]{EskinMasurAsymptForms} or \cite{VorobetsPerGeodsGenTransSurfs}).
For example, if one is interested in counts of cylinders, or counts of cylinders weighted by $\text{area}^s$, 
one may consider the Siegel--Veech measures
\begin{equation}\label{eqn:SVexamples}
\cV_{cyl}: (X,\omega) \mapsto 
\sum_{\text{cylinders}} \delta_{\text{hol}(cyl)}
\hspace{1em}
\text{and}
\hspace{1em}
\cV_{\text{area}}^{s}: (X,\omega) \mapsto 
\sum_{\text{cylinders}}
\frac{\text{area}^{s}({cyl})}{\text{area}^{s}(X, \omega)} \delta_{\text{hol}(cyl)}
,    \end{equation}
respectively, where the sums are taken over all cylinders on $(X, \omega)$. Here $\delta$ denotes a point mass. Observe that $\cV_{cyl} = \cV_{\text{area}}^0$.

Given a function $f:\bR^2 \rightarrow \bR$ with compact support, the {\em Siegel--Veech transform} of $f$ with respect to $\cV$ is a function $\hat f: \cM_A \rightarrow \bR$ defined by
\[\hat f(X, \omega) := \int_{\bR^2} f(\bv) \,d \cV(X, \omega)(\bv).\]
In the case that $\cV$ does not come with weights (i.e., is a sum of {\em unit} point masses), this reduces to a sum of the values of $f$ over the support of $\cV(X,\omega)$, which is how the Siegel--Veech transform is most usually written.

The {\em Siegel--Veech formula} \cite{VeechSiegelMeasures} relates the integral of the Siegel--Veech transform of a function over $\cM_A$ to the integral of the original function:
\[\frac{1}{\nu_A(\cM_A)} \int_{\cM_A} \hat f(X, \omega)\,d\nu_A = c_{\cV}(\cM_A) \int_{\bR^2} f(\bv) \,d\text{Leb}_{\bR^2}(\bv),\]
where $c_{\cV}(\cM_A)$ is a constant depending only on the choice of Siegel--Veech measure $\cV$ and not the function $f$. This is the {\em Siegel--Veech constant} associated to $\cV$.
The main theorem of \cite{EskinMasurAsymptForms} shows that this constant 
captures the asymptotics of the corresponding weighted count for $\nu_A$-almost every $(X, \omega) \in \cM_A$:
\begin{equation}\label{eqn:SVLimitEquation}
c_{\cV}(\cM_A) = 
\lim_{r \rightarrow \infty}
\frac{1}{\pi r^2} \int_{\bR^2} \chi_r(\bv) \,d\cV(X, \omega)(\bv),
\end{equation}
where $\chi_r$ is the indicator function of the ball of radius $r$ about $0$.
For example, in the cases of \eqref{eqn:SVexamples}, the associated constants give the asymptotic count (for generic surfaces) of cylinders whose circumference is at most $r$ (the {\em cylinder Siegel--Veech constant} $c_{cyl}$), and the same count weighted by the $\text{area}^s$ of the cylinders (the {\em $\text{area}^s$-Siegel--Veech constant} $c_{\text{area}^s}$), respectively.  In particular, the counting functions from the introduction can be expressed as
\begin{equation}\label{eqn:N_Count_SV}
N_{cyl}((X, \omega), L) = \int_{\bR^2} \chi_L(\bv) \,d\cV_{cyl}(X, \omega)(\bv)
\hspace{1em}
\text{and}
\hspace{1em}
N_{\text{area}^s}((X, \omega), L) = \int_{\bR^2} \chi_L(\bv) \,d\cV_{\text{area}}^{s}(X, \omega)(\bv).
\end{equation}

\para{Linearity}
Since Siegel--Veech constants $c_{\cV}$ are defined in terms of integrals, they behave nicely under linear combinations.
In particular, if $\{\cV_m\}$ are different maps on $\cM_A$ with respect to which we can perform a Siegel--Veech transform, then
\begin{equation}\label{eqn:SVconstantsadd}
c_{\cV} = \sum a_m c_{\cV_m},
\text{ where } 
\cV = \sum a_m \cV_m
\text{ for }
a_m \in \bR.    
\end{equation}

\subsection{Cusp volumes and configurations}
\label{subsec:SVandvols}
We now specialize to strata and indicate how this discussion generalizes to invariant subvarieties.
Fix $\rad >0$ and let $\chi_{\rad}$ be the characteristic function of the disc of radius $\rad$ centered at the origin in $\bR^2$; then the Siegel--Veech formula states that
\begin{equation}\label{eqn:SVform mcyl}
c_{cyl}(\cH_{1, lab}) = 
\frac{1}{\pi \rad^2} \frac{1}{\nu_1(\cH_{1,lab})}
\int_{\cH_{1,lab}} \hat \chi_\rad(X, \omega) \, d\nu_1. 
\end{equation}
where $\hat \chi_\rad$ is the Siegel--Veech transform with respect to $\cV_{cyl}$.
A key insight of \cite{EskinMasurZorich} is that while it is not possible to compute this exactly for any particular value of $\rad$, we can instead take a limit where $\rad$ tends to zero and consider the asymptotics of the integral.

Define the {\em $\rad$-cylinder cusp} $\cH_{1,lab}^{\rad}$ to be the set of all $(X, \omega)$ that contain a cylinder whose holonomy has length at most $\rad$. Then so long as $\rad$ is small enough, the Siegel--Veech transform $\hat \chi_\rad$ is supported entirely on $\cH_{1,lab}^{\rad}$.
Moreover, $\hat \chi_\rad$ is identically $1$ on most of the cusp.
Following \cite{EskinMasurZorich}, for fixed $\rad' = O(\rad)$, we define the ($\rad'$-){\em thin part} of the $\rad$-cusp
\[\cH_{1,lab}^{\rad, thin} \subset \cH_{1,lab}^{\rad}\]
to be the set of surfaces which have a saddle connection of length $\le\rad'$ that is not homologous to the core curve of the $\rad$-thin cylinder.
The precise value of $\rad'$ will not particularly matter, and it will be convenient to vary at will (always at the scale of $\rad$).
The ($\rad'$-){\em thick part} of the $\rad$-cusp is the complement of the thin part.

One can also define a decomposition of the cusp of arbitrary invariant subvarieties: in this case, the thin part of the $\rad$-cylinder cusp consists of those surfaces which have an $\rad'$-short saddle connection that does not remain parallel to the $\rad$-short cylinders as one deforms in a small neighborhood in $\cM$ (compare the notion of ``$\cM$-parallel'' in \cite{WrightCylDef}).

As recorded in \cite[Corollary 7.2]{EskinMasurZorich}, work of \cite{MasurSmillieHausdDim} implies that the volume of the cusp is concentrated in the thick part (and a similar result is true for invariant subvarieties, see \cite{AvilaMatheusYoccozRegAIS}):
\[\nu_1\left( \cH_{1,lab}^{\rad, thin} \right) = o(\rad^2).\]
In \cite[Lemma 7.3]{EskinMasurZorich} it is shown that the integral of $\hat \chi_\rad$ over the thin part is also negligible:
\[\int_{\cH_{1,lab}^{\rad, thin}} \hat \chi_\rad \, d\nu_1 = o(\rad^2).\]
This fact also holds for $\text{area}^s$-weighted Siegel--Veech transforms, and corresponding results also hold for arbitrary invariant subvarieties.

Combining these two estimates and the fact that $ \hat \chi_\rad$ is identically $1$ on the thick part, one gets that
\begin{equation}\label{eqn:restricttothick}
\int_{\cH_{1,lab}} \hat \chi_\rad \, d\nu_1
= \int_{\cH_{1,lab}^{\rad}} \hat{\chi}_{\rad} d\nu_1
= \int_{\cH_{1,lab}^{\rad, thick}} d\nu_1 + o(\rad^2).
\end{equation}
Plugging \eqref{eqn:restricttothick} back into \eqref{eqn:SVform mcyl} and taking the limit as $\rad$ tends to zero, one recovers \cite[Proposition 3.3]{EskinMasurZorich}:
\[
c(\cH_{1, lab}) 
= \lim_{\rad \to 0 }
\frac{1}{\pi \rad^2}
\frac{\nu_1(\cH_{1,lab}^{\rad})}{\nu_1(\cH_{1,lab})}.
\]

\para{Configurations}
Of course, in order for the scheme above to be useful, one needs to be able to compute the volume of the cusp. Moreover, to compute $\text{area}^s$ Siegel--Veech constants, one must also be able to integrate other functions over these domains.
A large portion of the hard, technical work of \cite{EskinMasurZorich} is to give local models for (the thick parts of) cusps that facilitate this integration; see that paper for a much more thorough discussion of the concepts we sketch below.

Any surface $(X, \omega)$ in the thick part $\smash{\cH_{1,lab}^{\rad, thick}}$ of the $\rad$-cusp contains a unique set of short homologous cylinders. 
Cutting out those cylinders, one is left with a subsurface with ($\rad$-short) geodesic boundary. 
The remaining boundary components can then be contracted to points to yield a (possibly disconnected) translation surface in the ``principal boundary.''
While this contraction process is generally not canonical, different choices result in only slightly different surfaces in the boundary, allowing one to model (a large measure subset of) $\smash{\cH_{1,lab}^{\rad, thick}}$ by surfaces in a boundary stratum together with data of how to glue the cylinders back in. Making this picture precise takes a great deal of work and some very involved computations; see \cite[Parts I and II]{EskinMasurZorich}.

What is important for us is that different components of the thick part of the cusp correspond to different boundary strata and different ways of gluing in cylinders.
All these data are encapsulated in the notion of a {\em configuration} $\cC$ \cite[\S3.2]{EskinMasurZorich}; we use $\smash{\cH_{1,lab}^{\rad, thick}}(\cC)$ to denote the component of the the thick part of the cusp where the cylinders in the configuration are thin.
In particular, we can partition the set of cylinders on $(X,\omega)$ by the combinatorial data of their configuration.
We can then introduce Siegel--Veech measures 
\[\cV_{\cC}: (X,\omega) \mapsto 
\sum_{\substack{\text{cylinders in} \\ \text{configuration }\cC}} \delta_{\text{hol}(cyl)}\]
as well as the corresponding $\text{area}^s$ weighted versions $\cV_{\cC}^s$ and their respective Siegel--Veech transforms and constants.
By definition and \eqref{eqn:restricttothick}, given any $f \in C_c(\bR^2)$, we have that
\[\int_{\cH_{1,lab}^{\rad}} \hat f^\cC d \nu_1
= \int_{\cH_{1,lab}^{\rad,thick}(\cC)} \hat f d\nu_1 + o(\rad^2),\]
where the Siegel--Veech transforms on the left and right are taken with respect to $\cV_{\cC}$ and $\cV_{cyl}$, respectively. A similar equation also holds for the area$^s$ versions.
In light of linearity \eqref{eqn:SVconstantsadd}, we have that
\[
c_{cyl}(\cH_{1,lab}) = \sum_\cC c_{\cC}(\cH_{1,lab})\
\text{ and }\
c_{\text{area}^s}(\cH_{1,lab}) = \sum_\cC c_{\cC}^s(\cH_{1,lab}),
\]
a fact that we will use many times below.

\section{From integrals to orbit quotients}
\label{SVConstsCyclicCovsSection}

In this section we prove a preliminary result, Theorem~\ref{GeneralFormula}, which facilitates the computation of the Siegel--Veech constants in Theorem~\ref{MainTheoremSVFormula} under certain number theoretic assumptions.
In \Cref{OrbitQuotientFcn:Sect} below, we verify that the requisite assumptions are always satisfied using the explicit formulas computed in Part~\ref{OrbitCard:Part}.

To state the theorem, we define two concepts from elementary number theory for which we are unaware of standard terminology, namely, prime power independent functions and $D$-GCD dependent functions.

\begin{definition}
The \emph{radical of an integer} $\text{rad}(N)$ is its largest square-free divisor.
\end{definition}

\begin{definition}
We say that a multiplicative function $f: \bN \to \bN$ is \emph{prime power independent}, if $f(N) = f(\text{rad}(N))$ for all $N$.
\end{definition}

\begin{example}
A well known example of a non-trivial prime power independent function is $\Phi(N)/N$.
\end{example}

\begin{definition}
A function $f: \bN \to \bN$ is \emph{$D$-GCD dependent}, if $f(N) = f(\gcd(N, D))$ for all $N$.
\end{definition}

\begin{definition}
Given the group $\overline{\Gamma}_D$ defined in Section~\ref{MonodromyVectorAction}, a branching vector $\balpha$, and a number $N$, define the \emph{orbit quotient function}
$$f_{\overline{\Gamma}_D \balpha}(N) = \frac{|(\overline{\Gamma}_D \balpha)_N|}{|\overline{\Gamma}_D \balpha|}.$$
We abuse notation and write $f_D(N)$ when the group and branching vector are understood.
\end{definition}

Recall from the Introduction that the \emph{monodromy cofactor} $d_{abs} \in \bN$ is the maximal factor of $d$ coprime to $\gcd(\delta(\balpha))$ and $d_{rel} = d/d_{abs}$.
We further factor $d = d_{abs}' d_{rel}'2^k $ so that $d_{abs}'|d_{abs}$ and $d_{rel}'|d_{rel}$ are odd.

\begin{theorem}
\label{GeneralFormula}
Let $s \geq 0$.
Let $\cM_{1,\delta}(\balpha)$ be the component of a locus of branched cyclic covers over $\cH_{1}$ specified by the branching vector $\balpha$, normalized to have unit area.
Suppose that for all $D | d$, 
\begin{itemize}
\item $f_D$ is $D$-GCD dependent,
\item $f_D$ can be expressed as a constant multiplied by a prime power independent function,
\item $f_{d_{abs}'}$ is a constant, which we denote by $C_{d_{abs}'}$, and
\item If the orbit of $\balpha$ is unsplit and $2|d_{abs}$, then $f_{2^k}(a)$ is a constant function.
\end{itemize}
With these assumptions, we have
\[\frac{c_{\text{area}^{s}}(\cM_{1, \delta}(\balpha))}{c_{\text{area}^{s}}(\cH_1)} 
= \frac{1}{d} 
\left(\sum_{\fd|d_{rel}'} f_{d_{rel}'}(\fd) \Phi\left(\frac{d_{rel}'}{\fd}\right) \fd^{3-s}\right) 
\left(C_{d_{abs}'} \sum_{\fd|d_{abs}'} \Phi\left(\frac{d_{abs}'}{\fd}\right) \fd^{3-s} \right)
\left( \sum_{\fd|2^k} f_{2^k}(\fd) \Phi\left(\frac{2^k}{\fd}\right) \fd^{3-s} \right).\]
\end{theorem}

\subsection{Siegel--Veech constants of cyclic covers}

If a translation surface does not cover a translation surface of lower genus, then the integrals involved in evaluating its Siegel--Veech constants must be performed directly on the orbit closure that contains the translation surface.
However, there is a convenient shortcut one can take when considering covers.
Rather than evaluate integrals in the space of the cover, a counting function can be defined to capture the finite combinatorial information of the covering; 
computing the Siegel--Veech constants for a locus of covers therefore reduces to evaluating an integral on the subvariety containing the \emph{base} translation surface.  This concept was implemented in \cite{EskinMasurSchmollRectBar}.  Later it was used to compute area-Siegel--Veech constants of all square-tiled surfaces in \cite[$\S$10]{EskinKontsevichZorich2}, where the base surface is a once punctured torus.
Subsequently, this strategy was used in \cite[Part I]{ChenMollerZagierQuasiModSiegelVeech}, where the base was an $n$-punctured torus. 
Here we generalize this technique to base surfaces of arbitrary genus, but restrict to cyclic covers.

\begin{remark}
Most of the manipulations in this section work for arbitrary $G$-covers, replacing the numerical quantities with numbers and sizes of cosets, but in order to simplify the formula one needs to understand the $\pi_1(\cH_{1,lab})$ action on $\Hom(\pi_1(X \setminus \cB), G)$, which is generally a very hard problem.
\end{remark}

\para{Defining and decomposing the Siegel--Veech constant}
Following the strategy above, we define a counting function on our base translation surface adapted to the counting problem on the cover. 
We begin with our locus of (area $d$) cyclic covers $\cM_{d, \delta}(\balpha)$ defined above covering a stratum component $\cH_{1, lab}$.
From the $\text{area}^s$ Siegel--Veech formula for $\cM_{d, \delta}(\balpha)$, we can compute the $\text{area}^s$ Siegel--Veech constant of $\cM_{d, \delta}(\balpha)$ by integrating the $\cV_{\text{area}}^{s}$ Siegel--Veech transform of the indicator function $\chi_\rad$ of a ball of radius $\rad$:
\[c_{\text{area}^{s}}(\cM_{d, \delta}(\balpha))
= 
\frac{1}{\pi \rad^2} \frac{1}{\nu_d(\cM_{d, \delta}(\balpha))}
\int_{\cM_{d, \delta}(\balpha)} \hat \chi_\rad^s\left(\widetilde X, \widetilde \omega\right) \, d\nu_d.\]

It is standard to compute Siegel--Veech constants for unit area translation surfaces. 
In \cite[$\S$10]{EskinKontsevichZorich2}, it was observed that for any counting problem with quadratic asymptotics and any scaling factor $\lambda > 0$,
\[N(( \widetilde X, \lambda \widetilde \omega), L) 
= N\left((\widetilde X, \widetilde \omega), \frac{L}{\lambda}\right) 
= \frac{1}{\lambda^2} N(( \widetilde X,  \widetilde \omega), L).\]
Therefore, in order to normalize to unit area, we multiply the count by $d = \lambda^2$ because if the base surface $(X, \omega)$ has unit area, then the covering surface has area $d$.
This yields
\[c_{\text{area}^{s}}(\cM_{1, \delta}(\balpha))
= 
d\cdot \frac{1}{\pi \rad^2} \frac{1}{\nu_d(\cM_{d, \delta}(\balpha))}
\int_{\cM_{d, \delta}(\balpha)} \hat \chi_\rad^s\left(\widetilde X, \widetilde \omega\right) \, d\nu_d.\]
We highlight that the integral is still taken over the locus of covers of area $d$.

As in \eqref{eqn:SVform mcyl}, the integral above is not computable, but we can compute its asymptotics as $\rad$ tends to zero.
By definition, the Siegel--Veech transform $\hat \chi_\rad^s$ is supported entirely on the $\rad$-cusp of $\cM_{d,\delta}(\balpha)$,
and by \eqref{eqn:restricttothick}, we can restrict to the thick part of the cusp at the cost of introducing an error term of order $o(\rad^2)$.

Suppose that $(\widetilde X, \widetilde \omega)$ is in the thick part of the $\rad$-cusp of $\cM_{d,\delta}(\balpha)$.
The base surface $(X, \omega)$ must then lie in the thick part of the $\rad$-cusp for the base stratum $\cH_{1, lab}$.\footnote{
The precise thickness constants will differ by a constant factor, but since they are absorbed into the $o(\rad^2)$, we will suppress this discussion.}
Thus, to every component of the thick part of the cusp of $\cM_{d,\delta}(\balpha)$, we can associate a configuration of homologous cylinders $\cC$ {\em on the base surface $(X, \omega)$}.
More generally, we can partition the set of holonomy vectors of cylinders on $(\widetilde X, \widetilde \omega) \in \cM_{d, \delta}(\balpha)$:
\[V(\widetilde X, \widetilde \omega) = \bigcup_\cC V_{\tilde \cC}(\widetilde X, \widetilde \omega),\]
where $V_{\tilde \cC}$ denotes the set of holonomies of cylinders on $(\widetilde X, \widetilde \omega)$ projecting down to a configuration $\cC$ on $(X, \omega)$. Note, this is why the index set ranges over configurations on the base surface.

Define $\text{area}^s$ Siegel--Veech measures associated to each of these sets
\[
\cV_{\tilde \cC}^s : \left(\tilde X, \tilde \omega \right) \mapsto 
\sum_{\bv \in V_{\tilde \cC}(\tilde X, \tilde\omega)} \left(\sum_{\{hol(cyl) = \bv\}}
\frac{\text{area}^{s}({cyl})}{\text{area}^{s}(\tilde X, \tilde \omega)} \right) \delta_{\bv},
\]
and let $\widehat{f}^{s, \tilde \cC}$ denote the Siegel--Veech transform of a function $f \in C_c(\bR^2)$ with respect to $\cV_{\tilde \cC}^s$.
Thus, by linearity of the Siegel--Veech transform \eqref{eqn:SVconstantsadd},
\begin{equation}\label{eqn:coverconfigSV}
c_{\text{area}^s}(\cM_{1, \delta}(\balpha))
= 
d \cdot \frac{1}{\pi\rad^2}
\frac{1}{\nu_d(\cM_{d, \delta}(\balpha))} 
\sum_{\cC}
\int_{\cM_{d, \delta}(\balpha)} 
\widehat{(\chi_{\rad})}^{s, \tilde \cC}\left(\widetilde X, \widetilde \omega\right) \,d\nu_d.
\end{equation}

\subsection{Explicitly expressing the counting function}
Next, we convert each integral in \eqref{eqn:coverconfigSV} over (the thick part of the cusp of) $\cM_{d, \delta}(\balpha)$ to an integral over $\cH_{1, lab}$. 
This introduces a new subtlety: even if $(X, \omega)$ is in the $\rad$-cusp of $\cH_{1, lab}$, the covering surface $(\widetilde X, \widetilde \omega)$ may not lie in the $\rad$-cusp of $\cM_{d, \delta}(\balpha)$, and even if it does, the preimages of homologous $\rad$-thin cylinders on $(X, \omega)$ can have different widths on $(\widetilde X, \widetilde \omega)$.

Recall that the fibers of $\cM_{d, \delta}(\balpha) \to \cH_{1,lab}$ are in bijection with the monodromy orbit $\overline{\Gamma}_d(\balpha)$, so that a particular cover 
$(\widetilde X, \widetilde \omega) \to (X, \omega)$
corresponds to a branching vector $\bbeta$. 
Suppose a cylinder $C$ on $(X, \omega)$ has holonomy vector $\bv$. Set $w_C$ to be the order of $\bbeta(C) \in \bZ/d\bZ$.
We recall from Lemma~\ref{lemma:cyllifts} that the preimage of $C$ in $(\widetilde X, \widetilde \omega)$ consists of 
\[\# (C) := d/ w_C = \gcd(\bbeta(C), d)\]
many cylinders, each of period $w_C \bv$.
These quantities are locally constant as we deform $(\widetilde X, \widetilde \omega)$, but implicitly they required a marking of the cylinder $C$ and an identification of the fiber over $(X, \omega)$ with (co)homology, and are thus not well-defined once we deform outside a small neighborhood.

All the same, given a configuration $\cC$ of homologous cylinders on $(X, \omega)$ and a branching vector $\bbeta \in H_1(X \setminus \mathcal B; \bZ/d\bZ)$, we define a function $ \Psi_{s, \rad, \cC, \bbeta}: \bR^2 \rightarrow \bR$ via
\[\Psi_{s, \rad, \cC, \bbeta}(\bv) = 
\frac{1}{m(\cC)}\sum_{C \in \cC}
(\# C)^{1-s}\chi_{\rad}(w_C \bv),
\]
where $m(\cC)$ is the number of cylinders in $\cC$ and $\chi_{\rad}$ is the indicator function of a ball of radius $\rad$.
Again, since this function implicitly depends on a local choice of marking of the cylinders of $\cC$, it is not well-defined after one moves far away in $\cH_{1, lab}$.
However, the {\em sum} over $\bbeta$ of all of these functions is monodromy invariant, and thus well-defined.

The purpose of defining these functions is that they capture exactly how many $\rad$-thin cylinders on the different covers $(\widetilde X, \widetilde \omega) \to (X, \omega)$ the $\rad$-thin cylinders on $(X, \omega)$ lift to.
This will allow us to transform the integrals of \eqref{eqn:coverconfigSV} into integrals involving $\Psi_{s, \rad, \cC, \bbeta}$ on the base space $\cH_{1, lab}$.

\begin{proposition}
\label{AvgCylHt}
For any component of any stratum of holomorphic abelian differentials, and for any $s\ge 0$, $\rad>0$, and configuration $\cC$,
\[\int_{\cM_{d, \delta}(\balpha)} \widehat{(\chi_{\rad})}^{s, \tilde \cC}
\left(\widetilde X, \widetilde \omega\right)\, d\nu_d 
=
\int_{\cH_{1, lab}} 
\sum_{\bbeta \in \overline{\Gamma}_d(\balpha)} \widehat \Psi_{s, \rad, \cC, \bbeta}(X, \omega)\, d\nu_1 + o(\rad^2),\]
where the Siegel--Veech transform on the left-hand side is taken with respect to $\cV_{\tilde \cC}^s$, and the Siegel--Veech transforms on the right-hand side are taken with respect to $\cV_{\cC}^s$.
\end{proposition}

The key point of the proof is to expand out the left-hand integral using period coordinates, manipulate some of the interior terms to separate variables, then collapse the integral back down.
The next lemma is crucial for the middle step
and is reminiscent of one of the key steps in \cite[Proof of Theorem 3.1]{ChenMollerZagierQuasiModSiegelVeech}.

\begin{lemma}\label{lem:simplexsymm}
For any $L > 0$, define
\[\Delta_{\le L}:=
\left\{(h_1,\ldots, h_m) \in \bR^m_+
\mid 
\sum h_i \le L \right\}.\]
Then for any $s \geq 0$, the integral
$\int_{\Delta_{\le L}}
h_i^s \, dh_1 \ldots dh_m$ is independent of $i$.
\end{lemma}
\begin{proof}
This follows from the fact that the Jacobian of a permutation matrix is 1.
\end{proof}

\begin{proof}[Proof of Proposition~\ref{AvgCylHt}]
By definition, the function $\widehat{(\chi_{\rad})}^{s, \tilde \cC}(\widetilde X, \widetilde \omega)$ is only nonzero for those $(\widetilde X, \widetilde \omega)$ such that the base surface $(X, \omega)$ has a collection of $\rad$-thin homologous cylinders in configuration $\cC$.
Thus, after restricting to thick parts (which introduces an error term of $o(\rad^2)$ by \eqref{eqn:restricttothick}), 
it is supported only on those components of $\cM_{d, \delta}(\balpha)^{\rad, thick}$ covering the corresponding thick part of the cusp $\cH_{1, lab}^{\rad, thick}(\cC)$.

For $(\tilde X, \tilde \omega)$ in such a thick part we can write out the Siegel--Veech transform explicitly.
Fix a marking of the base surface, hence an identification of $(\tilde X, \tilde \omega)$ with a cohomology class $\bbeta$.
Then if we let $h_C$ denote the height of cylinder $C \in \cC$ on $(X, \omega)$ and $\bv$ denote its holonomy, we write
\begin{align*}
\widehat{(\chi_{\rad})}^{s, \tilde \cC}\left(\widetilde X, \widetilde \omega\right)
& = \sum_{\tilde C \to C \in \cC}
\frac{\text{area}^s(\tilde C)}{\text{area}^s(\tilde X, \tilde \omega)} 
\chi_{\rad}(hol(\tilde C))\\
& =
\sum_{C \in \cC} 
\#(C) \frac{h_C^s w_C^s |\bv|^s }{d^s \text{area}^s(X, \omega)} 
\chi_{\rad}(w_C \bv) \\
&=
\sum_{C \in \cC} 
\#(C)^{1-s} \frac{h_C^s |\bv|^s}{\text{area}^s(X, \omega)} 
\chi_{\rad}(w_C \bv),
\end{align*}
where the first sum is the sum over all cylinders $\tilde C$ of $(\widetilde X, \widetilde \omega)$ covering some cylinder $C$ of $\cC$.

Since the covering map $\cP: \cM_{d, \delta}(\balpha) \to \cH_{1, lab}$ pushes Masur--Veech measure to Masur--Veech measure, we can push our integral down by (restricting to the thick part of the cusp and) summing over fibers:
\begin{align}\label{eqn:pushdownSV}
\nonumber
\int_{\cM_{d, \delta}(\balpha)} \widehat{(\chi_{\rad})}^{s, \tilde \cC}
\left(\widetilde X, \widetilde \omega\right)\, d\nu_d 
& =
\int_{\cH_{1, lab}^{\rad, thick}(\cC)}
\sum_{\cP^{-1}(X, \omega)}
\widehat{(\chi_{\rad})}^{s, \tilde \cC}\left(\widetilde X, \widetilde \omega\right)
\, d\nu_1 + o(\rad^2)\\
& = 
\int_{\cH_{1, lab}^{\rad, thick}(\cC)}
\sum_{\bbeta \in \overline{\Gamma}_d(\balpha)}
\sum_{C \in \cC} 
\#(C)^{1-s} \frac{h_C^s |\bv|^s}{\text{area}^s(X, \omega)} 
\chi_{\rad}(w_C \bv)
\, d\nu_1 + o(\rad^2).
\end{align}
We note once more that individual expressions involving $\sharp (C)$ and $w_C$ depend on a choice of marking, but since we are summing over all $\bbeta$ (equivalently, over the entire fiber of $\cP$) the total expression inside the integral is marking-independent.

The integral in \eqref{eqn:pushdownSV} now looks almost exactly like the Siegel--Veech transform of $\sum_{\bbeta} \Psi_{s, \rad, \cC, \bbeta}$, as desired, except the area$^s$ of cylinder $C$ is intertwined with the combinatorial factors.
To deal with this, we compute in period coordinates and use the independence established in Lemma \ref{lem:simplexsymm}.
\medskip

We first recall that by definition, the measure $\nu_1$ is computed by taking the Masur--Veech measure of its cone: $\nu_1(U) = \dim_{\bR} \cH \cdot \nu\{ tU \mid t \in [0,1]\}$.
Below, we will use $\smash{\cH_{\le 1, lab}^{\rad, thick}(\cC)}$ to denote the cone on $\smash{\cH_{1, lab}^{\rad, thick}(\cC)}$, so by the volume of a cone with unit height, we have
\[\frac{1}{\dim_{\bR} \cH}\int_{\cH_{1, lab}^{\rad, thick}(\cC)} f(X, \omega) d\nu_1
=
\int_{\cH_{\le1, lab}^{\rad, thick}(\cC)} f(X, \omega) d\nu\]
for every area-independent function $f$ (in particular, for the integrand in \eqref{eqn:pushdownSV}).

By definition, every surface $(X, \omega) \in \smash{\cH_{1, lab}^{\rad, thick}(\cC)}$ has a unique collection of $\rad$-thin homologous cylinders in configuration $\cC$.
We may therefore choose a finite collection of period coordinate charts covering the cone
$\smash{\cH_{\le 1, lab}^{\rad, thick}(\cC)}$, each of which corresponds to taking the holonomies of a basis for $H_1(X, \cB; \bZ)$ containing the class of the $C \in \cC$ and transversal saddle connections $T_C$ living in each $C$ and connecting singularities in opposite boundaries.
Said another way, we can fix a local family of charts $\varphi: \mathbb C^N \to \cH_{lab}$ such that $hol(C)$ and $\{hol(T_C)\}_{C\in\cC}$ are given by the first $|\cC| + 1$ coordinates in each chart.

With respect to each such chart, let $\bv = hol(C)$ and decompose each $hol(T_C)$ into a vector $t_C$ parallel to $\bv$ (representing twisting about $C$) and a vector $h_C$ orthogonal to $\bv$ (representing the height of $C$).
Since Masur--Veech measure is just Lebesgue measure in period coordinates, we can therefore disintegrate
\[ d \nu = \varphi_* \prod_{C \in \cC} (d h_C\, d t_C) \,d \bv \, d(\text{rest}),\]
where ``rest'' denotes the (complex) holonomies of the other saddle connections forming a homology basis. Note that when the genus of the base surface is 1 then ``rest'' is empty.

It will also be important to understand the structure of each chart. Fixing all of the holonomies except for those of $T_C$ specifies a translation structure with boundary on the complement of the cylinders of $\cC$.
We can glue in cylinders of arbitrary height and twisting to yield a translation surface in the correct stratum, and so we can extend our chart $\varphi$ to the product of half-planes $\{h_C > 0, t_C \in \R\}$ with some set $U$ in the orthogonal complement $\text{span}_{\C}(\bv, \text{rest})$ (note that increasing $t_C$ too much does a full twist about $C$, hence the extended map is no longer injective).

The most important part of this discussion for us is how the area behaves.
Fixing $(\bv, $rest$) \in U$ defines a structure of some area $A$ on the complement of $\cC$.
The total area is then given by $A + |\bv| \sum {h_C}$, so we get a surface of area $\le 1$ if and only if $\sum h_C \le (1-A)/|\bv|$.
\medskip

By passing to the cone and disintegrating in period coordinates as above, the integral in \eqref{eqn:pushdownSV} becomes
\begin{equation}\label{eqn:pulloutI}
\int_{\cH_{1, lab}^{\rad, thick}(\cC)}
\ldots d \nu
= {\dim_{\bR} \cH} 
\sum_{\text{charts}}
\int_{U}
\int_{0}^{|\bv|}\ldots \int_{0}^{|\bv|}
\underbrace{
\int_{\Delta_{\le (1-A)/|\bv|}}
\ldots 
\prod_{C \in \cC} d h_C 
}_{\text{I}}
\prod_{C \in \cC} d t_C
\,d \bv \, d(\text{rest}),
\end{equation}
(where we have suppressed the integrand for space).
Plugging the integrand back in, we notice that most of the terms can be pulled outside the innermost integral.
\begin{align*}
\text{I} & = 
\int_{\Delta_{\le (1-A)/|\bv|}}
\sum_{\bbeta \in \overline{\Gamma}_d(\balpha)}
\sum_{C \in \cC}
\#(C)^{1-s} \frac{h_C^s |\bv|^s}{\text{area}^s(X, \omega)} \chi_{\rad}(w_C \bv)
\,\prod_{C \in \cC} d h_C \\
& =
\sum_{\bbeta \in \overline{\Gamma}_d(\balpha)}
\sum_{C \in \cC}
\#(C)^{1-s} \frac{|\bv|^s}{\text{area}^s(X, \omega)}\chi_{\rad}(w_C \bv)
\underbrace{
\int_{\Delta_{\le (1-A)/|\bv|}} h_C^s \, \prod_{C \in \cC} d h_C}_{\text{II}_C}
\intertext{
Lemma \ref{lem:simplexsymm} ensures that the integrals $\text{II}_C$ are independent of our choice of $C \in \cC$; let their common value be denoted by $\text{II}$.  Then we can factor $\text{II}$ out of the sum to yield}
& =
\left( \sum_{\bbeta \in \overline{\Gamma}_d(\balpha)}
\sum_{C \in \cC}
\#(C)^{1-s} \chi_{\rad}(w_C \bv) \right) 
\frac{|\bv|^s}{\text{area}^s(X, \omega)} \cdot \text{II}\\
& = 
\left( \sum_{\bbeta \in \overline{\Gamma}_d(\balpha)}
\sum_{C \in \cC}
\#(C)^{1-s} \chi_{\rad}(w_C \bv) \right) 
\cdot \frac{1}{m(\cC)}
\left(\sum_{C \in \cC} \text{II}_C \frac{|\bv|^s}{\text{area}^s(X, \omega)} \right).
\intertext{Expanding each $\text{II}_C$ back out, we recognize the area$^s$ Siegel--Veech transform of $\sum_{\bbeta} \Psi_{s, \rad, \cC, \bbeta}$:}
& = 
\int_{\Delta_{\le (1-A)/|\bv|}}
\left(
\sum_{\bbeta \in \overline{\Gamma}_d(\balpha)}
\frac{1}{m(\cC)}
\sum_{C \in \cC}
\#(C)^{1-s} \chi_{\rad}(w_C \bv) \right) 
\cdot 
\left(\sum_{C \in \cC} \frac{h_C^s|\bv|^s}{\text{area}^s(X, \omega)} \right) 
\, \prod_{C \in \cC} dh_C\\
& = 
\int_{\Delta_{\le (1-A)/|\bv|}}
\sum_{\bbeta \in \overline{\Gamma}_d(\balpha)} \widehat \Psi_{s, \rad, \cC, \bbeta}(X, \omega)
\, \prod_{C \in \cC} dh_C.
\end{align*}
Using \eqref{eqn:pushdownSV}, plugging $\text{I}$ back into \eqref{eqn:pulloutI}, and collapsing the integral back down, we are left with
\[\int_{\cM_{d, \delta}(\balpha)} \widehat{(\chi_{\rad})}^{s, \tilde \cC}
\left(\widetilde X, \widetilde \omega\right)\, d\nu_d 
= 
\int_{\cH_{1, lab}^{\rad, thick}(\cC)}
\sum_{\bbeta \in \overline{\Gamma}_d(\balpha)} \widehat \Psi_{s, \rad, \cC, \bbeta}(X, \omega)\, d\nu_1
+o(\rad^2),\]
which is equal to what we wanted, up to another error of $o(\rad^2)$ coming from the thin part of $\cH_{1,lab}$ \eqref{eqn:restricttothick}.
\end{proof}

\subsection{Simplifying the counting function}
\label{SimplifyCountFcn:Section}

Eventually, we will use the Siegel--Veech transform on $\cH_{1,lab}$ to relate the integral from Proposition \ref{AvgCylHt} to the integrals of the $\Psi_{s, \rad, \cC, \bbeta}$ functions. We record their values below:

\begin{lemma}
\label{CharFcnInt}
For any $s \ge 0$, $\rad >0$, configuration $\cC$, and branching vector $\bbeta$,
\[\int_{\bR^2} \Psi_{s, \rad, \cC, \bbeta}(\bv)\,dx\,dy 
=
\pi \rad^2 \left(
\frac{1}{m(\cC) d^2}\sum_{C \in \cC}
\gcd(\bbeta(C), d)^{3-s}\right),
\]
where $m(\cC)$ is the number of cylinders in the configuration $\cC$.
\end{lemma}
\begin{proof}
The function $\Psi_{s, \rad, \cC, \bbeta}$ is a sum of indicator functions on balls of radius $\rad / w_C$, each weighted by $\#(C)^{1-s} / m(\cC)$. The equation follows by recalling that
$w_C = d/\#(C)$
and
$\#(C) = \gcd(\bbeta(C), d)$.
\end{proof}

We now record a Proposition that will allow us to further simplify the summation from Lemma~\ref{CharFcnInt}.
Recall the definitions of the relevant number theory terms from the beginning of the Section.

\begin{proposition}\label{MainSumSimpRedE}
Let $s \geq 0$.  Given an integer $D$, let $M$ be such that $\text{rad}(D)|M$.  Let $f: \bN \rightarrow \bZ$ be a prime power independent function.  Then 
\[\sum_{a = 1}^D f(\gcd(a, D)) \gcd(a + M, D)^{3-s} = \sum_{\fd|D} f(\fd) \Phi\left(\frac{D}{\fd}\right) \fd^{3-s}.\]
\end{proposition}
\begin{proof}
First observe that from the definition of radical, a prime $p |\text{rad}(M)$ if and only if $p | M$.
For each prime $p$, $p|\gcd(a, D)$ if and only if $p|a$ and $p|D$. Also, if $p|D$, then $p|M$. 
Hence, if $p|\gcd(a,D)$, then it also divides $a+M$.
This combined with $p | D$ implies $p|\gcd(a + M, D)$. 
If $p|\gcd(a + M, D)$, then $p|a+M$.  Since $p|M$, $p|a$ and we conclude $p|\gcd(a, D)$.
Therefore, $f(\gcd(a,D))=f(\gcd(a + M, D))$ because $f$ is prime power independent. This allows us to reindex to get
\begin{align*}
\sum_{a = 1}^{D} f(\gcd(a, D)) \gcd(a + M, D)^{3-s} 
&= \sum_{a = 1}^{D} f(\gcd(a+ M, D)) \gcd(a + M, D)^{3-s}\\
&= \sum_{a' = 1}^{D} f(\gcd(a', D)) \gcd(a', D)^{3-s}. 
\end{align*}
Since $\{a' \mid a' =1,...,D\}$ covers a full set of remainders modulo $D$, each value of $\gcd(a', D)^{3-s}$ occurs $\Phi(D/\gcd(a', D))$ times to conclude.
\end{proof}

\subsection{Conclusion of the proof of Theorem~\ref{GeneralFormula}}

We now have the pieces to put together the proof of the main theorem of this section.

\para{Reduction to a combinatorial formula}
For any $\rad > 0$, combining \eqref{eqn:coverconfigSV} and Proposition \ref{AvgCylHt} gives
\begin{align*}
c_{\text{area}^s}(\cM_{1, \delta}(\balpha))
& = 
d \frac{1}{\pi\rad^2}
\frac{1}{\nu_d(\cM_{d, \delta}(\balpha))} 
\sum_{\cC}
\int_{\cH_{1, lab}} 
\sum_{\bbeta\in \overline{\Gamma}_d(\balpha)} \widehat \Psi_{s, \rad, \cC, \bbeta}(X, \omega)\, d\nu_1 + o(\rad^2)\\
\intertext{Using the $\cV_{\cC}^s$ Siegel--Veech formula for $\cH_{1,lab}$ and Lemma \ref{CharFcnInt}, respectively,}
&= 
d \frac{1}{\pi\rad^2}
\frac{\nu_1(\cH_{1,lab})}{\nu_d(\cM_{d, \delta}(\balpha))} 
\sum_{\cC}
c_{\cC}^s(\cH_{1, lab}) 
\int_{\bR^2}
\sum_{\bbeta \in \overline{\Gamma}_d(\balpha)} \Psi_{s, \rad, \cC, \bbeta}(\bv)\,dx\,dy + o(\rad^2)\\
\\
&= 
\frac{1}{d}
\frac{\nu_1(\cH_{1,lab})}{\nu_d(\cM_{d, \delta}(\balpha))} 
\sum_{\cC}
c_{\cC}^s(\cH_{1, lab}) 
\frac{1}{m(\cC)}
\sum_{\bbeta \in \overline{\Gamma}_d(\balpha)}
\sum_{C \in \cC}
\gcd(\bbeta(C), d)^{3-s} + o(\rad^2).
\intertext{
At this point, we have an equation that is completely combinatorial, so sending $\rad$ to zero, we have
}
c_{\text{area}^s}(\cM_{1, \delta}(\balpha))
&=
\frac{1}{d}
\frac{\nu_1(\cH_{1,lab})}{\nu_d(\cM_{d, \delta}(\balpha))} 
\sum_{\cC}
c_{\cC}^s(\cH_{1, lab}) 
\frac{1}{m(\cC)}
\sum_{\bbeta \in \overline{\Gamma}_d(\balpha)}
\sum_{C \in \cC}
\gcd(\bbeta(C), d)^{3-s}.
\intertext{Since $\cP: \cM_{d, \delta}(\balpha) \rightarrow \cH_{1, lab}$ is a covering map of degree $|\overline{\Gamma}_d\balpha|$ (by Proposition~\ref{CovDeg}) that locally pushes forward $\nu_d$ to $\nu_1$, we can simplify the leading coefficient}
&= 
\frac{1}{d}
\frac{1}{|\overline{\Gamma}_d\balpha|} 
\sum_{\cC}
c_{\cC}^s(\cH_{1, lab}) 
\frac{1}{m(\cC)}
\underbrace{
\sum_{\bbeta \in \overline{\Gamma}_d(\balpha)}
\sum_{C \in \cC}
\gcd(\bbeta(C), d)^{3-s}}_{(\ast)}.
\numberthis\label{eqn:SV to gcd}
\end{align*}

\para{Recovering the orbit quotient function}
Let us now focus on $(\ast)$ for a fixed configuration $\cC$.
Enumerate the cylinders of $\cC$ as $C_1, \ldots, C_m$ and partition $\overline \Gamma_d(\balpha)$ depending on the value $a := \bbeta(C_1)$; label the pieces of the partition as $(\overline{\Gamma}_d\balpha)_{C_1; a}$ as in Equation~(\ref{equation:bargammasuba}).
This choice seems to {\em a priori} matter, but at the very end of the proof we will see that a different labeling of the cylinders results in the same formula.
Using this partition,
\[(\ast) = 
\sum_{a=1}^d 
\sum_{\bbeta \in (\overline{\Gamma}_d\balpha)_{C_1; a}}
\sum_{i=1}^{m(\cC)}
\gcd(\bbeta(C_i), d)^{3-s}.\]
Now we recall that specifying $\bbeta(C_1)$ also specifies the monodromy of all the other $C_i$ (Lemma \ref{Constantbvector:lemma}).
Thus, for fixed $a$ and $i$ the terms of the summand are independent of our choice of $\bbeta \in (\overline{\Gamma}_d\balpha)_{C_1; a}$.
Plugging $(\ast)$ back into \eqref{eqn:SV to gcd}, switching the orders of summation, and rearranging, we are left with
\begin{equation}\label{eqn:pulloutastast}
c_{\text{area}^s}(\cM_{1, \delta}(\balpha))
=
\frac{1}{d}
\sum_{\cC}
c_{\cC}^s(\cH_{1, lab}) 
\frac{1}{m(\cC)}
\sum_{i=1}^{m(\cC)}
\underbrace{
\sum_{a=1}^d 
\frac{|(\overline{\Gamma}_d\balpha)_{a}|}{|\overline{\Gamma}_d\balpha|}
\gcd(\bbeta(C_i), d)^{3-s}}
_{(\ast\ast)}.
\end{equation}
We remark that by the results in Appendix~\ref{section:cyltrans}, the size of $|(\overline{\Gamma}_d\balpha)_{a}|$ does not depend on our choice of $C_1$, just on $a$, which is why we have dropped it from the subscript.

\para{Simplifying $(\ast\ast)$}
At this point, we appeal to Proposition \ref{MainSumSimpRedE} and the assumptions of the Theorem to compute $(\ast\ast)$.
Factor $d = d_{rel}' d_{abs}' 2^k$, for $k \geq 0$, such that $d_{rel}'|d_{rel}$, $d_{abs}'|d_{abs}$, and both $d_{rel}'$ and $d_{abs}'$ are odd.  By Corollary~\ref{theorem:CRT} and the multiplicativity of the gcd function,
\[(\ast\ast)
= \prod_{D \in \{d_{rel}', d_{abs}', 2^k\}}
\left(\sum_{a = 1}^{D}  \frac{|(\overline{\Gamma}_{D}\balpha)_{a}|}{|\overline{\Gamma}_{D}\balpha|} \gcd(\bbeta(C_i), D)^{3-s} \right).\]
Each of the terms of the product contains an orbit quotient function $f_D(a)$, for $D \in \{d_{rel}', d_{abs}', 2^k\}$ as defined at the beginning of this section. 
Replacing these orbit quotient functions and using the third assumption of Theorem~\ref{GeneralFormula} that $f_{d_{abs}'}$ is a constant $C_{d_{abs}'}$,
the product can be rewritten as
\[\left( \sum_{a = 1}^{d_{rel}'}  f_{d_{rel}'}(a) \gcd(\bbeta(C_i), d_{rel}')^{3-s} \right) 
\left( C_{d_{abs}'} \sum_{a = 1}^{d_{abs}'}  \gcd(\bbeta(C_i), d_{abs}')^{3-s} \right)
\left( \sum_{a = 1}^{2^k} f_{2^k}(a) \gcd(\bbeta(C_i), 2^k)^{3-s} \right).\]

We focus on the first term (involving $d_{rel}'$).  Recall from Proposition~\ref{CylinderMonodromy:Prop} that $\bbeta(C_i) = a + b_i \gcd(\delta(\balpha))$ for some $b_i$.
Using the first assumption in the statement of the Theorem that $f_{d_{rel}'}$ is $d_{rel}'$-GCD dependent, we can write
\[\sum_{a = 1}^{d_{rel}'}  f_{d_{rel}'}(a) \gcd(\bbeta(C_i), d_{rel}')^{3-s} = \sum_{a = 1}^{d_{rel}'}  f_{d_{rel}'}(\gcd(a, d_{rel}')) \gcd(a + b_i \gcd(\delta(\balpha)), d_{rel}')^{3-s}.\]
By the second assumption of the Theorem, $f_{d_{rel}'}$ can be factored as a constant multiplied by a prime power independent function (compare Lemma \ref{dMprimeSimpLem} below).  Next we aim to apply Proposition~\ref{MainSumSimpRedE}. 
To do so, we observe that $\gcd(\delta(\balpha))|d_{rel}$ by definition of $d_{rel}$, and moreover, $\text{rad}(\gcd(\delta(\balpha))) = \text{rad}(d_{rel})$.
Since $d_{rel}'|d_{rel}$, we know $\text{rad}(d_{rel}') | b_i\gcd(\delta(\balpha))$ and thus the divisibility of the radical assumption in Proposition~\ref{MainSumSimpRedE} is satisfied for $M = b_i\gcd(\delta(\balpha))$.
Finally, we factor out the constant, apply Proposition~\ref{MainSumSimpRedE}, and multiply back by the constant that was factored out to obtain
\[ \sum_{a = 1}^{d_{rel}'}  f_{d_{rel}'}(\gcd(a, d_{rel}')) \gcd(a + b_i \gcd(\delta(\balpha)), d_{rel}')^{3-s} = \sum_{\fd|d_{rel}'} f_{d_{rel}'}(\fd) \Phi\left(\frac{d_{rel}'}{\fd}\right) \fd^{3-s}.\]

The second term can be simplified by reindexing to get
\[\sum_{a = 1}^{d_{abs}'}  \gcd(\bbeta(C_i), d_{abs}')^{3-s} = \sum_{a = 1}^{d_{abs}'} \gcd(a + b_i \gcd(\delta(\balpha)), d_{abs}')^{3-s} = \sum_{\fd|d_{abs}'} \Phi\left(\frac{d_{abs}'}{\fd}\right) \fd^{3-s}.\]

For the final term (involving $2^k$), there are several cases to consider.  If $2 | \gcd(\delta(\balpha))$, then we apply the same argument above for $d_{rel}'$ using Proposition~\ref{MainSumSimpRedE} to conclude
\begin{equation}
\label{PowerOf2Simp:Eqn}
\sum_{a = 1}^{2^k} f_{2^k}(a) \gcd(\bbeta(C_i), 2^k)^{3-s} = \sum_{\fd|2^k} f_{2^k}(\fd) \Phi\left(\frac{2^k}{\fd}\right) \fd^{3-s}.    
\end{equation}

Therefore, it suffices to consider the case where $2 \nmid d_{rel}$, in which case $2| d_{abs}$.  We split this remaining case into two additional cases depending on whether the orbit is split or unsplit.  If it is unsplit, then by the last assumption $f_{2^k}(a)$ is a constant function and we can use the reindexing argument as in the $d_{abs}'$ case.
This leaves the case where the orbit is split and $2|d_{abs}$, which is addressed by the following lemma.

\begin{lemma}\label{Even_b_i_Split:lemma}
Suppose that $d$ is even and the $\overline{\Gamma}_d$ orbit of the branching vector $\balpha$ is split. Then for any configuration of disjoint cylinders $\cC = \{C_1, \ldots, C_m\}$ on $(X, \omega) \in \cH$ all representing the same class in absolute homology, the parity of the $\balpha(C_i)$ all agree.
\end{lemma}
\begin{proof}
By the proof of Proposition \ref{CylinderMonodromy:Prop}, for all $i$ and $j$,
\[\balpha(C_i) - \balpha(C_j) = \sum_{p_k \in \cB} m_k \balpha(\gamma_k)\]
where $\gamma_k$ is some small loop about $p_k$ and $m_k \in \bZ$. We need only show that the right-hand sum is even: in order to do this, we will specify the $m_k$ more concretely.

Consider the operation where each cylinder is stretched to an infinite cylinder resulting in $m$ (infinite area) translation surfaces each carrying a collection of zeros and exactly two simple poles corresponding to homologous half-infinite cylinders. Consider one such component $W$; re-indexing if necessary, let us assume that its cylinders correspond to $C_i$ and $C_j$. 
This surface $W$ demonstrates the homology in $X$ between $C_i$ and $C_j$, and in particular we get
\begin{equation}\label{eqn:cyldiff}
\balpha(C_i) - \balpha(C_j) = \sum_{p_k \in W} \balpha(\gamma_k).
\end{equation}
Now by Riemann-Roch, the total order of the zeros minus the total order of the poles is even, so since the total order of the poles is $2$, then the total order of the zeros is even.
In particular, there is an even number of odd order zeros on $W$.

Now if $(X, \omega)$ is non-hyperelliptic and the orbit is split, then $\delta(\balpha) \equiv \kappa \mod 2$ by Proposition~\ref{prop:2splitorbits}.  Therefore, the branching over each odd zero is necessarily odd.
Combining this with the fact that there are an even number of odd-order zeros on $W$ implies that the sum in \eqref{eqn:cyldiff} must be even, and so the monodromies of $C_i$ and $C_j$ have the same parity.

Consider now the case where $(X, \omega)$ lies in a hyperelliptic stratum component and there are no regular marked points.
In this setting, we observe that $\cC$ must consist of a single cylinder. Indeed, if there were two disjoint homologous cylinders $C_1$ and $C_2$, then since the hyperelliptic involution fixes each cylinder setwise \cite[Lemma 2.1]{lindsay} and has genus zero quotient, this would imply that both of the components of $S \setminus (C_1 \cup C_2)$ have genus zero, and so $S$ would have to have genus one.
Thus, for hyperelliptic strata with regular marked points, the $m_k$ corresponding to loops about the zeros are all $0$.
The assumption that the orbit is split further implies that the branching over all of the regular marked points is even (by Proposition \ref{prop:2splitorbitshyp}), hence these terms do not change the parity of \eqref{eqn:cyldiff}.
\end{proof}

Therefore, if $2\nmid d_{rel}$, then $2 \nmid \gcd(\delta(\balpha))$ by definition of $d_{rel}$ and by Lemma~\ref{Even_b_i_Split:lemma}, $b_i$ is even.  In this case, we apply Proposition~\ref{MainSumSimpRedE} with $M = b_i \gcd(\delta(\balpha))$, so that $\text{rad}(2^k) = 2 | M$, to again obtain the simplification in Equation~\eqref{PowerOf2Simp:Eqn}.

Putting this all together, we get that for any configuration $\cC$ and all $i$,
\begin{equation}\label{eqn:astast}
(\ast\ast)
=
\left( \sum_{\fd|d_{rel}'} f_{d_{rel}'}(\fd) \Phi\left(\frac{d_{rel}'}{\fd}\right) \fd^{3-s}\right)
\left(C_{d_{abs}'}\sum_{\fd|d_{abs}'} \Phi\left(\frac{d_{abs}'}{\fd}\right) \fd^{3-s} \right)
\left( \sum_{\fd|2^k} f_{2^k}(\fd) \Phi\left(\frac{2^k}{\fd}\right) \fd^{3-s} \right).
\end{equation}

\para{Plugging back in}
We now observe that \eqref{eqn:astast} is independent of the cylinder $C_i$: it depends only on the orbit quotient functions $f_D$ and certain number theoretic data. This (together with the results of Appendix \ref{section:cyltrans}) implies that our initial indexing of the cylinders of $\cC$ did not matter, as when we sum over the cylinders of $\cC$, we get the same value.

Moreover, \eqref{eqn:astast} is {\em also} independent of the initial configuration $\cC$ and the number of cylinders in it. Thus, when we plug $(\ast \ast)$ back into \eqref{eqn:pulloutastast}, we can pull it out of all of the summations, yielding
\[
c_{\text{area}^s}(\cM_{1, \delta}(\balpha))
=
\frac{1}{d} \sum_{\cC} c_{\cC}^s(\cH_{1, lab}) 
\frac{1}{m(\cC)} \sum_{i=1}^m (\ast \ast) 
=
\frac{1}{d} (\ast \ast) 
\sum_{\cC} c_{\cC}^s(\cH_{1, lab})
= 
\frac{1}{d} (\ast \ast) c_{\text{area}^s}(\cH_1),
\]
where the last equality follows from the linearity of Siegel--Veech transforms \eqref{eqn:SVconstantsadd} and the fact that labeling points does not change Siegel--Veech constants (this is clear from their interpretations as asymptotic counts of cylinders).
This completes the proof of the Theorem.
 
\section{Evaluating the orbit quotient function}
\label{OrbitQuotientFcn:Sect}

Now we use the results of Part~\ref{OrbitCard:Part} to show that the hypotheses of Theorem~\ref{GeneralFormula} hold and evaluate the relevant orbit quotient functions.
By definition, in the factorization $d = d_{rel}' d_{abs}'2^k$ we have chosen $d_{rel}'$ and $d_{abs}'$ to be odd. 
Therefore, there is a single formula for each of these quantities for \emph{all} stratum components.

\begin{lemma}
\label{dMprimeSimpLem}
$f_{d_{rel}'}$ is $d_{rel}'$-GCD dependent and can be factored as a constant times a prime power independent function for all components of all strata.  Moreover, the following equation holds:
$$\sum_{\fd|d_{rel}'} f_{d_{rel}'}(\fd) \Phi\left(\frac{d_{rel}'}{\fd}\right) \fd^{3-s} = \frac{d_{rel}'}{d_{rel}'^{s - 2}\Phi_{2g}(d_{rel}')}\sum_{\fd|d_{rel}'} \Phi_{2g-1}\left(\frac{d_{rel}'}{\fd}\right) \frac{\Phi(\fd)}{\fd^{4 - 2g - s}}.$$
\end{lemma}

\begin{proof}
Let $\fd | d_{rel}'$, and let $p|d_{rel}'$.  By Proposition~\ref{prop:oddpstats},  if $p| \fd $, then $|(\bar \Gamma_{p^k} \balpha )_{\fd}| = \Phi_{2g-1}(p^k)$, and if $p \nmid \fd$, then $|(\bar \Gamma_{p^k} \balpha )_{\fd}| = p^{k(2g-1)}$.  These quantities are multiplicative, in particular, 
\[ |(\overline{\Gamma}_{d_{rel}'}\balpha)_{\fd}|= \prod_{p\nmid \fd}p^{k(2g-1)} \prod_{p | \fd} \Phi_{2g-1}(p^k)
=\left(\frac{d_{rel}'}{\fd}\right)^{2g-1}\Phi_{2g-1}(\fd).
\]
By Proposition~\ref{prop:oddpsize}, if $p|d_{rel}'$, then $\abs{\bar \Gamma_{p^k} \balpha} = \Phi_{2g}(p^k)$. Hence,
\begin{align*}f_{d_{rel}'}(\fd) 
= \frac{|(\overline{\Gamma}_{d_{rel}'}\balpha)_{\fd}|}{|\overline{\Gamma}_{d_{rel}'}\balpha|} 
& = \left( \frac{d_{rel}'}{\fd}\right)^{2g-1} \frac{\Phi_{2g-1}(\fd)}{\Phi_{2g}(d_{rel}')}
=  \frac{(d_{rel}')^{2g-1}}{\Phi_{2g}(d_{rel}')} \cdot
\frac{\Phi_{2g-1}(\fd)}{\fd^{2g-1} }.
\end{align*}
It is immediate that $f_{d_{rel}'}$ is $d_{rel}'$-GCD dependent.  
Moreover, we recognize the right-hand side of the equation is a constant depending only on $d'_{rel}$ times $\Phi_{2g-1}(\fd)/\fd^{2g-1}$, which is prime power independent by the definition of the Jordan totient function.
This facilitates the following simplification:
\begin{align*}
\sum_{\fd|d_{rel}'} f_{d_{rel}'}(\fd) \Phi\left(\frac{d_{rel}'}{\fd}\right) \fd^{3-s}
&=
\sum_{\fd|d_{rel}'} \left( \frac{d_{rel}'}{\fd}\right)^{2g-1} \frac{\Phi_{2g-1}(\fd)}{\Phi_{2g}(d_{rel}')} \Phi\left(\frac{d_{rel}'}{\fd}\right) \fd^{3-s}\\
&= 
\frac{(d_{rel}')^{s - 2}}{(d_{rel}')^{s - 2}d'_{rel} \Phi_{2g}(d'_{rel})} \cdot \left(\sum_{\fd|d'_{rel}} \frac{(d_{rel}')^{2g} \Phi_{2g-1}(\fd)}{\fd^{2g-4+s}} \Phi\left(\frac{d'_{rel}}{\fd}\right) \right)\\
&=
\frac{d'_{rel}}{(d_{rel}')^{s - 2}\Phi_{2g}(d'_{rel})} \cdot \left(\sum_{\fd|d'_{rel}} \left( \frac{d'_{rel}}{\fd}\right) ^{2g-4+s}\Phi_{2g-1}(\fd) \Phi\left(\frac{d'_{rel}}{\fd}\right) \right),
\end{align*}
completing the proof of the lemma.
\end{proof}

\begin{lemma}
\label{dcprimeSimpLem}
The function $f_{d_{abs}'}(\fd)$ is constant for all strata and
\[C_{d_{abs}'} \sum_{\fd|d_{abs}'} \Phi\left(\frac{d_{abs}'}{\fd}\right) \fd^{3-s} 
= \frac{1}{(d_{abs}')^{s - 2}} \sum_{\fd|d'_{abs}} \frac{\Phi(\fd)}{\fd^{3-s}}.\]
\end{lemma}

\begin{proof}
If $p|d_{abs}'$, then by Propositions~\ref{prop:oddpsize} and \ref{prop:oddpstats}, $\abs{\bar \Gamma_{p^k} \balpha} = p^{2gk}$ and for all $\fd$, $|(\bar \Gamma_{p^k} \balpha )_{\fd}| = p^{k(2g-1)}$, respectively. 
 Therefore, 
\[f_{d_{abs}'}(\fd) := \frac{|(\overline{\Gamma}_{d_{abs}'}\balpha)_{\fd}|}{|\overline{\Gamma}_{d_{abs}'}\balpha|} = \frac{(d_{abs}')^{2g-1}}{(d_{abs}')^{2g}} = \frac{1}{d_{abs}'} = C_{d_{abs}'}.\]
\end{proof}

The following lemma follows from Propositions~\ref{prop:oddpsize} and \ref{prop:oddpstats} in exactly the same way as Lemma~\ref{dcprimeSimpLem} does.

\begin{lemma}
\label{2tkUnsplitConst:Lemma}
If the orbit of $\balpha$ is unsplit and $2|d_{abs}$, then $f_{2^k}(a)$ is constant.
\end{lemma}

\subsection{Evaluation of orbit ratios for powers of 2}
\label{OrbitQuotientTwotk:Subsection}

Here we focus on the term
\[\sum_{\fd|2^k} f_{2^k}(\fd) \Phi\left(\frac{2^k}{\fd}\right) \fd^{3-s}.\]
Since we will have to break our proof into cases for when the base stratum is hyperelliptic and when it is not, 
let us first prove a preliminary (contingent) result that will allow us to get some number-theoretic manipulations out of the way.
To condense our expression for the final form of our formula, we define $\xi_2(k) = k+1$ and for $s \not= 2$, define
\[\xi_{s}(k) = 2^{k(2-s) +1} + 2^{(2-s)} \frac{1 - 2^{(k-1)(2-s)}}{1 - 2^{2-s}}.\]

\begin{lemma}
\label{SimpFirstModTwo}
If $f_{2^k}(N)$ can be factored as a constant multiplied by a prime power independent function, then
$$\sum_{\fd|2^k} f_{2^k}(\fd) \Phi\left(\frac{2^k}{\fd}\right) \fd^{3-s} = 2^{k-1} \left( f_{2^k}(1)  + f_{2^k}(0) \xi_{s}(k)\right).$$
\end{lemma}

\begin{proof}
The constant that can be factored out will play no role in this proof and can be safely ignored.  The key observation is that prime power independence implies $f_{2^k}(\fd) = f_{2^k}(0)$ if $2|\fd$.  Hence,
\begin{align*}
\sum_{\fd|2^k} f_{2^k}(\fd) \Phi\left(\frac{2^k}{\fd}\right) \fd^{3-s} 
&= \sum_{j=0}^k f_{2^k}(2^j) \Phi\left(\frac{2^k}{2^j}\right) (2^j)^{3-s}\\
&= f_{2^k}(1) \Phi\left(2^k\right) + \left(\sum_{j=1}^{k-1} f_{2^k}(0) \Phi\left(\frac{2^k}{2^j}\right) (2^j)^{3-s}\right) + f_{2^k}(0) \Phi(1)(2^k)^{3-s}\\
&= 2^{k-1}\left(f_{2^k}(1)  + f_{2^k}(0) \left(2^{k(2-s)+1} + \sum_{j=1}^{k-1} (2^{2-s})^j\right)\right).
\end{align*}
If $s \not= 2$, then we can simplify using a geometric sum, and if $s = 2$, we can simplify directly.
\end{proof}

\begin{proposition}
\label{NonHypTwoPowerFormulaProp}
Let $\cH$ be a non-hyperelliptic stratum component and let $\balpha$ be a branching vector. Let $\xi_{s}(k)$ be as defined above.
If the orbit of $\balpha$ is split, then $f_{2^k}(N)$ can be factored as a constant multiplied by a prime power independent function, $f_{2^k}$ is $2^k$-GCD dependent,
and
$$\sum_{\fd|2^k} f_{2^k}(\fd) \Phi\left(\frac{2^k}{\fd}\right) \fd^{3-s} = \frac{2^{g-1}(2^{g-1}+(-1)^{\psi(\balpha)}) + (2^{2g-2}-\tau(\balpha))\xi_{s}(k)}{2^{g-1}(2^g + (-1)^{\psi(\balpha)})-\tau(\balpha)}.$$
Moreover, if $2|d_{abs}$, then $\tau(\balpha) = 0$.
\end{proposition}

\begin{proof}
From Propositions~\ref{prop:2splitsize} and \ref{prop:2splitstats}, it follows that $f_{2^k}(N)$ is a constant times a prime power independent function. Furthermore, $2^k$-GCD dependence follows from Proposition~\ref{prop:oddporbits}.  If $2| d_{rel}$, then from Lemma~\ref{SimpFirstModTwo} and Propositions~\ref{prop:2splitsize} and \ref{prop:2splitstats}
\begin{align*}
\sum_{\fd|2^k} & f_{2^k}(\fd) \Phi\left(\frac{2^k}{\fd}\right) \fd^{3-s} = 2^{k-1} \left( f_{2^k}(1)  + f_{2^k}(0) \xi_{s}(k)\right)\\
& =
2^{k-1} \left( \frac{2^{(2k-1)g -k}(2^{g-1}+(-1)^{\psi(\balpha)})}{ 2^{2g(k-1)}\left(2^{g-1}(2^g +  (-1)^{\psi(\balpha)})-\tau(\balpha)\right)} + \frac{2^{(2g-1)(k-1)}(2^{2g-2}-\tau(\balpha))}{2^{2g(k-1)}\left(2^{g-1}(2^g + (-1)^{\psi(\balpha)})-\tau(\balpha)\right)}\xi_{s}(k)\right).
\end{align*}
For the last statement, we just note that if $2|d_{abs}$, then $\tau(\balpha) = 0$ because $2|d_{abs}$ implies $\gcd(\delta(\balpha))$ is odd and since the orbit is split, $\delta = \kappa \, (\text{mod } 2)$ by Proposition~\ref{prop:2splitorbits}.
\end{proof}

\begin{proposition}
\label{HypTwoPowerFormulaProp}
Let $\cH^{hyp}$ be a hyperelliptic stratum component and let $\balpha$ be a branching vector. Let $\xi_{s}(k)$ be as defined above. 
If the orbit of $\balpha$ is split, then $f_{2^k}(N)$ can be factored as a constant multiplied by a prime power independent function, $f_{2^k}$ is $2^k$-GCD dependent, and
\[\sum_{\fd|2^k} f_{2^k}(\fd) \phi\left(\frac{2^k}{\fd}\right) \fd^{3-s} =  \xi_{s}(k) + 2 (1-\xi_{s}(k)) \frac{ \binom{Br-2}{b(\balpha)-1}}{ \binom{Br}{b(\balpha)}}.\]
\end{proposition}

\begin{proof}
It follows from Propositions~\ref{prop:2splitsizehyp} and \ref{prop:2splitstatshyp} that $f_{2^k}(N)$ is a constant times a prime power independent function. As in the non-hyperelliptic case, $2^k$-GCD dependence follows from Proposition~\ref{prop:oddporbits}.
Substituting the formulas derived in Propositions~\ref{prop:2splitsizehyp} and \ref{prop:2splitstatshyp} into the formula from Lemma~\ref{SimpFirstModTwo} yields
\begin{align*}
\sum_{\fd|2^k}  f_{2^k}(\fd) &\Phi\left(\frac{2^k}{\fd}\right) \fd^{3-s} = 2^{k-1} \left( f_{2^k}(1)  + f_{2^k}(0) \xi_{s}(k)\right)\\
& = 2^{k-1} \left( \frac{2^{(2g-1)(k-1)-\sigma(\balpha)}\cdot 2 \binom{Br-2}{b(\balpha)-1}}{2^{2g(k-1)-\sigma(\balpha)} \binom{Br}{b(\balpha)}}  +  \frac{2^{(2g-1)(k-1)-\sigma(\balpha)}\cdot \left( \binom{Br-2}{b(\balpha)} + \binom{Br-2}{b(\balpha)-2} \right)}{2^{2g(k-1)-\sigma(\balpha)} \binom{Br}{b(\balpha)}} \xi_{s}(k)\right) \\
&= \frac{ 2 \binom{Br-2}{b(\balpha)-1}}{ \binom{Br}{b(\balpha)}}(1-\xi_{s}(k))   +  \frac{ \left( \binom{Br-2}{b(\balpha)} + 2 \binom{Br-2}{b(\balpha)-1} + \binom{Br-2}{b(\balpha)-2} \right)}{ \binom{Br}{b(\balpha)}} \xi_{s}(k).
\end{align*}
The final formula follows from the Pascal relation on binomial coefficients.
\end{proof}

\section{Main formula}\label{subsection:mainformula}

\para{Notational guide} Let us briefly recall the relevant notation and indicate where in the paper a full discussion can be found.

\begin{itemize}
\item $\cH$ is a connected component of a stratum $\cH(\kappa)$ of translation surfaces, possibly with regular marked points.
We use $\cH_{1}$ to denote its unit-area locus.
Recall our convention that $\cH$ is hyperelliptic if the stratum component obtained by removing all marked points is hyperelliptic.
\item $\cM_{\delta}$ is the locus of degree $d$ cyclic branched covers of surfaces with branching data given by $\delta$.
\item We factor the degree $d$ into ``absolute'' and ``relative'' parts $d_{abs}$ and $d_{rel}$.
If $d$ is even, we further factor $d = d_{abs}' d_{rel}'2^k $ such that $d_{abs}'|d_{abs}$ and $d_{rel}'|d_{rel}$ are odd.
\item The components of $\cM_\delta$ are classified by the monodromy orbits of the {\em branching vector} $\balpha \in H^1(X \setminus \cB; \bZ/d\bZ)$. We use $\cM_\delta(\balpha)$ to denote the component containing the branched covers of surfaces in $\cH$ with branching vector $\balpha$ and $\cM_{1,\delta}(\balpha)$ to denote the unit-area locus.
\item A monodromy orbit of branching vectors is {\em unsplit} if $\cM_\delta(\balpha) = \cM_\delta$ and is split otherwise.
Whether the orbit of $\balpha$ is split or unsplit is determined by Propositions \ref{prop:oddporbits}, \ref{prop:2splitorbits}, and \ref{prop:2splitorbitshyp}.
In particular, the orbit is split only when $d$ is even.
\item The {\em $\psi$ invariant}, valued in $\Z/2\Z$, is a complete invariant classifying split orbits when $\cH$ is non-hyperelliptic. See \Cref{section:psi}.
\item The quantity $\tau(\balpha) \in \Z/2\Z$ is defined in \Cref{prop:2splitsize}.
\item The quantities $\Br$ and $b(\balpha)$ are complete integer invariants classifying split orbits when $\cH$ is hyperelliptic. See \Cref{definition:br,def:bw}.
\item Recall the function from Section~\ref{OrbitQuotientTwotk:Subsection}
$$\xi_{s}(k) = \left\{ \begin{array}{cl} 2^{k(2-s) +1} + 2^{(2-s)} \frac{1 - 2^{(k-1)(2-s)}}{1 - 2^{2-s}} & \text{ if } s \not= 2 \\   k+1 & \text{ if } s = 2. \end{array} \right.$$
\item For positive integers $A$ and $B$, define a function
\[G_{s, g}(A, B) = \frac{A}{\Phi_{2g}(A)}\left(\sum_{\fd|A} \Phi_{2g-1}\left(\frac{A}{\fd}\right) \frac{\Phi(\fd)}{\fd^{4 - 2g - s}}  \right) \cdot \left(\sum_{\fd|B} \left(\frac{\Phi(\fd)}{\fd^{3-s}}\right) \right).\]
\end{itemize} 

\begin{theorem}
\label{MainTheoremSVFormula}
Let $s \geq 0$. Then 
${c_{\text{area}^{s}}(\cM_{1,\delta}(\balpha))}/{c_{\text{area}^{s}}(\cH_1)}$
is given by one of the following formulas:
\begin{itemize}
\item If $\cH_1$ is non-hyperelliptic and the orbit of $\balpha$ is split,
\[\frac{2^{k(s - 2)}}{d^{s-1}} \cdot \frac{2^{g-1}(2^{g-1}+(-1)^{\psi(\balpha)}) + (2^{2g-2}-\tau(\balpha))\xi_{s}(k)}{2^{g-1}(2^g + (-1)^{\psi(\balpha)})-\tau(\balpha)} \cdot G_{s, g}(d_{rel}', d_{abs}').\]
\item If $\cH_1$ is hyperelliptic and the orbit of $\balpha$ is split,
\[\frac{2^{k(s - 2)}}{d^{s-1}} \left( \xi_{s}(k) + 2 (1-\xi_{s}(k)) \frac{ \binom{Br-2}{b(\balpha)-1}}{ \binom{Br}{b(\balpha)}}\right) \cdot G_{s, g}(d_{rel}', d_{abs}').\]
\item Otherwise, the orbit of $\balpha$ is unsplit, and the formula is
\[d^{1-s} \cdot G_{s, g}(d_{rel}, d_{abs}).\]
\end{itemize}
\end{theorem}

\begin{proof}
Having derived every piece of the main formula in the sections above, we combine the results here. 
Lemmas~\ref{dMprimeSimpLem}, \ref{dcprimeSimpLem}, and \ref{2tkUnsplitConst:Lemma} along with Propositions~\ref{NonHypTwoPowerFormulaProp} and \ref{HypTwoPowerFormulaProp} show that the first two assumptions in Theorem~\ref{GeneralFormula} are satisfied.  Lemma~\ref{dcprimeSimpLem} verifies the third assumption of Theorem~\ref{GeneralFormula} holds.  Finally, Lemma~\ref{2tkUnsplitConst:Lemma} establishes that the fourth assumption of Theorem~\ref{GeneralFormula} holds.
Thus, combining Theorem~\ref{GeneralFormula} and Lemmas~\ref{dMprimeSimpLem}, \ref{dcprimeSimpLem} and \ref{2tkUnsplitConst:Lemma}, we have in all cases that the ratio
${c_{\text{area}^{s}}(\cM_{d, \delta}(\balpha))}/{c_{\text{area}^{s}}(\cH_1)}$ is given by
\begin{align*}
&\frac{1}{d} 
\left(\sum_{\fd|d_{rel}'} f_{d_{rel}'}(\fd) \Phi\left(\frac{d_{rel}'}{\fd}\right)\fd^{3-s}\right) 
\left(C_{d_{abs}'} \sum_{\fd|d_{abs}'} \Phi\left(\frac{d_{abs}'}{\fd}\right) \fd^{3-s} \right)\left( \sum_{\fd|2^k} f_{2^k}(\fd) \Phi\left(\frac{2^k}{\fd}\right) \fd^{3-s} \right) \\
&\qquad =
\frac{1}{d} \left(\frac{d_{rel}'}{(d_{rel}')^{s - 2}\Phi_{2g}(d_{rel}')}\sum_{\fd|d_{rel}'} \Phi_{2g-1}\left(\frac{d_{rel}'}{\fd}\right) \frac{\Phi(\fd)}{\fd^{4 - 2g - s}}  \right)
\left(\frac{1}{(d'_{abs})^{s - 2}} \sum_{\fd|d'_{abs}} \left(\frac{\Phi(\fd)}{\fd^{3-s}}\right) \right)\\
& \qquad \qquad \cdot
\left( \frac{2^{k(s-2)}}{2^{k(s-2)}} \sum_{\fd|2^k} f_{2^k}(\fd) \Phi\left(\frac{2^k}{\fd}\right) \fd^{3-s} \right).    
\end{align*}
We have added an extra term to the $2^k$ term for the convenience of the reader. 
From this final expression, the reader can readily apply Propositions~\ref{NonHypTwoPowerFormulaProp} and \ref{HypTwoPowerFormulaProp} to get the desired set of formulas.
\end{proof}

We emphasize that this result also holds when $\cH_1$ is a stratum of marked tori (in which case all orbits are unsplit); specializing to the $g=1$ case yields Corollary \ref{TorusCovAreaSVCor} from the introduction.

\subsection{area- and area cubed--Siegel--Veech constants}
\label{FullAreaSVConst:Section}

Given the particular importance of area-Siegel--Veech constants in the formula for the sum of the Lyapunov exponents of the Kontsevich-Zorich cocycle \cite{EskinKontsevichZorich2}, we explicitly state the formula for them for future ease of reference.  Observe: $\xi_{1}(k) = 3\cdot 2^k - 2$ and
$$G_{1, g}(d_{rel}, d_{abs}) = \frac{d_{rel}}{\Phi_{2g}(d_{rel})}\left(\sum_{\fd|d_{rel}} \Phi_{2g-1}\left(\frac{d_{rel}}{\fd}\right) \frac{\Phi(\fd)}{\fd^{3 - 2g}}  \right) \cdot \left(\sum_{\fd|d_{abs}} \left(\frac{\Phi(\fd)}{\fd^{2}}\right) \right).$$

\begin{theorem}
\label{SVFormulaAlpha_1}
The ratio
${c_{\text{area}}(\cM_{1,\delta}(\balpha))}/{c_{\text{area}}(\cH_1)}$
is given by one of the following formulas:
\begin{itemize}
\item If $\cH$ is non-hyperelliptic and the orbit of $\balpha$ is split,
\[2^{-k} \cdot \frac{2^{g-1}(2^{g-1}+(-1)^{\psi(\balpha)}) + (2^{2g-2}-\tau(\balpha))(3\cdot 2^k - 2)}{2^{g-1}(2^g + (-1)^{\psi(\balpha)})-\tau(\balpha)} \cdot G_{1, g}(d_{rel}', d_{abs}')\]
\item If $\cH$ is hyperelliptic and the orbit of $\balpha$ is split, then
\[2^{-k} \left( 3\cdot 2^k - 2 + 6 (1- 2^k) \frac{ \binom{Br-2}{b(\balpha)-1}}{ \binom{Br}{b(\balpha)}}\right) \cdot G_{1, g}(d_{rel}', d_{abs}')\]
\item Otherwise, the orbit of $\balpha$ is unsplit and the formula is given by $G_{1, g}(d_{rel}, d_{abs}).$
\end{itemize}
\end{theorem}

\

We now address the simplification that occurs for $area^3$ Siegel--Veech constants.

\begin{proof}[Proof of \Cref{Area3SVdOdd}]
This follows from the simplification $\xi_{3}(k) = 1$ and the elementary identities:
$$ \Phi_{2g}(d_{rel}') = \sum_{\fd|d_{rel}'} \Phi_{2g-1}\left(\frac{d_{rel}'}{\fd}\right) \frac{\Phi(\fd)}{\fd^{1 - 2g}} \qquad \text{and} \qquad d'_{abs} = \sum_{\fd|d'_{abs}} \Phi(\fd).$$
\end{proof}

\appendix

\section{Symplectic modules}\label{section:sympmodules}

Here we collect some basic results on the action of the symplectic group $\Sp(2g, \Z/d\Z)$ and certain of its subgroups on $(\Z/d\Z)^{2g}$ (equipped with the symplectic form $\pair{\cdot,\cdot}$), which do not seem to be readily available in the literature. We allow the possibility that $d = 0$ throughout. 

Recall that for $d$ even, there are subgroups $\Sp(2g, \Z/d\Z)[2]$ and $\Sp(2g, \Z/d\Z)[q]$ defined as follows: $\Sp(2g, \Z/d\Z)[2]$ is the kernel of the reduction map $\Sp(2g,\Z/d\Z) \to \Sp(2g, \Z/2\Z)$, and, for a chosen $\Z/2\Z$ quadratic form $q$ on $(\Z/d\Z)^{2g}$ (i.e. the pullback of a quadratic form on $(\Z/2\Z)^{2g}$ under reduction mod $2$), the subgroup $\Sp(2g, \Z/d\Z)[q]$ is the stabilizer of $q$ under the action of $\Sp(2g,\Z/d\Z)$ on the set of quadratic forms on $(\Z/2\Z)^{2g}$. 

The common theme in the results presented in \Cref{prop:symporbits} below is that the orbits of each of these subgroups on $(\Z/2\Z)^{2g}$ is ``predicted'' by the evident invariant.

\begin{proposition}\label{prop:symporbits}
Fix $g \ge 1$ and $d \ge 0$, and let $x, y \in (\Z/d\Z)^{2g}$ be given.
\begin{enumerate}
\item\label{item:sp} There exists $A \in \Sp(2g, \Z/d\Z)$ for which $Ax = y$ if and only if $\gcd(x) = \gcd(y)$.
\item\label{item:sp2} For $d$ even, there exists $B \in \Sp(2g, \Z/d\Z)[2]$ for which $Bx = y$ if and only if $\gcd(x) = \gcd(y)$ and $x = y \pmod 2$.
\item\label{item:spq} For $d$ even, there exists $C \in \Sp(2g, \Z/d\Z)[q]$ for which $Cx = y$ if and only if $\gcd(x) = \gcd(y)$ and $q(x) = q(y)$.
\end{enumerate}
\end{proposition}

\begin{proof}
These results are classical, and we content ourselves with a sketch of the proofs. To begin, we show that it suffices to take $d = 0$ throughout. According to \cite[Theorem 1]{NewmanSmart}, the reduction map $\Sp(2g, \Z) \to \Sp(2g, \Z/d\Z)$ is surjective for all $d \ge 0$.  As $\Sp(2g, \Z/d\Z)[2]$ is defined as the kernel of the reduction map $\Sp(2g, \Z/d\Z) \to \Sp(2g, \Z/2\Z)$, it follows that reduction mod $d$ induces a surjection $\Sp(2g,\Z)[2] \to \Sp(2g, \Z/d\Z)[2]$, and a similar argument shows that $\Sp(2g,\Z)[q]$ surjects onto $\Sp(2g, \Z/d\Z)[q]$. 
Recall that an element of $(\Z/d\Z)^{2g}$ is {\em primitive} if its components generate $\Z/d\Z$. We note that every primitive element $\bar v \in (\Z/d\Z)^{2g}$ lifts to a primitive element $v \in \Z^{2g}$: this can be seen directly by elementary number theory, or else derived from the well-known surjectivity of $\SL_{2g}(\Z) \to \SL_{2g}(\Z/d\Z)$, extending $\bar v$ to a unimodular basis of $(\Z/d\Z)^{2g}$ (possible since $2g > 1$), and taking an integral lift of the associated matrix. Finally, we note that in case $d$ is even, both the mod-$2$ reduction and $q$-values of $v$ and $\bar v$ coincide. Taken together, these facts imply that to classify orbits mod $d$, it suffices to take integral lifts and classify the corresponding orbits in $\Z^{2g}$.

Next, we reduce to the case $\gcd(x) = 1$ by dividing out the common factor. The key idea in all three parts is an ``extension lemma'' in the spirit of Witt's Theorem, which in its most basic form (suitable for \eqref{item:sp}) asserts that an arbitrary primitive element of $\Z^{2g}$ can be extended to a symplectic basis. In \eqref{item:sp2}, a stronger extension lemma asserts that given a symplectic basis $\{x_1,y_1, \dots,x_g, y_g\}$ and a vector $z$ such that $z = x_1 \pmod 2$, then $z$ can be completed to a symplectic basis $\{z = z_1,w_1, \dots, z_g, w_g\}$ such that $z_i = x_i$ and $w_i = y_i \pmod 2$ for all $i$. And in \eqref{item:spq}, the relevant extension lemma asserts that given a $\Z/2\Z$ quadratic form $q$, a geometric symplectic basis $B = \{x_1, y_1, \dots, x_g, y_g\}$, and a primitive element $z$ satisfying $q(z) = q(x_1)$, then $z$ can be completed to a symplectic basis $\{z = z_1,w_1, \dots, z_g, w_g\}$ such that $q(z_i) = q(x_i)$ and $q(w_i) = q(y_i)$ for $1 \le i \le g$. 

Granting such extension lemmas, the proof of the theorem follows quickly. For \eqref{item:sp}, apply the basic extension lemma to $x$ and to $y$, producing symplectic bases $B, B'$, and let $A \in \GL(2g, \Z)$ be determined by the condition $A B = B'$. Then in fact $A \in \Sp(2g, \Z)$, since an element of $\GL(2g, \Z)$ is symplectic if and only if it takes any symplectic basis to a symplectic basis. For \eqref{item:sp2}, apply the basic extension lemma to $x$, producing a symplectic basis $B$. Then apply the strong extension lemma to $y$, producing a symplectic basis $B'$ such that every element of $B$ has the same mod-$2$ reduction as its counterpart in $B'$. The element of $\GL(2g, \Z)$ taking $B$ to $B'$ then is seen to be an element of $\Sp(2g, \Z)[2]$. The proof of \eqref{item:spq} is similar. 
\end{proof}

The following lemma was invoked in the course of proving \Cref{theorem:homologicalmonodromy}.
\begin{lemma}\label{lemma:3tuple}
Let $q$ be a quadratic form on $H_1(X; \Z/2\Z)$. For $g \ge 3$, any tuple $(a,b,c) \in (\bZ/2\bZ)^3$ is realizable via elements $v, w \in H_1(X; \bZ/2\bZ)$ so that
\[
q(v) = a, \quad q(w) = b, \quad \pair{v,w} = c.
\]
\end{lemma}
\begin{proof}
Choose a symplectic basis $\{x_1, y_1, \dots, x_g, y_g\}$ for $H_1(X; \Z/2\Z)$. Suppose first that $c = 1$. Define a quadratic form $q'$ on $H_1(X; \Z/2\Z)$ for which $q'(x_1) = a$ and $q'(y_1) = b$. The values for $q'(x_i)$ and $q'(y_i)$ for $1 < i < g$ can be chosen arbitrarily; then choose $q'(x_g) =1$ and $q'(y_g)$ to be whichever of $0,1$ makes $\Arf(q')=\Arf(q)$. Thus $q$ and $q'$ lie in the same orbit of $\Sp(2g, \Z/2\Z)$ by Proposition~\ref{prop:symporbits}, so that $q'(x) = q(Ax)$ for some $A \in \Sp(2g,\Z/2\Z)$. Taking then $v = Ax_1$ and $w = A y_1$ gives $v,w$ with the required properties.

The argument in the case $c = 0$ is similar. This time, choose $q'(x_1) = a$ and $q'(x_2) = b$. As $g \ge 3$, it is possible to choose $q'(x_g) = 1$ and $q'(y_g)$ suitably so that $\Arf(q') = \Arf(q)$, and the rest of the argument follows as above.
\end{proof}

\section{Action on the homology classes of cylinders}\label{section:cyltrans}
We return to the setting of Sections \ref{Sect:SetupRedHom}-\ref{section:splithyp}. Recall from \Cref{FixedHomClassSection} that the crucial quantities
\[
\abs{(\bar \Gamma_{d} \balpha )_{c;a}} = \abs{\{\bbeta \in \bar \Gamma_{d} \balpha \mid \bbeta(c) = a \pmod{d}\}}
\]
are independent of the choice of cylinder $C$.  We emphasize that in the non-hyperelliptic case, this was proven in the above sections.  However, the proof of this fact in the hyperelliptic setting relied on the result in this appendix, which we establish below. In this section, we provide a conceptual explanation for this phenomenon, by showing that the monodromy group acts transitively on homology classes of cylinders as elements of $H_1(X \setminus \cB; \Z)$. For a non-hyperelliptic stratum component in genus $g \ge 5$, this follows from \cite[Corollary 1.2]{strata3}, which proves the stronger result that the mapping class group monodromy acts transitively on {\em isotopy classes} of cores of cylinders. A similar strategy will be used in the hyperelliptic setting. However, in this paper we obtain results for {\em all} stratum components, including the non-hyperelliptic components in genus $g < 5$ to which the results of \cite{strata3} do not apply.

\begin{proposition}\label{prop:nonhypcyltrans}
Let $\cH$ denote a non-hyperelliptic component of a stratum, and let $(X,\omega) \in \cH$ be a basepoint with distinguished set $\cB$. Then the monodromy group $\bar \Gamma \leqslant \PAut(H_1(X\setminus \cB; \Z))$ acts transitively on the set of homology classes $[c] \in H_1(X\setminus\cB; \Z)$ represented by core curves of cylinders on $(X,\omega)$.
\end{proposition}

The same result holds in the hyperelliptic setting, but the method of proof will differ.

\begin{proposition}\label{prop:hypcyltrans}
Let $\cH$ denote a hyperelliptic component of a stratum, and let $(X, \omega) \in \cH$ be a basepoint with distinguished set $\cB$ (possibly including some ordinary marked points). Then the monodromy group $\bar \Gamma \leqslant \PAut(H_1(X\setminus \cB; \Z))$ acts transitively on the set of homology classes $[c] \in H_1(X\setminus\cB; \Z)$ represented by core curves of cylinders on $(X, \omega)$.
\end{proposition}

\begin{corollary}
\label{CylInv:Lemma}
Let $\cH$ denote a component of a stratum (hyperelliptic or otherwise), and let $(X, \omega) \in \cH$ be a basepoint with distinguished set $\cB$. Let $c$ and $c' \in H_1(X \setminus \cB; \bZ)$ represent the cores of cylinders $C$ and $C'$, respectively.  Then
$$|\{\beta \in \overline{\Gamma}_d \alpha | \, \beta(c) = a\}| = |\{\beta \in \overline{\Gamma}_d \alpha | \, \beta(c') = a\}|.$$
\end{corollary}

\begin{proof}
By \Cref{prop:nonhypcyltrans} or \Cref{prop:hypcyltrans} as appropriate, the monodromy action of the stratum on $H_1(X \setminus \cB; \Z/d\Z)$ is transitive on the set of classes $[c]$ represented by core curves of cylinders. Let $f \in \overline{\Gamma}_d$ take $c$ to $c'$. Then 
\[
 \{\beta \in \overline{\Gamma}_d \alpha | \, \beta(c') = a\} = \{\beta \in \overline{\Gamma}_d \alpha | \, \beta(f(c)) = a\}  =  f \cdot \{\beta \in \overline{\Gamma}_d \alpha | \, \beta(c) = a\}.
\]
In particular, this shows the desired equality of cardinalities.
\end{proof}

\subsection{Non-hyperelliptic}

\para{Duality} To prove \Cref{prop:nonhypcyltrans}, it is necessary to discuss some linear-algebraic preliminaries. The core subtlety in what follows is that the monodromy computation \Cref{theorem:homologicalmonodromy} determines the action of $\bar \Gamma$ on {\em relative} homology $H_1(X, \cB; \Z)$. On the other hand, the cylinder classes lie in {\em excision} homology $H_1(X \setminus \cB; \Z)$. These spaces are dual to each other via the relative intersection pairing 
\[
\pair{\cdot, \cdot} : H_1(X\setminus\cB; \Z) \otimes H_1(X,\cB; \Z) \to \Z
\]
discussed in \Cref{subsection:hommonodromy}. Accordingly, we can consider $\gamma \in \bar \Gamma$ as acting on $H_1(X\setminus \cB; \Z)$ (denoted $\gamma^*$), given by the adjoint formula
\[
\pair{\gamma^* x, y} = \pair{x,\gamma y}.
\]

\para{Outline} To prove \Cref{prop:nonhypcyltrans}, in Definition \ref{definition:qstar} we will construct yet another algebraic distillate of the winding number function $\phi$ induced from the translation surface structure on $X$, this time, a function $q^*: H_1(X\setminus \cB; \Z) \to \Z/2\Z$ that behaves like a $\Z/2\Z$ quadratic form (note that the intersection pairing $\pair{\cdot, \cdot}$ on $H_1(X\setminus \cB;\Z)$ is {\em degenerate}). In \Cref{lemma:qifftheta}, we then show that the adjoint action of $\bar \Gamma$ on $H_1(X\setminus \cB;\Z)$ is characterized by $q^*$-invariance. Using a geometric interpretation of $q^*$ originally obtained in \Cref{lemma:qformula}, we then prove \Cref{prop:nonhypcyltrans}.\\

\para{Duality and geometric splittings}
Let $X^\circ$ be a geometric splitting, inducing an algebraic splitting $H_1(X, \cB; \Z) \cong H_1(X;\Z) \oplus \tilde H_0(\cB; \Z)$ as in \Cref{definition:geomsplit}. This also induces a splitting 
\begin{equation}\label{equation:geomsplitexc}
H_1(X\setminus \cB;\Z) \cong H_1(X;\Z) \oplus \tilde H_0(\cB;\Z)
\end{equation}
via the identifications $H_1(X;\Z) \cong H_1(X^\circ; \Z) \leqslant H_1(X\setminus\cB;\Z)$ and $\tilde H_0(\cB;\Z) \cong \ker(i_*)$, where $i_*: H_1(X\setminus \cB;\Z) \to H_1(X;\Z)$ is induced by inclusion. 

Observe that the subspaces $H_1(X;\Z)$ of $H_1(X, \cB;\Z)$ and $H_1(X \setminus\cB;\Z)$ are dual to each other under the relative intersection pairing $\pair{\cdot, \cdot}$, and that each is isomorphic to $H_1(X^\circ;\Z)$. Hence, these subspaces can be canonically identified. This leads to the following subtlety: there are two natural ways to pass from a basis of $H_1(X, \cB; \Z)$ to a basis for $H_1(X \setminus \cB; \Z)$: either by taking the dual basis with respect to the intersection pairing, or else by identifying the common subspace $H_1(X;\Z)$, and identifying the complementary factors by $\pair{\cdot, \cdot}$. 

In the sequel we will need to understand the relationship between these. Let $\beta = \{x_1, y_1,\dots,x_g, y_g, a_1, \dots, a_{n-1}\}$ be a basis for $H_1(X, \cB; \Z)$ such that the restriction $\{x_1, y_1, \dots, x_g, y_g\}$ to $H_1(X; \Z)$ is a symplectic basis. Let $\beta^* \subset H_1(X \setminus \cB; \Z)$ denote the dual basis to $\beta$ induced by the pairing $\pair{\cdot, \cdot}$ on $H_1(X\setminus\cB; \Z) \otimes H_1(X,\cB; \Z)$, and let $\beta' \subset H_1(X \setminus \cB; \Z)$ denote the basis $\{x_1, y_1, \dots, x_g,y_g, a_1^*, \dots, a_{n-1}^*\}$ given by identifying the common subspace $H_1(X;\Z)$ and adjoining the dual basis elements $a_i^*$ to $a_i$. 

\begin{lemma}
    \label{lemma:matrixformulas}
    Let $\gamma = \MAOI \in \PAut(H_1(X, \cB;\Z))$ be given (in coordinates $\beta$ as discussed above). With $\beta^*$ and $\beta'$ as above, the matrix for $\gamma^*$ in $\beta^*$ is given by the transpose, while the matrix for $\gamma^*$ in $\beta'$ is given by
    \[
    [\gamma^*] = \begin{pmatrix}
        M^{-1} & 0 \\ A^T & I
    \end{pmatrix}.
    \]
\end{lemma}
\begin{proof}
That $\gamma^*$ is represented in the dual basis by the transpose is a general property of dual bases. The second claim follows from a routine calculation involving the change-of-basis matrix from $\beta^*$ to $\beta'$, keeping in mind that for a symplectic matrix $M$, the formula $M^{-1} = J^{-1} M^T J$ (where $J$ represents the symplectic form) holds, and that conjugating by $J$ also passes from a basis to its symplectic dual on $H_1(X;\Z)$.
\end{proof}

\begin{definition}\label{definition:qstar}
In the coordinates on $H_1(X\setminus \cB;\Z)$ given by the geometric splitting as in \eqref{equation:geomsplitexc}, the form $q^*: H_1(X\setminus \cB; \Z) \to \Z/2\Z$ is given by
\[
q^*(x,b) = q(x) + \pair{b, \kappa},
\]
where $(q,\kappa)$ is the {\em associated data} of the geometric splitting, cf. \Cref{lemma:discussion}. Note that the expression $q(x)$ exploits the canonical identification of the subspaces $H_1(X;\Z)$ of $H_1(X, \cB;\Z)$ and $H_1(X \setminus \cB;\Z)$. 
\end{definition}

Note that by Lemmas \ref{lemma:induced}, \ref{lemma:qformula}, when $c \subset X \setminus \cB$ is a simple closed curve with homology class $[c] = (x,b)$ relative to the geometric splitting, there is an equality
\begin{equation}\label{eqn:qphi}
q^*(x,b) = \phi(c) + 1 \pmod 2,
\end{equation}
showing coordinate-independence of $q^*$. In particular, since $\phi(c) = 0$ for $c$ the core curve of a cylinder, it follows that $q^*(x,b) = 1$ for $(x,b)$ the homology class of any cylinder.

\begin{lemma}\label{lemma:qifftheta}
Let $\gamma \in \PAut(H_1(X, \cB; \Z))$ be given. Relative to a geometric splitting $X^\circ$ of $X$ with associated data $(q,\kappa)$, the condition $\Theta_{q,\kappa}(\gamma) = 0$ holds if and only if $\gamma^* \in \PAut(H_1(X\setminus \cB;\Z))$ satisfies 
\[
q^*(\gamma^*(x,b))  = q^*(x,b)
\]
for all $(x,b) \in H_1(X\setminus \cB; \Z)$.
\end{lemma}
\begin{proof}
Recall from \Cref{definition:theta} that if $\gamma = \MAOI$ in the coordinates given by a geometric splitting, then for all $x \in H_1(X; \Z/2\Z)$,
\[
\Theta_{q,\kappa}(\gamma)(x) = q(Mx) + q(x) + \pair{Mx, A\kappa}.
\]
Recalling that $q^*$ is valued in $\Z/2\Z$, we compute the quantity $q^*(\gamma^*(x,b)) + q^*(x,b)$ using Lemma \ref{lemma:matrixformulas}:
\begin{align*}
q^*(\gamma^*(x,b))+q^*(x,b) & = q^*\left(\MAOI^*(x,b) \right) + q^*(x,b)\\
&= q^* \left(M^{-1}x, A^Tx + b \right) + q^*(x,b)\\
&= q(M^{-1}x) + \pair{A^Tx + b,\kappa} + q(x) + \pair{b,\kappa}\\
&= q(M^{-1}x) + q(x) + \pair{A^Tx, \kappa}\\
&= q(M^{-1}x) + q(x) + \pair{x, A\kappa}\\
&= \Theta_{q,\kappa}(\gamma)(M^{-1}x).
\end{align*}
This calculation shows that $\gamma \in \ker(\Theta_{q,\kappa})$ if and only if $\gamma^*$ preserves $q^*$, as desired.
\end{proof}

\begin{proof}[Proof of \Cref{prop:nonhypcyltrans}]
Let $c$ and $c'$ denote the core curves of cylinders $C$ and $C'$ on $X\setminus\cB$, respectively. Choose a geometric splitting $X^\circ$ with associated data $(q,\kappa)$ and relative to this, write $[c] = (c_0, c_1)$ and $[c'] = (c_0', c_1')$ in the induced coordinates on $H_1(X \setminus \cB; \Z)$. Since $c, c'$ are core curves of cylinders, each one is {\em nonseparating} on $X \setminus \cB$, so that the elements $c_0$ and $c_0'$ of $H_1(X^\circ; \Z) \leqslant H_1(X \setminus \cB; \Z)$ are both primitive. By the discussion following \Cref{definition:qstar} above, $q^*(c_0,c_1) = q^*(c_0', c_1') = 1$, since $\phi(C) = \phi(C') = 0$.

The proof will now follow from a slightly more general result, where we only assume the $q^*$ values are {\em equal}, not necessarily both $1$.

\para{Claim} {\em Let $(a,b), (a', b') \in H_1(X\setminus \cB; \Z)$ be given such that $a$ and $a'$ are both primitive and such that $q^*(a,b) = q^*(a',b')$. Then there exists $\gamma \in \ker(\Theta_{q,\kappa})$ such that $(a',b') = \gamma^* (a,b)$.}\\ 

\noindent {\em Proof of Claim:} We seek to construct $\gamma = \MAOI$ such that $a' = M^{-1}a$ and $b'-b = A^Ta$ that additionally satisfies the constraint
\begin{equation}\label{eqn:kernelphi}
q(M^{-1} x) + q(x) + \pair{x, A\kappa} = 0
\end{equation}
for all $x \in H_1(X;\Z)$. 

Assume first that $\kappa = 0$. In this case, the condition $q^*(a,b) = q^*(a',b')$ simplifies to $q(a) = q(a')$. Since $a$ and $a'$ are assumed to be primitive, by \Cref{prop:symporbits}.\ref{item:spq}, there exists $M \in \Sp(2g,\Z)[q]$ such that $M^{-1}a = a'$. Choosing any $A$ such that $A^Ta = b'-b$, it follows that $\gamma = \MAOI$ satisfies  $\gamma^* (a,b) = (a', b')$ and that \eqref{eqn:kernelphi} holds. 

Assume then that $\kappa \ne 0$. Choose any $M \in \Sp(2g,\Z)$ such that $M^{-1}a = a'$. Note that the function $x \mapsto q(M^{-1} x) + q(x)$ defines an element $\bar \eta$ of $H^1(X; \Z/2\Z)$; define $\eta \in H^1(X;\Z)$ as an integral lift. The matrix $A$ must then satisfy the conditions
\begin{equation}\label{aconditions}
A^Ta = b'-b \qquad \mbox{and} \qquad A \kappa = \eta.
\end{equation}
If such an $A$ exists, then 
\[
\pair{b'-b, \kappa} = \pair{A^Ta, \kappa} = \pair{a, A\kappa} = \pair{a, \eta}.
\]
Conversely, an argument in elementary linear algebra shows that if $\pair{b'-b,\kappa} = \pair{a, \eta}$, then there is an $A$ satisfying \eqref{aconditions}.
Recalling the definition of $\eta$,
\[
\pair{a, \eta} = q(M^{-1}a) + q(a),
\]
so that (recalling that we are working mod $2$),
\[
\pair{b'-b,\kappa} + \pair{a, \eta} = q(M^{-1}a) + \pair{b', \kappa} + q(a) + \pair{b,\kappa} = q^*(a',b') + q^*(a,b) = 0
\]
by hypothesis. 
\end{proof}

\subsection{Hyperelliptic}
\begin{proof}[Proof of \Cref{prop:hypcyltrans}]
In the hyperelliptic setting, we will deduce this from the stronger claim that the topological monodromy $\Gamma \leqslant \Mod(X, \cB)$ acts transitively on the set of isotopy classes of cores of cylinders. To see this, we write the distinguished set $\cB$ as $\cB = \cZ \sqcup \cP$ as before, with $\cZ$ the set of zeros and $\cP$ a set of $k \ge 0$ ordinary marked points. Recall from \Cref{lemma:HyperellipticTopMonodromy} and  \Cref{lemma:forgetpoints} that there is a short exact sequence
\begin{equation}\label{equation:hypgammaSES}
1 \to PB_k(X \setminus \cB) \to \Gamma \to \SMod(X, \cZ) \to 1,
\end{equation}
where $PB_k(X \setminus \cB)$ is the surface braid group consisting of point pushes of the set $\cP$ on the surface $X\setminus \cB$.

Recall from \Cref{lemma:HyperellipticTopMonodromy} that $\SMod(X, \cZ)$ is a $\Z/2\Z$ extension of a braid group on $\CP^1$, where the marked points correspond to the singularities of the associated quadratic differential. This acts {\em transitively} on the set of isotopy classes of simple arcs on $X/\pair{\iota}$, both of whose endpoints lie in the set of  distinguished points. We will show that under the projection $X \mapsto X / \pair{\iota}$, the core of any cylinder gives a double branched cover of such an arc. 

According to \cite[Lemma 2.1]{lindsay}, every maximal cylinder $C \subset X$ is $\iota$-invariant: $\iota(C) = C$. As $\iota$ acts geometrically on the translation surface $X$ by rotation by $\pi$, it follows that the core curve $c$ of $C$ is left invariant by $\iota$ and gives a double branched cover of a simple arc on $X/ \pair{\iota}$ connecting two Weierstrass points. Thus any two such core curves $c, c'$ are related by an element of $\SMod(X, \cZ)$. To complete the argument, we must consider the isotopy class of a core curve of a cylinder on the fully marked surface $X \setminus \cB$. But from the description of $\Gamma$ given in \eqref{equation:hypgammaSES}, it is clear that $\Gamma$ acts transitively on the set of isotopy classes of curves $\{\tilde c\}$ on $X \setminus \cB$ that are isotopic to a fixed curve $c \subset X \setminus \cZ$ after forgetting the ordinary points $\cP$. We conclude that $\Gamma$ acts transitively on the set of core curves $c$ of cylinders on $X$ as desired.
\end{proof}

\bibliography{fullbibliotex}{}

\end{document}